\documentclass[11pt,reqno]{amsart}

\usepackage{booktabs} 
\usepackage{array} 
\usepackage{paralist} 
\usepackage{verbatim} 
\usepackage{amssymb}
\usepackage{amsthm}
\usepackage{amsmath,amsfonts,amssymb}
\usepackage{esint}
\usepackage{graphics}
\usepackage{enumerate}
\usepackage{mathtools}
\usepackage{xfrac}
\usepackage{nicefrac}
\usepackage{subcaption}
\usepackage{stmaryrd}
\usepackage[normalem]{ulem}
\usepackage{cancel}
\usepackage{enumitem}
\usepackage{fancyvrb}
\usepackage{bm}
\usepackage{dsfont} 
\usepackage{bbm} 

\allowdisplaybreaks[4]

\usepackage[usenames,dvipsnames]{xcolor}
\usepackage[colorlinks=true, pdfstartview=FitV, linkcolor=blue, citecolor=blue, urlcolor=blue]{hyperref}
\usepackage[normalem]{ulem}

\usepackage{tikz}
\usetikzlibrary{calc}
\usepackage{pgf}
\usetikzlibrary{external}
\numberwithin{equation}{section}
\numberwithin{figure}{section}

\newtheorem{theorem}{Theorem}[section]

\newtheorem{lemma}[theorem]{Lemma}

\theoremstyle{definition}
\newtheorem{definition}[theorem]{Definition}

\newtheorem{remark}[theorem]{Remark}

\usepackage{dsfont}
\usepackage{bbm}
\newcommand{\N}{\mathbb{N}}
\newcommand{\R}{\mathbb{R}}

\newcommand{\cB}{\mathcal{B}}

\newcommand{\cA}{\mathcal{A}}

\newcommand{\cE}{\mathcal{E}}

\newcommand{\cM}{\mathcal{M}}

\newcommand{\cG}{\mathcal{G}}

\newcommand{\cF}{\mathcal{F}}

\newcommand{\cR}{\mathcal{R}}

\newcommand\sfE{{\boldsymbol{\mathsf E}}}

\newcommand{\eps}{\varepsilon}

\newcommand{\data}{\textnormal{\tt data}}
\renewcommand{\rho}{\varrho}

\DeclareMathOperator{\supp}{supp}

\DeclareMathOperator{\tail}{Tail}

\DeclareMathOperator{\exc}{\sfE}
\DeclareMathOperator{\RH}{\normalfont \textsf{RH}}
\DeclareMathOperator{\TAIL}{\normalfont \textsf{Tail}}

\renewcommand{\d}{\mathrm{\,d}}

\renewcommand{\leq}{\leqslant}
\renewcommand{\geq}{\geqslant}
\renewcommand{\subset}{\subseteq}

\DeclareMathAlphabet{\mathmybb}{U}{bbold}{m}{n}
\newcommand{\indc}{{\mathmybb{1}}}

\def\Xint#1{\mathchoice
{\XXint\displaystyle\textstyle{#1}}%
{\XXint\textstyle\scriptstyle{#1}}%
{\XXint\scriptstyle\scriptscriptstyle{#1}}%
{\XXint\scriptscriptstyle\scriptscriptstyle{#1}}%
\!\int}
\def\XXint#1#2#3{{\setbox0=\hbox{$#1{#2#3}{\int}$}
\vcenter{\hbox{$#2#3$}}\kern-.5\wd0}}
\def\dashint{\Xint-}

\def\XXiint#1#2#3{\setbox0=\hbox{$#1{#2#3}{\iint}$}
    \vcenter{\hbox{$#2#3$}}\kern-0.5\wd0}

\begin{document}
\allowdisplaybreaks
\title[Gradient self-improvement for 
mixed local and nonlocal parabolic equations]{Gradient self-improvement for\\ 
mixed local and nonlocal parabolic equations}

\author{Kenta Nakamura}
\address[Kenta Nakamura]{Institute of Natural Sciences, Nihon University, Tokyo, Japan}
\email{kentanak55@gmail.com}


\begin{abstract}
We introduce a new method for studying gradient higher integrability for mixed local and nonlocal parabolic equations. More precisely, for $p>2d/(d+2)$ and $s \in (0,1)$, we prove that if the inhomogeneity $F \in L^{p(1+\sigma)}_{\mathrm{loc}}$ for some $\sigma>0$, then every weak solution satisfies $\nabla u \in L^{p(1+\eps)}_{\mathrm{loc}}$ for some $\eps>0$, together with quantitative local estimates. The proof relies on stopping time arguments and an intrinsic Calder\'{o}n-Zygmund-type covering decomposition in which the inhomogeneity and nonlocal energy contributions are treated by classical and fractional maximal estimates. 
\end{abstract}

\maketitle

\setcounter{tocdepth}{1}
\tableofcontents

\section{Introduction}

\subsection{Informal summary of results}

Our aim in this paper is to establish the gradient higher integrability of solutions to mixed local and nonlocal parabolic equations of the form
\begin{equation}\label{maineq}
\partial_tu-\nabla \cdot \mathbf{a} \left(x,t,u,\nabla u\right)\,+\mathcal{L}_{s,p}u=-\nabla \cdot \left(|F|^{p-2}F\right) \quad \mbox{in}\quad \Omega_T\,,
\end{equation}
where  $\Omega$ is a bounded open subset of $\R^d$ for $d \geq 2$ and $\Omega_T:=\Omega \times (0,T)$ for any $T \in (0,\infty)$ and $\nabla u=(\partial_{x_i}u)_{1 \leq i \leq d}$ denotes the spatial gradient of $u$, whereas the divergence of a vector field $\mathbf{f}$ is $\nabla \cdot \mathbf{f}:=\sum_{i=1}^d \partial_{x_i} f_i$ with $(f_i)$ being the entries of $\mathbf{f}$. The vector-valued function $\mathbf{a}: \Omega_T \times \R \times \R^d \to \R^d$ is assumed to satisfy~(A1) and (A2), while $\mathcal{L}_{s,p}$ is an integro-differential operator defined by
\begin{align}\label{Frac. p-Lap.}
\mathcal{L}_{s,p}u(x,t)&:=2\,\mathrm{p.v.}\int_{\R^d} \mathds{K}(x,y,t)|u(x,t)-u(y,t)|^{p-2}(u(x,t)-u(y,t))\d{y} \notag\\[1mm]
&\,=2\lim_{\eps \to 0} \int_{\R^d \setminus B_\eps(x)} \mathds{K}(x,y,t)|u(x,t)-u(y,t)|^{p-2}(u(x,t)-u(y,t))\d{y}\,,
\end{align}
where the kernel $\mathds{K}=\mathds{K}(x,y,t) : \R^d \times \R^d \times (0,T) \to [0,\infty)$ is a measurable function fulfilling the symmetry and uniform ellipticity conditions described in~(A3) and (A4). In addition, the $\R^d$-valued inhomogeneity $F$ is assumed to be locally $L^p$-integrable.

Higher integrability of the spatial gradient is one of the basic and fundamental self-improving properties in linear and nonlinear regularity theory. In the purely local setting, it is obtained from a reverse H\"{o}lder inequality combined with the Gehring lemma. Indeed, Gehring’s lemma was originally developed to study the higher integrability of the Jacobian of a quasiconformal mapping~\cite{Geh73}; see also~\cite{Str80}. In the elliptic setting, Elcrat and Meyers~\cite{ME75} proved in particular that the self-improving property holds. In the parabolic case, Giaquinta and Struwe~\cite{GiSt82} established the reverse H\"{o}lder inequalities and the higher integrability for weak solutions for the quasilinear parabolic systems. In the breakthrough work ~\cite{KiLe00}, Kinnunen and Lewis established the gradient higher integrability for weak solutions of general parabolic systems with $p$-growth. Later, Acerbi and Mingione~\cite{AM07} showed the gradient higher integrability result for nonhomogeneous, degenerate or singular parabolic $p$-Laplacian systems. Turning to the porous medium equation, Gianazza and Schwarzacher~\cite{GS19} proved the gradient higher integrability for nonnegative solutions. Recently, the gradient higher integrability for weak solutions of the doubly nonlinear parabolic equations or systems, including the Trudinger equation, was obtained in~\cite{BDKS20, BDS22}. Besides, the gradient higher integrability for weak solutions of the parabolic double phase problem was proved in~\cite{KKM23}. This approach has been developed in a wide range of elliptic and parabolic problems; see for instance~\cite{Par09, DeF20, HO21} and the references therein. On the other hand, regularity theory for nonlinear nonlocal equations has developed rapidly over the last two decades, beginning with the elliptic theory and extending to parabolic problems; we refer for instance to~\cite{Kas09, CV10, CCV11, FK13, DKP14, DKP16, Coz17, BLS18, DeFP19, Str19, Now23a, Now23b, CKW23, KW24, BDLMBS25, Zha26} and the references therein. In particular, Kuusi, Mingione and Sire~\cite{KMS15} established nonlocal self-improvement properties for solutions of a nonlocal equation with measure coefficients. Their argument  relies on a fractional version of Gehring's lemma for fractional Sobolev functions. As far as we are aware, its parabolic version has remained open for quite some years. Recently, Diening and Nowak~\cite{DN25} established the Calder\'{o}n-Zygmund estimate for the fractional $p$-Laplace equation with possibly discontinuous coefficients of VMO-type by means of pointwise estimates involving certain fractional sharp maximal functions.

Mixed local and nonlocal elliptic and parabolic equations lie at the interface of these two theories. Operators of this type arise naturally from superpositions of stochastic processes at different scales, such as classical random walks and L\'{e}vy flight; see~\cite{BDVV21, BDVV22} and the references therein. In the elliptic case, from the perspective of calculus of variations, most of the regularity estimates from local boundedness to gradient H\"{o}lder regularity were established in the seminal work by De Filippis and Mingione~\cite{DM24}; see also~\cite{GK22, BKL24}. For the parabolic case, local boundedness, Harnack estimates and related qualitative properties have been studied in this framework; see for instance~\cite{GK23, GK24, Nak22a, Nak22b}. In the linear parabolic case, Das~\cite{Das24} established H\"{o}lder continuity of the spatial gradient by means of a perturbative approach. The question considered here is of a different nature in the mixed local and nonlocal regime. What can we say about the self-improvement property for the spatial gradient of solutions, that is, $\nabla u\in L^p_{\mathrm{loc}}$ implies $\nabla u\in L^{p(1+\varepsilon)}_{\mathrm{loc}}$ for $\eps>0$, under the minimal structural assumptions above?  In fact, to the best of our knowledge, the higher integrability result for weak solutions of the mixed local and nonlocal parabolic equations under the general structural assumptions has remained open for quite some years. One of the main difficulties is that the nonlocal contribution cannot simply be treated as a ``lower-order term'' in the usual Gehring lemma. Indeed, after deriving a reverse H\"{o}lder inequality, one is left with a nonlocal energy term which exists at every stopping scale. Its dependence on the radius is different from that of the inhomogeneity, and it is not controlled by the classical parabolic maximal operator. Thus, even though the local part determines the intrinsic geometry, the nonlocal interaction still remains in the self-improvement argument in an essential way.

Nevertheless, our main result (see Theorem~\ref{t.gradhigher} below) shows that this difficulty can be overcome. Namely, it asserts roughly that, if $p>2d/(d+2)$ and $F\in L^{p(1+\sigma)}_{\mathrm{loc}}(\Omega_T;\R^d)$ for some $\sigma>0$, then every weak solution of~\eqref{maineq} satisfies $\nabla u\in L^{p(1+\eps)}_{\mathrm{loc}}(\Omega_T;\R^d)$ for some $\eps>0$ depending only on $\data$ and $\sigma$.

\subsection{Hypotheses and statement of the main result}
Before stating the main result, we give the minimal notation needed for the statements below and the precise assumptions. Throughout the paper, we work in dimension $d \geq 2$ and the Euclidean norm on $\R^d$ is denoted by $|\cdot |$. %
We denote the Lebesgue measure of a set $U \subset \R^d$ by $|U|$. Moreover, we denote volume-normalized integrals by
\[
(f)_U:=\dashint_{U}f\d{x}:=\frac{1}{|U|}\int_{U}f\d{x}\,.
\]
We require that the vector field $\mathbf{a}:\Omega_T \times \R \times \R^d \to \R^d$ satisfies the following assumptions.
\begin{enumerate}
\item[(A1)]
$\mathbf{a}$ is a Carath\`{e}odory vector-valued function, namely, it is  measurable with respect to $(x,t) \in \Omega_T$ for every $(u,\xi)  \in \R \times \R^d$ and continuous with respect to $(u,\xi)$ for almost every $(x,t) \in \Omega_T$.

\smallskip

\item[(A2)]
There exist constants $0<\nu \leq L<\infty$ such that $\mathbf{a}$ satisfies the $p$-growth and coercivity conditions
\begin{equation*}
\begin{cases}
\mathbf{a}(x,t,u,\xi) \cdot \xi \geq \nu |\xi|^p\,, \\
\left|\mathbf{a}(x,t,u,\xi)\right|\leq L |\xi|^{p-1}
\end{cases}
\end{equation*}
for a.e. $(x,t) \in \Omega_T$ and any $(u,\xi) \in \R \times \R^d$.
\end{enumerate}

We further impose that the integro-differential operator $\mathcal{L}_{s,p}$ defined by~\eqref{Frac. p-Lap.} satisfies the following assumptions:

\begin{enumerate}
\item[(A3)]
The associated kernel $\mathds{K}: \R^d \times \R^d \times (0,T) \to [0,\infty)$ is a measurable function satisfying, for every $(x,y,t) \in \R^d \times \R^d \times (0,T)$, the symmetry property
\begin{equation*}
\mathds{K}(x,y,t)=\mathds{K}(y,x,t)\,.
\end{equation*}

\smallskip

\item[(A4)]
$\mathds{K}$ satisfies the uniform ellipticity condition, for almost every $(x,y,t) \in \R^d \times \R^d \times (0,T)$, 
\begin{equation}
\frac{\Lambda^{-1}}{|x-y|^{d+sp}}\leq \mathds{K}(x,y,t) \leq \frac{\Lambda}{|x-y|^{d+sp}}
\end{equation}
with a fixed ellipticity constant $\Lambda  \in [1, \infty)$ and a fractional order $s \in (0,1)$. 
\end{enumerate}

We are now in a position to present the statement of the main result.

\begin{theorem}[Gradient higher integrability]
\label{t.gradhigher}
Let $p>\frac{2d}{d+2}$, $ s \in (0,1)$ and fix $\theta:=\frac{(1-s)p}{p-1}$. Then, there exists $\eps_0\in (0,1)$ depending only on $\data := (d,s,p,\nu, L, \Lambda) $ such that, whenever $F \in L^{p(1+\sigma)}_{\mathrm{loc}}\left(\Omega_T\,;\R^d\right)$ for some $\sigma>0$ and $u$ is a weak solution to~\eqref{maineq} in the sense of Definition~\ref{dfn u} under the assumptions~(A1)--(A4), then
\[
\nabla u \in L_{\mathrm{loc}}^{p(1+\eps_1)} (\Omega_T \,; \R^d )\,,
\]
where
\[
\eps_1:=
\left\{
\begin{aligned}
&\min \bigl\{\eps_0, \nicefrac{\sigma}{2},\,\nicefrac{\theta}{p(d+2-\theta)}\bigr\}  &\mbox{if} \quad & \theta < d+2\,, \\
&\min \bigl\{\eps_0, \,\nicefrac{\sigma}{2} \bigr\}  &\mbox{if} \quad & \theta \geq  d+2\,.
\end{aligned}
\right.
\]
More precisely, there also exist $C \in [1,\infty)$ and $R_0 \in (0,1)$, both depending only on $\data$ and $\sigma$, such that for every $\eps \in (0,\eps_1]$ and every cylinder $Q_R \subset \Omega_T$  with $R \in (0,R_0]$, if $\theta < d+2$ then
\begin{align}\label{e.gradhigher1}
\dashint_{Q_{R/2}} |\nabla u|^{p(1+\eps)} \d{x}\d{t} &\leq C\left(\dashint_{Q_R}\big[|\nabla u|^p+(|F|^p+1)^{1+\sigma}\big]\d{x}\d{t} \right)^{1+\nicefrac{\eps p}{\gamma_p}} \notag\\
&\quad  \quad \quad +C\left(\dashint_{Q_R}\cE[u]\d{x}\d{t} \right)^{\frac{d+2}{d+2-\theta}}\,,
\end{align}
whereas, in the case $\theta \geq d+2$, we have
\begin{align}\label{e.gradhigher2}
\dashint_{Q_{R/2}} |\nabla u|^{p(1+\eps)} \d{x}\d{t} &\leq C \delta^{-\nicefrac{\eps p}{\gamma_p} }\left(\dashint_{Q_R}\big[|\nabla u|^p+(|F|^p+1)^{1+\sigma}\big]\d{x}\d{t} \right)^{1+\nicefrac{\eps p}{\gamma_p}}\notag \\
&\quad  \quad \quad +C\delta \left(\dashint_{Q_R}\cE[u]\d{x}\d{t} \right)^{\frac{2(\eps p+\gamma_p)}{4+(2-p)(\theta-d-2)}}\,,
\end{align}
whenever $\delta \in (0,1]$. Here, $\cE[u]$ and $\gamma_p$ are the nonlocal energy and the scaling deficit, respectively:
\[
\cE[u](x,t):=\int_{\R^d} \frac{\big|u(y,t)-u(x,t)\big|^p}{|y-x|^{d+sp}}\d{y}
\quad \mbox{and} \quad
\gamma_p:=
\left\{
\begin{aligned}
&\quad \,\,\,2  &\mbox{if} \quad &p \geq 2\,, \\
&\tfrac{p(d+2)-2d}{2}  &\mbox{if} \quad &\tfrac{2d}{d+2}<p <2\,.
\end{aligned}
\right.
\]
\end{theorem}
\begin{remark}
Observe that
\[
\theta=\frac{(1-s)p}{p-1} \geq d+2 \quad \mbox{and} \quad p >\frac{2d}{d+2} \quad \Longrightarrow \quad d=2,3\,,
\]
so, this case can occur only in dimensions $d=2,3$. Also, as we will explain later in Section~\ref{Sect.5.3},~\eqref{e.gradhigher2} is only valid for $\frac{2d}{d+2}<p <2$.
\end{remark}

\subsection{An overview of the approach}
Let us describe an overview of the approach. The first step is to derive the gluing lemma which controls the oscillation in time of spatial averages of $u$ on a ball. Besides the usual local quantities involving $\nabla u$ and $F$, the resulting estimate contains the nonlocal energy
\[
\mathcal E[u](x,t):=\int_{\R^d} \frac{|u(y,t)-u(x,t)|^p}{|x-y|^{d+sp}}\d{y}\,.
\]
The radius in the gluing lemma is selected independently of the two instants. This allows us to derive a family of parabolic Sobolev-Poincar\'{e} inequalities adapted to the intrinsic geometry of the problem (Section~\ref{Sect.3}). Note that most of these estimates are actually valid for the full range $p>1$. The restriction $p> 2d/ (d+2)$ arises only when the $L^2$-excess has to be controlled through the parabolic Sobolev-Poincar\'{e} inequalities, see Lemma~\ref{SP3}. This threshold arises from the parabolic Sobolev embedding (or Gagliardo-Nirenberg inequality) rather than by the nonlocal term itself. It is well-known that combining the Caccioppoli inequality with the Sobolev-Poincar\'{e} inequality yields a reverse H\"{o}lder inequality for $\nabla u$. Unlike its purely local counterpart, however, the Caccioppoli inequality still contains a nonlocal tail. We therefore prove a ``$p$-tail''-controlled inequality (Lemma~\ref{t.p-tail-controlled inequality}), which reduces this nonlocal tail to the intrinsic energy and a scale-dependent nonlinear energy. 
\smallskip

The main argument for self-improvement properties is carried out by a careful inspection of local and nonlocal effects. Controlling the nonlocal quantity is a pivotal ingredient and cannot be accomplished by the stopping time argument alone. Instead, we employ an intrinsic Calder\'{o}n-Zygmund-type covering decomposition. The underlying Calder\'{o}n-Zygmund strategy goes back to Caffarelli and Peral in their seminal work~\cite{CP98}. Later, Mingione~\cite{Min07} devised a new framework to implement this strategy in the nonlinear elliptic setting. See also Armstrong and Daniel~\cite{AD16} in the homogenization setting and the references therein. Given a super-level set of $|\nabla u|$, we associate with each point a maximal intrinsic cylinder and split the resulting Vitali family into three classes: $\cG$, $\cB_{\tail}$ and $\cB_{\mathrm{inhom}}$. On the good cylinders $\cG$, both the inhomogeneity and the nonlocal energy are small compared with the intrinsic scaling. The reverse H\"{o}lder inequality then provides the usual gain from the lower exponent $\kappa p$. The bad inhomogeneity cylinders $\cB_{\mathrm{inhom}}$ are controlled by the classical parabolic maximal operator. The technically new part is condensed in the bad tail cylinders $\cB_{\tail}$. When $\theta<d+2$, the nonlocal quantity
\[
r^\theta\dashint_{Q_r^\lambda}\cE[u](x,t)\d{x}\d{t}
\]
has exactly the scaling of a fractional maximal operator of order $\theta$ in the parabolic dimension $d+2$. We use its weak-type estimate (Lemma~\ref{t.weak L1-estimate}) to control the ``bad tail'' family $\cB_{\tail}$. The exponent $q:=\frac{d+2}{d+2-\theta}$ then naturally appears in the quantitative gradient estimate.
\smallskip

On the other hand, when $\theta\geq d+2$, a fractional maximal estimate of the preceding type is no longer available. Instead, the intrinsic scaling itself becomes sufficiently strong: after a suitable choice of the initial level $\lambda_0$, the nonlocal energy is automatically below the stopping threshold on every selected cylinder, eventually leading to $\cB_{\tail}=\varnothing$. Thus, the threshold $\theta=d+2$ essentially  changes the scheme of the proof, that is, when $\theta <d+2$ the nonlocal term is controlled through fractional maximal estimates, while for $\theta \geq d+2$ the bad-tail alternative is ruled out directly by careful inspection using the intrinsic scaling. So, once the super-level estimate is obtained, the remaining argument is classical and fundamental. Indeed, integrating the estimate with respect to the level parameter, using Fubini’s theorem combined with Cavalieri's principle and rearranging, now yields the desired gradient higher integrability.
\smallskip

We emphasize that our argument has to handle the local and nonlocal terms simultaneously, so the resulting decomposition separates the two relevant technical complications in the classical self-improvement scheme and treats them using maximal operators of different orders. To the best of our knowledge, this viewpoint is new and we expect it to be useful in other mixed or nonlocal parabolic problems.

\subsection{Intrinsic geometry}
We now briefly discuss an informal motivation for the intrinsic geometry used throughout the paper. The so-called ``intrinsic geometry'' was originally developed by DiBenedetto and Friedman~\cite{DF85} in the context of the parabolic $p$-Laplace equation; see also~\cite{DiB93}. 

Let $Q_{\rho,\vartheta}:=B_\rho \times (-\vartheta, \vartheta) \subset \R^d \times \R$. For this, consider the model case:
\begin{equation}\label{e.model}
\partial_tu-\nabla \cdot \left(|\nabla u|^{p-2}\nabla u\right)+(-\Delta)_p^su=-\nabla \cdot \left(|F|^{p-2}F\right) \quad \mbox{in} \quad Q_{\rho,\vartheta}\,.
\end{equation}
For a solution $u$ of~\eqref{e.model}, consider the scaled solution $v(y,\tau):=u(\rho y, \vartheta \tau )$ for $(y,\tau) \in B_1 \times (-1,1)$. Likewise, set $G(y,\tau):=\rho F(\rho y, \vartheta \tau)$. Infinitesimally, we can regard $|\nabla u|$ and $|F|$ as
\[
|\nabla u| \simeq \left(\dashint_{Q_{\rho,\vartheta}}|\nabla u|^p\d{x}\d{t}\right)^{\frac{1}{p}}\quad \mbox{and} \quad |F| \simeq \left(\dashint_{Q_{\rho,\vartheta}}|F|^p\d{x}\d{t}\right)^{\frac{1}{p}}\,,
\]
where the symbol $\simeq$ is used only at the level of infinitesimal scales. At this infinitesimal level, we are formally led to 
\begin{align*}
&\nabla \cdot \left(|\nabla u|^{p-2}\nabla u\right) \simeq \left(\dashint_{Q_{\rho,\vartheta}}|\nabla u|^p\d{x}\d{t}\right)^{\frac{p-2}{p}} \Delta u, \quad \mbox{and} \quad \\
& (-\Delta)^s_p u \simeq \left(\dashint_{Q_{\rho,\vartheta}}|\nabla u|^p\d{x}\d{t}\right)^{\frac{p-2}{p}} \rho^{(1-s)(p-2)} (-\Delta)^su\,.
\end{align*}
Thus, at this formal level, we obtain that
\begin{align*}
\partial_\tau v&=\vartheta \cdot \partial_t u =\vartheta \Bigl[\nabla \cdot \left(|\nabla u|^{p-2}\nabla u\right) -  (-\Delta)^s_p u - \nabla \cdot \left(|F|^{p-2}F\right)\Bigr]\\
&\simeq \vartheta \left[ \left(\dashint_{Q_{\rho,\vartheta}}|\nabla u|^p\d{x}\d{t}\right)^{\frac{p-2}{p}}\left(\frac{\Delta_y v}{\rho^2}-\rho^{(1-s)(p-2)}\cdot \frac{(-\Delta_y)^sv}{\rho^{2s}}\right) \right.\\
&\left. \quad \quad \quad \quad \quad \quad - \left(\dashint_{Q_{\rho,\vartheta}}|F|^p\d{x}\d{t}\right)^{\frac{p-2}{p}} \frac{ \nabla_y \cdot G}{\rho^2} \right]\\
&\simeq \frac{\vartheta}{\rho^2}\left(\dashint_{Q_{\rho, \vartheta}}\bigl(|\nabla u|^p+|F|^p \bigr)\d{x}\d{t}\right)^{\frac{p-2}{p}}\Bigl[\Delta_y v-\rho^{(1-s)p}(-\Delta_y)^sv-\nabla_y \cdot G \Bigr]\,.
\end{align*}
To make the equation linear in its own geometry, we select $\vartheta$ so that
\begin{equation*}
\frac{\vartheta}{\rho^2}\left(\dashint_{Q_{\rho, \vartheta}}\bigl(|\nabla u|^p+|F|^p \bigr)\d{x}\d{t}\right)^{\frac{p-2}{p}} =1 
\quad \iff \quad \vartheta= \underbrace{ \left(\dashint_{Q_{\rho,\vartheta}}\bigl(|\nabla u|^p+|F|^p \bigr)\d{x}\d{t}\right)^{\frac{2-p}{p}}}_{\simeq \lambda^{2-p}}\rho^2\,.
\end{equation*}
In this way, infinitesimally, the nonlinear equation becomes a linear equation in its own geometry, while the local part dictates the intrinsic geometry. Therefore, we adopt our intrinsic geometry as
\[
Q_{\rho}^\lambda:=B_\rho \times \left(-\lambda^{2-p}\rho^2, \lambda^{2-p} \rho^2 \right)\quad \mbox{with}  \quad \lambda \simeq \left(\dashint_{Q_{\rho}^\lambda}\bigl(|\nabla u|^p+|F|^p \bigr)\d{x}\d{t}\right)^{\nicefrac{1}{p}}\,.
\]
For the details, see Definition~\ref{t.intrinsic} below.
\smallskip

\subsection{Outline of the paper}
In Section~\ref{Sect.2}, we summarize general notation, functional spaces and preliminary results used in the paper. Section~\ref{Sect.3} is devoted to the parabolic Sobolev-Poincar\'{e} inequalities; in particular, Section~\ref{Sect.3.1} contains the key estimate used in Section~\ref{Sect.3.2}.
In Section~\ref{Sect.4}, we prove that the spatial gradient of a weak solution satisfies a reverse
H\"{o}lder inequality. The proof of Theorem~\ref{t.gradhigher} is given in Section~\ref{Sect.5}, which is the analytic core of the paper. For the convenience of the reader, we split this section into four subsections: In Section~\ref{Sect.5.1} we set up preliminary materials. Section~\ref{Sect.5.2} is devoted to the stopping time argument, whose scope is the local term. In Section~\ref{Sect.5.3}, in order to control the nonlocal effect, we demonstrate an intrinsic Calder\'{o}n-Zygmund-type covering decomposition using the classical and fractional maximal operators. Finally, Section~\ref{Sect.5.4} gives the final step of the proof of Theorem~\ref{t.gradhigher}.

\section{Preliminaries}\label{Sect.2}
\subsection{Notation}
Throughout the paper, let $p>1$ and $s \in (0,1)$ unless otherwise stated. Fix $q_\ast:=\max\{1,p-1\}$. 
Given $a, b\in \R$, we write  $a \wedge b:=\min\{a,b\}$ and $a\vee b:=\max\{a,b\}$. We also denote $a_+=a \vee 0$ and $a_-:=-(a \wedge 0)$. For the indicator function of a set $A$, we write $\indc_A$. The symbols $C$ and $c$ denote positive constants which may vary from line to line. We specify the dependency of the constants by writing. To lighten the notation, we denote
\[
\data := (d,s,p,\nu, L, \Lambda)\,.
\]
This shorthand allows us to denote constants $C$ which depend on $(d,s,p,\nu, L, \Lambda)$ by simply $C(\data)$ instead of $C(d,s,p,\nu, L, \Lambda)$.
\smallskip

Let $\Omega \subset \R^d$ denote an open, bounded domain with Lipschitz boundary. For $T \in (0,\infty)$, $\Omega_T:=\Omega \times (0,T)$ denotes a space-time cylinder. With center $x_0 \in \R^d$ and the radius $\rho>0$, we denote
\[
B_\rho(x_0):=\bigl\{x \in \R^d : |x-x_0|<\rho\bigr\}\,.
\]
We omit denoting the center, that is, abbreviating $B_\rho \equiv B_\rho(x_0)$, if it is not necessary or clear from the context. Moreover, we denote by 
\[
Q_{\rho}^{\lambda}(z_0):=B_\rho(x_0) \times \left(t_0-\lambda^{2-p} \rho^{2}, t_0+\lambda^{2-p} \rho^{2}\right)\,.
\]
the usual intrinsic cylinder with radius $\rho>0$, time width $2\lambda^{2-p}\rho^2>0$ and center $z_0=(x_0,t_0) \in\R^d \times \R$, where $\lambda>0$ is a scaling parameter to be determined in the arguments. Similarly as before, we omit denoting the center if it is not necessary, namely, $Q_\rho^{\lambda} \equiv Q_\rho^{\lambda}(z_0)$. In particular, we denote $Q_\rho \equiv Q_\rho^1$ for short.

We often use the normalized version of $L^p$ norms: for every $p \in [1,\infty)$ and $f \in L^p(U)$, we set
\[
\|f\|_{\underline{L}^p(U)}:=\left(\dashint_U |f|^p\d{x}\right)^{\nicefrac{1}{p}}=|U|^{-\nicefrac{1}{p}}\|f\|_{L^p(U)}\,.
\]
For a vector-valued $F \in L^p(U; \R^d)$, we write $\|F\|_{\underline{L}^p(U)}:=\||F|\|_{\underline{L}^p(U)}$.

\subsection{Functional spaces}\label{Function spaces}
In this subsection, we present various functional spaces used in the paper. The fractional Sobolev space $W^{s,p}(\R^d)$ is defined via
\[
W^{s,p}(\R^d):=\Big\{w \in L^p(\R^d) : [w]_{W^{s,p}(\R^d)}<+\infty\Big\}\,,
\]
where
\[
\displaystyle [w]_{W^{s,p}(\R^d)}:=\left(\iint_{\R^d \times \R^d} \frac{|w(x)-w(y)|^p}{|x-y|^{d+sp}}\d{x}\d{y}\right)^{\nicefrac{1}{p}}
\]
is the Gagliardo-Slobodecki\u{i} seminorm. $W^{s,p}(\R^d)$ is a Banach space endowed with the norm
\[
\|w\|_{W^{s,p}(\R^d)}:=\|w\|_{L^p(\R^d)}+[w]_{W^{s,p}(\R^d)}\,.
\]
In a similar fashion, the fractional Sobolev spaces $W^{s,p}(\Omega)$ in a domain $\Omega \subset \R^d$ can be defined. The fractional Sobolev space with zero boundary values is defined as
\[
W_0^{s,p}(\Omega):=\Big\{w \in W^{s,p}(\R^d): w=0\,\,\textrm{on}\,\,\R^d \setminus \Omega \Big\}\,.
\]
The fundamental and useful results in the fractional Sobolev spaces are presented in the comprehensive monographs~\cite{DNPV12, FeRo24}.

As discovered in~\cite{DKP14, DKP16}, the \emph{nonlocal tail} which captures the long-range interactions caused by the nonlocal problem is defined by
\[
\tail\left(u\,;x_0,\rho\right):=\left(\rho^{sp}\int_{\R^d \setminus B_\rho(x_0)} \frac{|u(x)|^{p-1}}{|x-x_0|^{d+sp}}\d{x}\right)^{\nicefrac{1}{(p-1)}};
\]
when $x_0=0$ or it is obvious from the context, we write $\tail\left(u\,; \rho\right) =\tail\left(u\,;x_0,\rho\right)$. Similarly, we define ``nonlocal $p$-tail''  by
\[
\tail_p\left(u\,;x_0,\rho\right):=\left(\rho^{sp}\int_{\R^d \setminus B_\rho(x_0)} \frac{|u(x)|^{p}}{|x-x_0|^{d+sp}}\d{x}\right)^{\nicefrac{1}{p}};
\]
the associated weighted Lebesgue space is given by
\begin{equation}\label{e.tail space}
L_{sp}^{p}(\R^d):=\left\{ u\in L_{\mathrm{loc}}^{p}(\R^d): \int_{\R^d}\frac{|u(x)|^{p}}{(1+|x|)^{d+sp}}\d{x} <\infty\right\}.
\end{equation}
A straightforward computation assures that
\[
L_{sp}^{p}(\R^d):=\left\{ u\in L_{\mathrm{loc}}^{p}(\R^d): \tail_p\left(u\,;x_0,\rho\right)<\infty, \quad \forall x_0 \in \R^d, \quad \rho>0\right\}\,.
\]

Concerning the parabolic setting, we refer to the Bochner integral:  given $p \in [1,\infty)$, $I \subset \R$ and an arbitrary Banach space $X$, we denote by $L^p(I ; X)$ the space of Lebesgue-measurable mappings $u : I \to X$ such that
\[
\|u\|_{L^p(I ; X)}:=\left(\int_I \|u(t)\|_X^p\d{t} \right)^{\nicefrac{1}{p}}<\infty\,.
\]
Finally, $C(I; X)$ denotes the space of continuous maps $t \mapsto u(t) \in X$ with respect to the norm on $X$, that is,
\[
\|u(t)-u(t_0) \|_X \to 0 \quad \mbox{as} \quad t \to t_0, \quad \forall t_0 \in I\,.
\]
\subsection{Some elementary tools}
Throughout the paper, we repeatedly use the following formula:
\begin{equation}\label{formula}
\int_{\R^d \setminus B_\rho(x_0)}\frac{\d{x}}{|x-x_0|^{d+sp}}=\frac{d \omega_d}{sp}\rho^{-sp\,},
\end{equation}
where $\omega_d$ denotes the volume of the unit ball in $\R^d$.

We recall the Gagliardo-Nirenberg inequality.
\begin{lemma}[Gagliardo-Nirenberg inequality]\label{GN}
Let  $1 \leq \chi_1,\chi_2,\chi_3 <\infty$ and $\vartheta \in (0,1)$ satisfy
\[
-\frac{d}{\chi_1} \leq \vartheta \left(1-\frac{d}{\chi_2}\right)-(1-\vartheta)\frac{d}{\chi_3}\,.
\]
Then, for every function $w \in W^{1,\chi_2}(B_\rho)$ there exists a constant $C(d, \chi_1,\chi_2,\chi_3,\vartheta)<\infty$ so that
\[
\dashint_{B_\rho}\left|\frac{w}{\rho}\right|^{\chi_1}\d{x} \leq C\left(\dashint_{B_\rho}\left(\,\left|\frac{w}{\rho}\right|^{\chi_2}+|\nabla w|^{\chi_2}\right)\d{x}\right)^{\frac{\theta \chi_1}{\chi_2}}\left(\dashint_{B_\rho}\left|\frac{w}{\rho}\right|^{\chi_3}\d{x}\right)^{\frac{(1-\theta) \chi_1}{\chi_3}}\,.
\]
\end{lemma}

We need the following useful lemma. The proof is similar to that of~\cite[Lemma 3.5]{BDKS20}.
\begin{lemma}\label{t.useful lemma}
Let $p \geq 1$ and  $V \subset U \subset \R^d$ be measurable sets such that $0<|V| \leq |U|<\infty$. For every $v \in L^p(U)$, 
\[
\dashint_U\left|v(x)-(v)_V\right|^p\d{x} \leq 2^p\frac{|U|}{|V|}\,\dashint_U\left|v(x)-(v)_U\right|^p\d{x}\,.
\]
\end{lemma}
\begin{proof}
By H\"{o}lder's inequality and the elementary inequality
\[
(a+b)^{p} \leq 2^{p-1}(a^p+b^p) \quad \mbox{for} \quad a, b \in \R_{\geq 0}\,, 
\]
we observe that
\begin{align*}
\dashint_U\left|v(x)-(v)_V\right|^p \d{x}&\leq 2^{p-1}\dashint_U\dashint_V\left|v(x)-(v)_U\right|^p \d{y}\d{x} \\
&\quad \quad +2^{p-1}\dashint_U\dashint_{V}\left|(v)_U-v(y)\right|^p\d{y}\d{x}\,,
\end{align*}
and thus rearranging gives the desired result.
\end{proof}
We finally record the classical iteration lemma. The proof is done by~\cite[Lemma 6.1]{Giusti} with a minor adaptation.
\begin{lemma}\label{t.iteration}
Let $A$, $B$, $\alpha$ be nonnegative constants. Suppose that $h:[R_0, R_1] \to [0, \infty)$ satisfies
\[
h(t) \leq \frac{1}{2}h(s)+\frac{A}{(s-t)^\alpha}+B
\]
for every $R_0 \leq t<s<R_1$. Then there exists a constant $C(\alpha)<\infty$ such that
\[
h(R_0) \leq C\left(\frac{A}{(R_1-R_0)^\alpha}+B\right)\,.
\]
\end{lemma}
%

\subsection{Mollifiers}
We remark that the time derivative of weak solutions is not integrable, which is defined only in the weak sense, and therefore we require the~\emph{exponential mollifications in time}, which are defined for $w\in L^1(\Omega_T,\R)$ and $h>0$,
\[
\llbracket w\rrbracket_h(x,t):=\displaystyle \frac{1}{h} \int_0^t e^{\frac{\vartheta-t}h}w(x,\vartheta)\d{\vartheta} \quad \mbox{and} \quad \llbracket w\rrbracket_{\bar{h}}(x,t):=\displaystyle \frac{1}{h} \int_t^T e^{\frac{t-\vartheta}h}w(x,\vartheta)\d{\vartheta}\,.
\]
These were originally designed for the doubly nonlinear equations. Nevertheless, instead of the usual Steklov average, we adopt them. The following relevant properties are valid; see~\cite[Lemma 2.2]{KL06} and~\cite[Appendix B]{BDM13} for details.

\begin{lemma}\label{t.mollifier}
Fix $q \geq 1$ arbitrarily. The following statements are true:
\begin{enumerate}
\item[(i)] If $w \in L^q(\Omega_T)$, then $ \llbracket w\rrbracket_h \in L^q(\Omega_T)$. Moreover, 
$$\left\| \llbracket w\rrbracket_h \right\|_{L^q(\Omega_T)} \leq \left\|w \right\|_{L^q(\Omega_T)}$$ and $\llbracket w\rrbracket_h  \to w$ strongly in $L^q(\Omega_T)$ and almost everywhere in $\Omega_T$ as $h \searrow  0$. The same statements hold for $\llbracket w\rrbracket_{\bar{h}}$ as well.\\
\item[(ii)] $\llbracket w\rrbracket_h \in C\left([0,T]; L^q(\Omega)\right)$ and $\llbracket w\rrbracket_{\bar{h}} \in C\left([0,T]; L^q(\Omega)\right)$.\\
\item[(iii)] The differential formula holds almost everywhere in $\Omega_T$:
\[
\partial_t \llbracket w\rrbracket_h =\frac{1}{h} \left(w-\llbracket w\rrbracket_h \right) \quad \mbox{and} \quad \partial_t \llbracket w\rrbracket_{\bar{h}} =\frac{1}{h} \left(\llbracket w\rrbracket_{\bar{h}}-w \right)\,.
\]
\item[(iv)] If $w \in C\left([0,T]; L^q(\Omega)\right)$, then $\llbracket w(\cdot, t)\rrbracket_h \to w(\cdot, t)$ strongly in $L^q(\Omega)$ and almost everywhere in $\Omega$ for every $t \in [0,T]$ as $h \searrow  0$. The same statements hold for $\llbracket w\rrbracket_{\bar{h}}$ as well.
\end{enumerate}
\end{lemma}

\subsection{Weak formulation}
We adopt the notion of weak solutions to~\eqref{maineq} in the following.
\begin{definition}[Weak solution]\label{dfn u}
Let $p>1$, $s \in (0,1)$ and assume that (A1)--(A4) hold and that $F \in L^p_{\mathrm{loc}}\left(\Omega_T\,;\R^d\right)$. We identify a function
\begin{equation}\label{e.class}
u  \in  C\left([0,T]\,; L^2(\Omega) \right) \cap L^p\left(0,T\,; W^{1,p}(\Omega)\right) \cap L^p\left(0,T\,; L^{p}_{sp}(\R^d)\right)
\end{equation}
as a \emph{weak solution} of~\eqref{maineq} if and only if the energy identity
\begin{align}\label{weak form}
\int_{\Omega_T}\Big(-u \,\partial_t\varphi &+\mathbf{a}(x,t,u,\nabla u)\cdot \nabla \varphi\Big)\d{x}\d{t}+\int_0^T \mathcal{B}_K(u(t),\varphi(t))\d{t} \notag\\
&=\int_{\Omega_T}|F|^{p-2}F\cdot \nabla \varphi\d{x}\d{t}
\end{align}
holds whenever $\varphi \in C^\infty_0(\Omega_T)$, where the associated weak energy form $\cB_{K}$ is given by
\begin{align*}
\mathcal{B}_K(u(t),\varphi(t)):=\iint_{\R^d \times \R^d} \mathds{K}(x,y,t)&|u(x,t)-u(y,t)|^{p-2}\left(u(x,t)-u(y,t)\right)\\
&\times \left(\varphi(x,t)-\varphi(y,t)\right)\d{x}\d{y}\,.
\end{align*}

\end{definition}

We finally list the mollified type of weak formulation~\eqref{weak form}. The proof is done by means of Fubini's theorem for the double integral and a careful reading of~\cite[Lemma 2.10]{Nak22a} in the doubly nonlinear framework.

\begin{lemma}\label{dfn u-h}\normalfont
Let $p>1$, $s \in (0,1)$ and assume that~(A1)--(A4) and $F \in L^p_{\mathrm{loc}}\left(\Omega_T\,;\R^d\right)$. Let $u$ be a weak solution to~\eqref{maineq} in the sense of Definition~\ref{dfn u}. Then,
\begin{align}\label{weak form-h}
\int_{\Omega_T}\Big(\partial_t \llbracket u \rrbracket _h\varphi &+\llbracket \mathbf{a}(x,t,u,\nabla u)\rrbracket_h\cdot \nabla \varphi\Big)\d{x}\d{t} +\int_0^T\mathcal{B}_{K, h}(u(t),\varphi(t))\d{t} \notag\\
&=\int_{\Omega_T}\llbracket |F|^{p-2}F\rrbracket _h \cdot \nabla \varphi\d{x}\d{t}
\end{align}
for every $\varphi \in C^\infty_0(\Omega_T)$, where
\begin{align*}
\mathcal{B}_{K, h}(u(t),\varphi(t)):=\iint_{\R^d \times \R^d} \llbracket \mathds{K}(x,y,t)&|u(x,t)-u(y,t)|^{p-2}\left(u(x,t)-u(y,t)\right) \rrbracket_h\\
&\times \left(\varphi(x,t)-\varphi(y,t)\right)\d{x}\d{y}\,.
\end{align*}
\end{lemma}
\subsection{Caccioppoli inequality}

We conclude this section by stating the Caccioppoli inequality. 

%
%
\begin{lemma}[Caccioppoli type estimate]\label{t.caccioppoli}
Let $p >1$, $s \in (0,1)$ and let $u$ be a weak solution to~\eqref{maineq} in the sense of Definition~\ref{dfn u} under the assumptions~(A1)--(A4). Suppose further that $F \in L^p_{\mathrm{loc}}\left(\Omega_T\,;\R^d\right)$. There exists a constant $C$ depending only on $\data$ such that 
\begin{align}\label{e.caccioppoli}
&\sup_{t \in I^\lambda_{\rho}(t_0)}\dashint_{B_{\rho}(x_0)}\frac{\left|u(t)-k\right|^2}{\lambda^{2-p}\rho^2}\d{x}+\dashint_{Q_{\rho}^\lambda(z_0)}|\nabla (u-k)|^p\d{x}\d{t} \notag\\[2mm]
&\quad \quad +\dashint_{I_{\rho}^\lambda(t_0)}\int_{B_{\rho}(x_0)}\dashint_{B_{\rho}(x_0)}\frac{|\left(u(x,t)-k\right)-\left(u(y,t)-k\right)|^p}{|x-y|^{d+sp}}\d{x}\d{y}\d{t} \notag\\[2mm]
&\leq  C \dashint_{Q_R^\lambda(z_0)}\left[\frac{\left|u-k\right|^p}{(R-\rho)^p}+\frac{\left|u-k\right|^2}{\lambda^{2-p}\left(R^2-\rho^2\right)}+|F|^p \right]\d{x}\d{t} \notag\\[2mm]
&\quad \quad \quad +\frac{C}{(R-\rho)^p}\dashint_{I_{R}^\lambda(t_0)}\int_{B_{R}(x_0)}\dashint_{B_{R}(x_0)}\frac{\left|u(x,t)-k\right|^p}{|x-y|^{d-(1-s)p}}\d{x}\d{y}\d{t} \notag\\[2mm]
&\quad \quad \quad +C\dashint_{Q_{R}^\lambda(z_0)}\left(\sup_{x\in B_{(\rho+R)/2}(x_0)}\int_{\R^d \setminus B_{R}(x_0)}\frac{\left|u(y,t)-k\right|^{p-1}}{|x-y|^{d+sp}}\d{y}\right)\left|u-k\right|\d{x}\d{t}
\end{align}
whenever concentric cylinders $Q_\rho^\lambda (z_0) \subset Q_R^\lambda(z_0)$ and level $k \in \R$.
\end{lemma}
\begin{proof}
The starting point is Lemma~\ref{dfn u-h}. Using the property of the mollifier introduced in Lemma~\ref{t.mollifier}, the proof is almost a verbatim repetition of the proof of~\cite[Proposition 3.1]{Nak22a}. We omit the details.
\end{proof}

\section{Parabolic Sobolev-Poincar\'{e} inequalities}\label{Sect.3}
In this section, we establish certain parabolic Sobolev-Poincar\'{e} inequalities. Generally, weak solutions of the parabolic equations are not necessarily differentiable with respect to time, and thus one cannot directly apply a Sobolev-Poincar\'{e} inequality in $\R^{d+1}$. Hence our task is to establish a new type of Sobolev-Poincar\'{e} inequality by means of the Gluing lemma, taking into account the lack of differentiability with respect to time.

\subsection{Gluing lemma}\label{Sect.3.1}
We start by establishing the Gluing lemma, which gives an error estimate for time differences of $(u(t))_{B_{\bar{\rho}}}$ for some $\bar{\rho}$.

\begin{lemma}[Gluing lemma]\label{t.gluing}
Fix $p >1$ and $s \in (0,1)$ and let $\theta:=\frac{(1-s)p}{p-1}$. Let $u$ be a weak solution to~\eqref{maineq} in the sense of Definition~\ref{dfn u} under the assumptions~(A1)--(A4).  Further, suppose that $F \in L^p_{\mathrm{loc}}\left(\Omega_T\,;\R^d\right)$. Then, for every space-time cylinder $Q_\rho^{\lambda}(z_0) \Subset \Omega_T$ and $t_1,t_2 \in I_\rho^\lambda(t_0) \subset (0,T)$ with $t_1<t_2$, there exists a radius $\bar{\rho} \in \left[\frac{\rho}{2}, \frac{3\rho}{4}\right]$ such that, for a constant $C(\data)<\infty$,
\begin{align*}
\Big|(u(t_2))_{B_{\bar{\rho}}(x_0)}-(u(t_1))_{B_{\bar{\rho}}(x_0)}\Big| &\leq C\lambda^{2-p}\rho \,\dashint_{Q_\rho^{\lambda}(z_0)}\bigl(|\nabla u|^{p-1}+|F|^{p-1}\bigr)\d{x}\d{t}\\
&\quad  \quad \quad +C\lambda^{2-p}\rho \left(\rho^\theta \dashint_{Q_\rho^\lambda(z_0)}\cE[u]\d{x}\d{t}\right)^{\nicefrac{(p-1)}{p}}\,.
\end{align*}
\end{lemma}

\begin{proof}
Without loss of generality, we may assume that $z_0=(x_0,t_0)=(0,0)$. 
Let $t_1,t_2 \in I_\rho^{\lambda}$ be such that $t_1<t_2$ and take $r \in [\rho/2,3\rho/4]=:J$. For sufficiently small $\eps>0$, we introduce $\zeta_\eps \in C_c^1(0,T)$ to be a Lipschitz approximation of $\indc_{\{t_1 \leq t \leq t_2\}}$, that is, $\zeta_\eps=1$ on $(t_1,t_2)$, $\zeta_\eps=0$ on $(-\lambda^{2-p}\rho^2, t_1-\eps] \cup [t_2+\eps, \lambda^{2-p}\rho^2)$, while linearly interpolated otherwise.
Furthermore, for any $\delta \in (0,\rho/8)$ and $r \in (0,\rho)$, we define the following radial function $\psi_{\delta, r}$ by
\[
\psi_{\delta, r}(x):=
\left\{
\begin{aligned}
& 1 &\mbox{if}& \quad 0\leq |x| \leq r\,, \\
& 1+\tfrac{1}{\delta}(r-|x|) &\mbox{if} & \quad r < |x| \leq r+\delta, \quad \mbox{and} 
\\ 
& 0  &\mbox{if} &\quad r+\delta \leq |x| \leq \rho\,.
\end{aligned}
\right.
\]
Testing~\eqref{weak form} with $\varphi=\psi_{\delta, r}(x)\zeta_\eps(t)$, we get
\begin{align*}
-\dashint_{t_1-\eps}^{t_1}&\int_{\R^d}u(t)\psi_{\delta,r}(x)\d{x}\d{t}+\dashint_{t_2}^{t_2+\eps}\int_{\R^d}u(t)\psi_{\delta, r}(x)\d{x}\d{t} \\
&+\int_{I_\rho^{\lambda}} \cB \left(u(t), \psi_{\delta, r}\zeta_\eps(t)\right)\d{t}\\
&=\int_{I_\rho^{\lambda}}\left(\int_{B_{r+\delta}}\Bigl[-\mathbf{a}(x,t,u,\nabla u)+|F|^{p-2}F\Bigr]\cdot \nabla \psi_{\delta,r}(x)\d{x}\right)\zeta_\eps(t)\d{t}\,.
\end{align*}
%
Sending $\eps \to 0$ in the above display, we find that
\begin{align*}
-\int_{\R^d}u(t_1)&\psi_{\delta,r}(x)\d{x}\d{t}+\int_{\R^d}u(t_2)\psi_{\delta, r}(x)\d{x}\d{t} \\
&+\int_{t_1}^{t_2} \cB \left(u(t), \psi_{\delta, r}(\cdot)\right)\d{t}\\
&=\int_{t_1}^{t_2}\left(\int_{B_{\rho}}\Bigl[-\mathbf{a}(x,t,u,\nabla u)+|F|^{p-2}F\Bigr]\cdot \nabla \psi_{\delta,r}(x)\d{x}\right)\d{t}\,,
\end{align*}
and therefore (A2) yields
\begin{align}\label{e.gruing1}
\frac{1}{|Q_\rho^\lambda|}&\left|\int_{\R^d}\bigl(u(t_2)-u(t_1)\bigr)\psi_{\delta, r}\left(x\right)\d{x}\right| \notag\\
&\leq \frac{1}{|Q_\rho^\lambda|}\int_{I_\rho^{\lambda}}\left(\int_{B_{\rho}}\Bigl[L|\nabla u|^{p-1}+|F|^{p-1}\Bigr]\cdot \bigl|\nabla \psi_{\delta,r}(x)\bigr|\d{x}\right)\d{t} \notag\\
&\quad \quad +\frac{1}{|Q_\rho^\lambda|}\int_{I_\rho^{\lambda}}\left|\cB \bigl(u(t), \psi_{\delta, r}\bigr)\right|\d{t} \notag\\
&=:\mathbf{I}_\delta(r)+\mathbf{II}_\delta(r)\,.
\end{align}
Using Fubini's theorem, we estimate that
\begin{align*}
\int_J\mathbf{I}_\delta(r)\d{r} &\leq c(L) \dashint_{Q_\rho^\lambda} \bigl(|\nabla u|^{p-1}+|F|^{p-1}\bigr)\underbrace{\left(\int_J \bigl|\nabla \psi_{\delta,r}(x)\bigr|\d{r}\right)}_{\leq 1}\d{x}\d{t} \\
&\leq c(L) \dashint_{Q_\rho^\lambda} \bigl(|\nabla u|^{p-1}+|F|^{p-1}\bigr)\d{x}\d{t}\,.
\end{align*}
Next, observe that, for every $x, y \in \R^d$,
\begin{equation}\label{e.gruing2}
\int_J\left|\psi_{\delta, r}(x)-\psi_{\delta, r}(y)\right|\d{r} \leq \left(\rho \wedge |x-y|\right).
\end{equation}
Indeed, by $0\leq \psi_{\delta, r} \leq 1$, it readily follows that
\[
\int_J\left|\psi_{\delta, r}(x)-\psi_{\delta, r}(y)\right|\d{r} \leq 2|J|<\rho\,.
\]
On the other hand, the fundamental theorem of calculus yields
\begin{align*}
\left|\psi_{\delta, r}(x)-\psi_{\delta, r}(y)\right|
&=\left|\int_0^1 (x-y)\cdot \nabla \phi_{\delta, r}(\tau x+(1-\tau)y)\d{\tau}\right|\\
&\leq |x-y| \int_0^1\left|\nabla \phi_{\delta, r}(\tau x+(1-\tau)y)\right| \d{\tau}\,,
\end{align*}
and therefore, by Fubini's theorem we find that
\begin{align*}
\int_J\left|\psi_{\delta, r}(x)-\psi_{\delta, r}(y)\right|\d{r} &\leq |x-y| \int_0^1 \left(\int_J \left|\nabla \phi_{\delta, r}(\tau x+(1-\tau)y)\right| \d{r}\right)\d{\tau} \leq |x-y|\,.
\end{align*}
Combining these yields~\eqref{e.gruing2}. 
Note that, since $\supp \psi_{\delta, r} \subset B_\rho$, by the symmetry of $K$ as in~(A3), we have that
\begin{equation*}
\left|\cB \bigl(u(t), \psi_{\delta, r}\bigr) \right| \leq  2\Lambda \int_{B_\rho}\int_{\R^d} \frac{\bigl|u(x,t)-u(y,t)\bigr|^{p-1}\bigl|\psi_{\delta,r}(x)-\psi_{\delta,r}(y)\bigr|}{|x-y|^{d+sp}}\d{y}\d{x}\,.
\end{equation*}
Thus, appealing to H\"{o}lder's inequality with $\left(\frac{p}{p-1}, p\right)$, we obtain that, using~\eqref{e.gruing2}
\begin{align*}
\int_J\mathbf{II}_\delta(r)\d{r} &\leq C\left(\dashint_{Q_\rho^\lambda}\cE[u]\d{x}\d{t}\right)^{\nicefrac{(p-1)}{p}}\left(\dashint_{Q_\rho^\lambda}\int_{\R^d} \frac{\left(\rho \wedge |x-y|\right)^p}{|x-y|^{d+sp}}\d{y}\d{x}\d{t}\right)^{\nicefrac{1}{p}}\\
&=C \left(\rho^\theta \dashint_{Q_\rho^\lambda}\cE[u]\d{x}\d{t}\right)^{\nicefrac{(p-1)}{p}}
\end{align*}
for a constant $C(d,s,p,\Lambda)<\infty$. Thus, by the pigeonhole principle, there exists $r_\delta \in J$ such that 
\begin{equation*}
\mathbf{I}_\delta(r_\delta)+\mathbf{II}_\delta(r_\delta) \leq \frac{2}{|J|}\int_J \bigl(\mathbf{I}_\delta(r)+\mathbf{II}_\delta(r)\bigr)\d{r}\,.
\end{equation*}
Note that $r_\delta \in J$ is independent of the choice of $t_1$ and $t_2$. Combining these with~\eqref{e.gruing1} yields, for a constant $C(\data)<\infty$, 
\begin{align}\label{e.gruing3}
\frac{1}{|Q_\rho^\lambda|}&\left|\int_{\R^d}\bigl(u(t_2)-u(t_1)\bigr)\psi_{\delta, r_\delta}\left(x\right)\d{x}\right| \notag\\
&\leq \frac{C}{|J|}\dashint_{Q_\rho^\lambda} \bigl(|\nabla u|^{p-1}+|F|^{p-1}\bigr)\d{x}\d{t}+\frac{C}{|J|}\left(\rho^\theta \dashint_{Q_\rho^\lambda}\cE[u]\d{x}\d{t}\right)^{\nicefrac{(p-1)}{p}}\,.
\end{align}
Select $\delta_j \downarrow 0$ as $j \to \infty$. Since $r_j:=r_{\delta_j} \in J$, passing to a not relabelled subsequence, there exists $\bar{\rho} \in J$ such that $r_j \to \bar{\rho}$ as $j \to \infty$ and thus, for almost every $x, y \in \R^d$, we find that
\[
\psi_{\delta_j,r_j} \to \indc_{B_{\bar{\rho}}}\,.
\]
Since $ 0 \leq \psi_{\delta_j,r_j} \leq 1$, all these functions are supported in $B_\rho$ and $u(t) \in L^2(B_\rho)$, in view of~\eqref{e.class}, we have, for each $t \in I_\rho^\lambda$, 
\[
\int_{\R^d}u(t)\psi_{\delta_j,r_j}(x)\d{x} \rightarrow \int_{B_{\bar{\rho}}}u(t)\d{x}
\]
in the limit $j \to \infty$. Hence, appealing to~\eqref{e.gruing3} with $\delta=\delta_j$ and $r=r_j$, and passing to the limit $j \to \infty$ in the resulting display, we obtain that, for every $t_1 <t_2$ in $I_\rho^\lambda$, 
\begin{align*}
\frac{|B_{\bar{\rho}}|}{|Q_\rho^\lambda|}&\left|\dashint_{B_{\bar{\rho}}}\bigl(u(t_2)-u(t_1)\bigr)\d{x}\right| \\
&\leq \frac{4C}{\rho}\dashint_{Q_\rho^\lambda}\bigl(|\nabla u|^{p-1}+|F|^{p-1}\bigr)\d{x}\d{t}+\frac{4C}{\rho}\left(\rho^\theta \dashint_{Q_\rho^\lambda}\cE[u]\d{x}\d{t}\right)^{\nicefrac{(p-1)}{p}}\,.
\end{align*}
Finally, rearranging yields the desired conclusion.
\end{proof}
As a consequence of this lemma, we deduce the following:
\begin{lemma}\label{t.gluing application}
Under the assumption of Lemma~\ref{t.gluing}, let $\kappa$ be an exponent so that $q_\ast /p \leq \kappa \leq 1$. Let $\bar{\rho} \in \bigl[\nicefrac{\rho}{2},\nicefrac{3\rho}{4}\bigr]$ be as in Lemma~\ref{t.gluing}. There exists $C(\data)<\infty$ such that
\begin{align}\label{e.gluing application ineq}
\begin{split}
&\dashint_{I_\rho^\lambda(t_0)}\dashint_{I_\rho^\lambda(t_0)}\Big|(u(t_2))_{B_{\bar \rho}(x_0)}-(u(t_1))_{B_{\bar \rho}(x_0)}\Big|^{\kappa p}\d{t}_1\d{t}_2 \\
&\quad \leq C (\lambda^{2-p}\rho)^{\kappa p}\left[\left(\dashint_{Q_\rho^\lambda(z_0)}\big[|\nabla u|^{\kappa p}+|F|^{\kappa p}\big]\d{x}\d{t}\right)^{p-1}+\left(\rho^\theta \dashint_{Q_\rho^\lambda}\cE[u]\d{x}\d{t}\right)^{\kappa (p-1)}\right]\,.
\end{split}
\end{align}
\end{lemma}
\begin{proof}
Note that $\kappa p \geq p-1$. By applying the H\"{o}lder inequality and Lemma~\ref{t.gluing}, we find that,
\begin{align*}
&\dashint_{I_\rho^\lambda(t_0)}\dashint_{I_\rho^\lambda(t_0)}\Big|(u(t_2))_{B_{\bar \rho}(x_0)}-(u(t_1))_{B_{\bar \rho}(x_0)}\Big|^{\kappa p}\d{t}_1\d{t}_2 \\[2mm]
&\leq C\left(\lambda^{2-p}\rho\right)^{\kappa p} \left[\left(\dashint_{Q_\rho^{\lambda}(z_0)}\big[|\nabla u|^{p-1}+|F|^{p-1}\big]\d{x}\d{t}\right)^{\kappa p} +\left(\rho^\theta \dashint_{Q_\rho^\lambda}\cE[u]\d{x}\d{t}\right)^{\kappa (p-1)} \right]\\[2mm]
&\leq C\left(\lambda^{2-p}\rho\right)^{\kappa p} \left[\left(\dashint_{Q_\rho^{\lambda}(z_0)}\big[|\nabla u|^{\kappa p}+|F|^{\kappa p}\big]\d{x}\d{t}\right)^{p-1}+\left(\rho^\theta \dashint_{Q_\rho^\lambda}\cE[u]\d{x}\d{t}\right)^{\kappa (p-1)}\right]\,.
\end{align*}
The proof of Lemma~\ref{t.gluing application} is now complete.
\end{proof}

\subsection{Parabolic Sobolev-Poincar\'{e} inequalities}\label{Sect.3.2}
We henceforth use the \emph{excess functional} in an intrinsic cylinder $Q_\rho^{\lambda}(z_0)$ defined by
\[
\exc_q\left(u, Q_\rho^{\lambda}(z_0)\right):=\left(\dashint_{Q_\rho^{\lambda}(z_0)}\bigl|u-(u)_{Q_\rho^{\lambda}(z_0)}\bigr|^q\d{x}\d{t}\right)^{\nicefrac{1}{q}}\,.
\]

First, we deduce a basic parabolic Poincar\'{e} type inequality. See also the recent and interesting paper on self-improving properties for parabolic Poincar\'{e} inequality~\cite{KK26}. 
\begin{lemma}[Poincar\'{e} inequality]\label{SP1}
Under the assumption of Lemma~\ref{t.gluing}, let $\kappa$ be an exponent so that $q_\ast/p \kappa \leq 1$. There exists a constant $C(\data)<\infty$ such that, for every space-time cylinder $Q_\rho^\lambda (z_0) \Subset \Omega_T$ with $0<\rho \leq 1$ and $\lambda>0$ being a given scaling parameter, we have
\begin{align}\label{SP1display}
\left[\frac{\exc_{\kappa p}\left(u, Q_\rho^\lambda(z_0)\right)}{\rho}\right]^{\kappa p}&\leq C\lambda^{(2-p)\kappa p}\left[\left(\dashint_{Q_\rho^{\lambda}(z_0)}\bigl(|\nabla u|^{\kappa p}+|F|^{\kappa p}\bigr)\d{x}\d{t}\right)^{p-1} \right.\notag\\
&\quad \quad \quad \quad \quad \quad \quad \quad \quad \left.+\left(\rho^\theta \dashint_{Q_\rho^\lambda}\cE[u]\d{x}\d{t}\right)^{\kappa (p-1)}\right]\,.\end{align}
\end{lemma}
\begin{proof}
We may let $z_0=0$ without loss of generality. We begin with decomposing:
\begin{equation}\label{SP1 eq.1}
\left[\frac{\exc_{\kappa p}\left(u, Q_\rho^\lambda \right)}{\rho}\right]^{\kappa p}  \leq C\bigl(\mathbf{I}+\mathbf{II}+\mathbf{III}\bigr)\,,
\end{equation}
where
\begin{equation*}
\left\{
\begin{aligned}
& \mathbf{I}:=\dashint_{Q_\rho^\lambda}\frac{\left|u-(u(t))_{B_{\bar{\rho}}}\right|^{\kappa p}}{\rho^{\kappa p}}\d{x}\d{t} \,, \\
& \mathbf{II}:=\dashint_{Q_\rho^\lambda}\frac{\left|(u(t))_{B_{\bar{\rho}}}-(u(\tau))_{B_{\bar{\rho}}}\right|^{\kappa p}}{\rho^{\kappa p}}\d{x}\d{t}\,, \quad \mbox{and} 
\\ 
& \mathbf{III}:=\dashint_{Q_\rho^\lambda}\frac{\left|(u(\tau))_{B_{\bar{\rho}}}-(u)_{Q_{\rho}}\right|^{\kappa p}}{\rho^{\kappa p}}\d{x}\d{t}\,.
\end{aligned}
\right.
\end{equation*}
In the definition of $\mathbf{I}$--$\mathbf{III}$, the radius $\bar{\rho} \in [\rho/2, 3\rho/4]$ is determined in Lemma~\ref{t.gluing}.

By applying Lemma~\ref{t.useful lemma} and the Poincar\'{e} inequality, 
\[
\mathbf{I} \leq C\dashint_{Q_\rho^\lambda}|\nabla u|^{\kappa p}\d{x}\d{t}\,.
\]
The H\"{o}lder and Poincar\'{e} inequalities and Lemma~\ref{t.useful lemma} yield that 
\begin{align*}
\mathbf{III}&=\frac{1}{\rho^{\kappa p}}\left|\dashint_{Q_\rho^\lambda}\big[u-(u(\tau))_{B_{\bar{\rho}}}\big]\d{x}\d{\tau} \right|^{\kappa p} \\[2mm]
&\leq \frac{1}{\rho^{\kappa p}}\dashint_{Q_\rho^\lambda}\left|u-(u(\tau))_{B_{\bar{\rho}}}\right| ^{\kappa p}\d{x}\d{\tau} \leq C\dashint_{Q_\rho^\lambda}|\nabla u|^{\kappa p}\d{x}\d{t}\,.
\end{align*}
Finally, by applying H\"{o}lder's inequality and Lemma~\ref{t.gluing application}, we observe that
\begin{align*}
\mathbf{II}& =\frac{1}{\rho^{\kappa p}}\dashint_{I_\rho^\lambda}\left|\dashint_{I_\rho^\lambda}\big[(u(t))_{B_{\bar{\rho}}}-(u(\tau))_{B_{\bar{\rho}}}\big]\d{\tau}\right|^{\kappa p}\d{t} \\
& \leq \frac{1}{\rho^{\kappa p}}\dashint_{I_\rho^\lambda} \dashint_{I_\rho^\lambda}\big|(u(t))_{B_{\bar{\rho}}}-(u(\tau))_{B_{\bar{\rho}}}\big|^{\kappa p}\d{t}\d{\tau}\\
& \leq C \lambda^{(2-p)\kappa p}\left[\left(\dashint_{Q_\rho^{\lambda}(z_0)}\big[|\nabla u|^{\kappa p}+|F|^{\kappa p}\big]\d{x}\d{t}\right)^{p-1}+\left(\rho^\theta \dashint_{Q_\rho^\lambda}\cE[u]\d{x}\d{t}\right)^{\kappa (p-1)}\right]
\end{align*}
for some $\bar{\rho} \in \left[\nicefrac{\rho}{2},\nicefrac{3\rho}{4}\right]$ determined in Lemma~\ref{t.gluing}. Collecting the preceding estimates to~\eqref{SP1 eq.1}, the proof of Lemma~\ref{SP1} is complete.
\end{proof}

Before giving the next two parabolic Sobolev-Poincar\'{e} inequalities, we introduce the notion of ``$\lambda$-intrinsic geometry'' as follows:
\begin{definition}[$\lambda$-intrinsic geometry]\label{t.intrinsic}
Let $u$ be a weak solution to~\eqref{maineq} in the sense of Definition~\ref{dfn u}, under the  assumptions~(A1)--(A4) and suppose that $F \in L^p_{\mathrm{loc}}\left(\Omega_T, \R^d\right)$. A space-time cylinder $Q^\lambda_{\rho}(z_0) \Subset \Omega_T$ with $0 <\rho \leq 1$ satisfies \emph {$\lambda$-intrinsic geometry} if and only if there exists a constant $K\geq 1$ such that
\begin{equation}\label{intrinsic 1}
\frac{1}{K}\dashint_{Q_{\rho}^\lambda(z_0)}\big(|\nabla u|^p+|F|^p\big)\d{x}\d{t} \leq \lambda^p \leq K\dashint_{Q_{\rho/2}^\lambda (z_0)}\big(|\nabla u|^p+|F|^p\big)\d{x}\d{t}\,.
\end{equation}
Notice, by a simple calculation, that ~\eqref{intrinsic 1} implies 
\begin{equation}\label{intrinsic 2}
\lambda^p \leq 2^{d+2}K\dashint_{Q_\rho^\lambda(z_0)}\bigl(|\nabla u|^p+|F|^p\bigr)\d{x}\d{t}\,.
\end{equation}
\end{definition}
\begin{lemma}[Sobolev-Poincar\'{e} inequality I]\label{SP2}
Under the assumption of Lemma~\ref{t.gluing}, we further suppose that $Q_\rho^\lambda (z_0) \Subset \Omega_T$ satisfies $\lambda$-intrinsic geometry condition~\eqref{intrinsic 1}--\eqref{intrinsic 2}. There exists a constant $C(\data,K)<\infty$ such that
\begin{align}\label{SP2display}
\left[\frac{\exc_{p}\left(u, Q_\rho^\lambda(z_0)\right)}{\rho}\right]^{p}&\leq \eps \left(\sup_{t \in I_\rho^\lambda(t_0)}\dashint_{B_\rho(x_0)}\frac{\bigl|u(t)-(u)_{Q_\rho^\lambda(z_0)}\bigr|^2}{\lambda^{2-p} \rho^2}\d{x} +\dashint_{Q_\rho^\lambda(z_0)}|\nabla u|^p\d{x}\d{t}\right) \notag \\[2mm]
&\quad +C\eps^{-\bm{\gamma}_{p}}\left[\left(\dashint_{Q_\rho^\lambda(z_0)}|\nabla u|^{\kappa p}\d{x}\d{t}\right)^{\nicefrac{1}{\kappa}} \right. \notag\\[2mm]
&\quad \quad \quad  \left.+\boldsymbol{\mathfrak{T}}_{p, \lambda}(\rho)\left(\rho^\theta \dashint_{Q_\rho^\lambda(z_0)}\cE[u]\d{x}\d{t}\right)+\dashint_{Q_\rho^\lambda(z_0)}|F|^p\d{x}\d{t}\right]
\end{align}
whenever $\eps \in (0,1]$, where 
\[
\bm{\gamma}_{p}:=
\left\{
\begin{aligned}
&\tfrac{(1-\kappa)p}{2-(1-\kappa)p}  &\mbox{if} \quad & p \geq 2\,, \\
&\tfrac{1-\kappa(p-1)}{\kappa(p-1)}  &\mbox{if} \quad & 1<p<2\,
\end{aligned}
\right.
\,\, \mbox{and} \,\,\,
\bm{\mathfrak{T}}_{p, \lambda}(\rho):=
\left\{
\begin{aligned}
&\frac{1}{\lambda^{p \sigma_0}}\left(\rho^\theta \dashint_{Q_\rho^\lambda(z_0)}\cE[u]\d{x}\d{t}\right)^{\sigma_0}  &\mbox{if} \quad & p \geq 2\,, \\
&1  &\mbox{if} \quad & 1<p<2\,,
\end{aligned}
\right.
\]
and $\sigma_0:=\frac{(p-2)(1+\kappa)}{2-(1-\kappa)p}$ for short. The parameter $\kappa$ is chosen so that $\kappa:=\max\left\{\frac{d}{d+2},\frac{q_\ast}{p}\right\}<1$.
\end{lemma}
\begin{proof} Without loss of generality, we may let $z_0=0$. Let $\kappa \in (0,1)$ be such that $\kappa =\max\left\{\frac{d}{d+2},\frac{q_\ast}{p}\right\}$.
Appealing to the Gagliardo-Nirenberg inequality (Lemma~\ref{GN}) with
\[
(\chi_1,\,\chi_2,\,\chi_3,\,\vartheta)=(p,\,\kappa p,\,2,\,\kappa),
\]
which is admissible since
\[
-\frac{d}{p} \leq \kappa \left(1-\frac{d}{\kappa p}\right)-(1-\kappa)\frac{d}{2} \quad \iff \quad \kappa \geq \frac{d}{d+2},
\]
we have that
\begin{align}\label{SP2 eq.1}
\left[\frac{\exc_p\left(u, Q_\rho^\lambda\right)}{\rho}\right]^{p} &\leq C \dashint_{I_\rho^\lambda}\left(\dashint_{B_\rho}\left[\left|\frac{u-(u)_{Q_\rho^\lambda}}{\rho}\right|^{\kappa p}+|\nabla u|^{\kappa p}\right]\d{x}\right)  \notag\\
&\quad \quad \quad  \quad \quad \quad \times \left(\dashint_{B_\rho}\left|\frac{u-(u)_{Q_\rho^\lambda}}{\rho}\right|^{2}\d{x}\right)^{\nicefrac{(1-\kappa)p}{2}}\d{t} \notag\\[2mm]
&\leq C\left(\sup_{t \in I_\rho^\lambda}\dashint_{B_\rho}\left|\frac{u(t)-(u)_{Q_\rho^\lambda}}{\rho}\right|^{2}\d{x}\right)^{\nicefrac{(1-\kappa)p}{2}}\notag\\
&\quad \quad \quad  \quad \quad \quad \times \left(\dashint_{Q_\rho^\lambda}\left[\left|\frac{u-(u)_{Q_\rho^\lambda}}{\rho}\right|^{\kappa p}+|\nabla u|^{\kappa p}\right]\d{x}\d{t}\right)
\end{align}
with a constant $C(d,p,L)<\infty$.
We focus on the last integral. Applying Lemma~\ref{SP1}, we have
\begin{align*}
\dashint_{Q_\rho^\lambda}&\left|\frac{u-(u)_{Q_\rho^\lambda}}{\rho}\right|^{\kappa p}\d{x}\d{t} \\
&\leq C\lambda^{(2-p)\kappa p}\left[\left(\dashint_{Q_\rho^\lambda}\big[|\nabla u|^{\kappa p}+|F|^{\kappa p} \big]\d{x}\d{t}\right)^{p-1} +\left(\rho^\theta \dashint_{Q_\rho^\lambda}\cE[u]\d{x}\d{t}\right)^{\kappa (p-1)}\right]\,,
\end{align*}
thereby combining this with~\eqref{SP2 eq.1} yields
\begin{align}\label{SP2 eq.2}
\left[\frac{\exc_{p}\left(u, Q_\rho^\lambda\right)}{\rho}\right]^{p} &\leq C\lambda^{\frac{2-p}{2}(1-\kappa)p}\left(\sup_{t \in I_\rho^\lambda}\dashint_{B_\rho}\frac{\bigl|u(t)-(u)_{Q_\rho^\lambda}\bigr|^2}{\lambda^{2-p}\rho^2}\d{x}\right)^{\nicefrac{(1-\kappa)p}{2}} \notag\\[2mm]
&\times \left[\lambda^{(2-p)\kappa p}\left(\dashint_{Q_\rho^\lambda}\big[|\nabla u|^{\kappa p}+|F|^{\kappa p} \big]\d{x}\d{t}\right)^{p-1}  \right. \notag\\
&\quad \quad \quad \quad  \left. +\lambda^{(2-p)\kappa p}\left(\rho^\theta \dashint_{Q_\rho^\lambda}\cE[u]\d{x}\d{t}\right)^{\kappa (p-1)} +\dashint_{Q_\rho^\lambda}|\nabla u|^{\kappa p}\d{x}\d{t}\right]\,.
\end{align}
To lighten the notation we write as follows:

\begin{equation}\label{SP2 eq.3}
\left\{
\begin{aligned}
& \mathbf{S}:=\sup_{t \in I_\rho^\lambda}\dashint_{B_\rho}\frac{\bigl|u(t)-(u)_{Q_\rho^\lambda}\bigr|^2}{\lambda^{2-p}\rho^2}\d{x}\,, \\
& \mathbf{I}:=\dashint_{Q_\rho^\lambda}\big[|\nabla u|^{\kappa p}+|F|^{\kappa p} \big]\d{x}\d{t}\,, \quad \mbox{and} 
\\ 
& \mathbf{T}:=\rho^\theta \dashint_{Q_\rho^\lambda}\cE[u]\d{x}\d{t}\,.
\end{aligned}
\right.
\end{equation}
Thus,~\eqref{SP2 eq.2} implies that
\begin{align}\label{SP2 eq.4}
\left[\frac{\exc_{p}\left(u, Q_\rho^\lambda\right)}{\rho}\right]^{p} \leq C\lambda^{\frac{2-p}{2}(1-\kappa)p}\mathbf{S}^{\frac{(1-\kappa)p}{2}}\mathbf{I}+C\lambda^{\frac{2-p}{2}(1-\kappa)p+(2-p)\kappa p}\mathbf{S}^{\frac{(1-\kappa)p}{2}}\left(\mathbf{I}^{p-1}+\mathbf{T}^{\kappa (p-1)}\right)\,.
\end{align}
We split the estimate of the integral on the right side of~\eqref{SP2 eq.4} into two cases: $p \geq 2$ and $1<p<2$. When considering $p\geq 2$, we use H\"{o}lder's inequality and  $\lambda$-intrinsic geometry condition~\eqref{intrinsic 2} to observe that
\begin{equation*}
\mathbf{I} \leq \left(\dashint_{Q_\rho^\lambda}\big[|\nabla u|^{p}+|F|^{p} \big]\d{x}\d{t}\right)^{\kappa} \stackrel{\eqref{intrinsic 2}}{\leq} (2^{d+2}K\lambda^p)^\kappa 
\quad \iff \quad \left[\frac{1}{2^{d+2}K}\mathbf{I}^{\nicefrac{1}{\kappa}}\right]^{\nicefrac{1}{p}} \leq \lambda,
\end{equation*}
while, appealing to Young's inequality with the exponents $\left(\frac{2}{(1-\kappa)p},\frac{2}{2-(1-\kappa)p}\right)$, we bound
\begin{align}\label{SP2 eq.5}
\left[\frac{\exc_{p}\left(u, Q_\rho^\lambda\right)}{\rho}\right]^{p}  &\leq C \mathbf{I}^{\frac{1}{\kappa p}\frac{2-p}{2}(1-\kappa)p+1}\mathbf{S}^{\frac{(1-\kappa)p}{2}} +C\mathbf{I}^{\frac{1}{\kappa p}\left[\frac{2-p}{2}(1-\kappa)p+(2-p)\kappa p\right]}\mathbf{S}^{\frac{(1-\kappa)p}{2}}\mathbf{I}^{p-1} \notag\\
&\quad \quad \quad  \quad \quad \quad +C\lambda^{\left[\frac{2-p}{2}(1-\kappa)p+(2-p)\kappa p\right]}\mathbf{S}^{\frac{(1-\kappa)p}{2}} \mathbf{T}^{\kappa (p-1)} \notag\\
&=2C\mathbf{I}^{\frac{2-p+p\kappa}{2\kappa}}\mathbf{S}^{\frac{(1-\kappa)p}{2}} +\underbrace{C\lambda^{-\frac{(p-2)(1+\kappa)p}{2}}\mathbf{S}^{\frac{(1-\kappa)p}{2}} \mathbf{T}^{\kappa (p-1)}}_{=:\mathbf{U}}\notag \\
&\leq \frac{\eps}{2} \mathbf{S}+C\eps^{-\frac{(1-\kappa)p}{2-(1-\kappa)p}}\mathbf{I}^{\nicefrac{1}{\kappa}} +\mathbf{U}
\end{align}
whenever $\eps \in (0,1]$, where $C(\data, K) <\infty$. Here, in the penultimate line, we have used the following manipulations:
\[
\left\{
\begin{aligned}
&\tfrac{1}{\kappa p}\tfrac{2-p}{2}(1-\kappa)p+1=\tfrac{2-p+p\kappa}{2\kappa}\,, \quad \mbox{and} \\
&\tfrac{1}{\kappa p}\left[\tfrac{2-p}{2}(1-\kappa)p+(2-p)\kappa p\right]+p-1=\tfrac{2-p+p\kappa}{2\kappa}.
\end{aligned}
\right.
\]
The remaining task is to estimate $\mathbf{U}$. Denote $\beta:=(1-\kappa)p/2$ and $\sigma:=\frac{(p-2)(1+\kappa)}{2-(1-\kappa)p}$ for short. Note that
\[
1+\sigma=\tfrac{2\kappa (p-1)}{2-(1-\kappa)p} \quad \Longrightarrow \quad 
\left\{
\begin{aligned}
& (1+\sigma)(1-\beta)=\kappa(p-1)\,, \quad \mbox{and} 
\\ 
& p\sigma(1-\beta)=\tfrac{(p-2)(1+\kappa)p}{2}\,.
\end{aligned}
\right.
\]
Appealing to Young's inequality with $\left(\frac{2}{(1-\kappa)p}, \frac{2}{2-(1-\kappa)p}\right)=\bigl(\nicefrac{1}{\beta},\,\nicefrac{1}{(1-\beta)}\bigr)$, we have that 
\begin{equation*}
\mathbf{U} =C\mathbf{S}^{\beta} \left[\mathbf{T} \left(\frac{\mathbf{T}}{\lambda^p}\right)^\sigma \right]^{1-\beta} \leq \frac{\eps}{2}\mathbf{S} + C\eps^{-\frac{\beta}{1-\beta}}\mathbf{T}\left(\frac{\mathbf{T}}{\lambda^p}\right)^\sigma\,.
\end{equation*}
Combining this with~\eqref{SP2 eq.5} yields the statement~\eqref{SP2display} in the case $p \geq 2$.

In the latter case $1<p<2$, observe that
\[
\mathbf{I} \leq (2^{d+2}K\lambda^p)^\kappa \quad \iff \quad \mathbf{I} \leq  (2^{d+2}K\lambda^p)^{\kappa(2-p)}\mathbf{I}^{p-1},
\]
thereby getting
\begin{align*}
\left[\frac{\exc_{p}\left(u, Q_\rho^\lambda\right)}{\rho}\right]^{p} &\leq C\lambda^{\frac{2-p}{2}(1-\kappa)p}\mathbf{S}^{\frac{(1-\kappa)p}{2}}\mathbf{I}+C\lambda^{\frac{2-p}{2}(1-\kappa)p+(2-p)\kappa p}\mathbf{S}^{\frac{(1-\kappa)p}{2}}\bigl(\mathbf{I}^{p-1}+\mathbf{T}^{\kappa (p-1)}\bigr)\\
&\leq \left[1+(2^{d+2}K)^{\kappa(2-p)}\right]C \lambda^{\frac{p(2-p)(1+\kappa)}{2}}\mathbf{S}^{\frac{(1-\kappa)p}{2}}\bigl(\mathbf{I}^{p-1}+\mathbf{T}^{\kappa (p-1)}\bigr)
\end{align*}
with a constant $C(d,p)<\infty$. Appealing to Young's inequality for three factors
\[
\left(\tfrac{2}{(1-\kappa)p},\tfrac{2}{(2-p)(1+\kappa)},\tfrac{1}{\kappa(p-1)}\right),
\]
we obtain that, for any $\delta >0$,
\[
\left[\frac{\exc_{p}\left(u, Q_\rho^\lambda\right)}{\rho}\right]^{p} \leq \delta\mathbf{S} +\delta \lambda^p+C\delta^{-\frac{1-\kappa(p-1)}{\kappa(p-1)}}\bigl(\mathbf{I}^{\nicefrac{1}{\kappa}}+\mathbf{T}\bigr)\,.
\]
Thus, the $\lambda$-intrinsic geometric condition~\eqref{intrinsic 2} and the H\"{o}lder inequality yield that
\begin{align*}
\left[\frac{\exc_{p}\left(u, Q_\rho^\lambda\right)}{\rho}\right]^{p} &\stackrel{\eqref{intrinsic 2}}{\leq} \delta 2^{d+2}K\left[\mathbf{S}+\left(\dashint_{Q_\rho^\lambda}\big[|\nabla u|^p+|F|^p\big]\d{x}\d{t}\right)\right]+C\delta^{-\frac{1-\kappa(p-1)}{\kappa(p-1)}}\bigl(\mathbf{I}^{\nicefrac{1}{\kappa}}+\mathbf{T}\bigr)\,,
\end{align*}
and hence, by taking $\delta=\eps/(2^{d+2}K)$ for any $\eps \in (0,1]$ the result~\eqref{SP2display} in turn also follows in the case $1<p<2$. The proof of Lemma~\ref{SP2} is complete.
\end{proof}

Finally, we prove the third version of the Sobolev-Poincar\'{e} inequality. We emphasize that the restriction $p>2d/(d+2)$ enters only at this stage, through the requirement $2d/[p(d+2)]<1$.

\begin{lemma}[Sobolev-Poincar\'{e} inequality III]\label{SP3}
Fix $p>\frac{2d}{d+2}$ and $ s \in (0,1)$ and set $\theta:=\frac{(1-s)p}{p-1}$. Let $\kappa=\max\left\{\frac{2d}{p(d+2)}, \frac{d}{d+2},\frac{q_\ast}{p}\right\}$. Under the same assumption as in Lemma~\ref{SP2}, there exists a constant $C(\data, K)<\infty$ such that
\begin{align}\label{SP3display}
\lambda^{p-2}\left[\frac{\exc_{2}\left(u, Q_\rho^\lambda\right)}{\rho}\right]^{2} &\leq \eps\left(\sup_{t \in I_\rho^\lambda(t_0)}\dashint_{B_\rho(x_0)}\frac{\bigl|u(t)-(u)_{Q_\rho^\lambda(z_0)}\bigr|^2}{\lambda^{2-p} \rho^2}\d{x} +\dashint_{Q_\rho^\lambda(z_0)}|\nabla u|^p\d{x}\d{t} \right)\notag\\[2mm]
& \quad \quad +C\eps^{-\left(\frac{p}{2} \vee \frac{1}{p-1}\right)}\left[\left(\dashint_{Q_\rho^\lambda(z_0)}|\nabla u|^{\kappa p}\d{x}\d{t}\right)^{\nicefrac{1}{\kappa}}\right.\notag\\[2mm]
& \quad \quad \quad \quad \quad\left. +\boldsymbol{\mathfrak{T}}_{p, \lambda}(\rho)\left(\rho^\theta \dashint_{Q_\rho^\lambda(z_0)}\cE[u]\d{x}\d{t}\right) +\dashint_{Q_\rho^\lambda(z_0)}|F|^p\d{x}\d{t}\right]
\end{align}
holds true whenever $\eps \in (0,1]$ and $\bm{\mathfrak{T}}_{p, \lambda}(\rho)$ is given by
\[
\bm{\mathfrak{T}}_{p, \lambda}(\rho):=
\left\{
\begin{aligned}
&\frac{1}{\lambda^{p \sigma_0}}\left(\rho^\theta \dashint_{Q_\rho^\lambda(z_0)}\cE[u]\d{x}\d{t}\right)^{\sigma_0}  &\mbox{if} \quad & p \geq 2\,, \\
&1  &\mbox{if} \quad & 1<p<2\,,
\end{aligned}
\right.
\]
with $\sigma_0:=\frac{(p-2)(1+\kappa)}{2-(1-\kappa)p}$ for short.
\end{lemma}
\begin{proof}
As usual, we may let $z_0=0$. We start from the Poincar\'{e} inequality:
\[
\left(\dashint_{B_\rho}\left|u-(u(t))_{B_\rho}\right|^2\d{x}\right)^{\frac{d}{d+2}} \leq c(d)\rho^{\frac{2d}{d+2}}\dashint_{B_\rho}|\nabla u(t)|^{\frac{2d}{d+2}}\d{x}\,.
\]
H\"{o}lder's inequality gives 
\begin{align*}
\left(\dashint_{B_\rho}\left|(u(t))_{B_\rho}-(u)_{Q_\rho^\lambda}\right|^2\d{x}\right)^{\frac{d}{d+2}}=\left|\,\dashint_{B_\rho}\big[u(t)-(u)_{Q_\rho^\lambda}\big]\d{x}\right|^{\frac{2d}{d+2}} \leq \dashint_{B_\rho}\left|u(t)-(u)_{Q_\rho^\lambda}\right|^{\frac{2d}{d+2}}\d{x}\,.
\end{align*}
Combining these estimates, we get
\begin{align*}
\left[\frac{\exc_{2}\left(u, Q_\rho^\lambda\right)}{\rho}\right]^{2} &\leq \left(\sup_{t \in I_\rho^\lambda}\dashint_{B_\rho}\frac{\bigl|u-(u)_{Q_\rho^\lambda}\bigr|^2}{\rho^2}\d{x}\right)^{\frac{2}{d+2}}\dashint_{I_\rho^\lambda}\left(\dashint_{B_\rho}\frac{\bigl|u-(u)_{Q_\rho^\lambda}\bigr|^2}{\rho^2}\d{x}\right)^{\frac{d}{d+2}}\d{t} \\[2mm]
&\leq c(d)\left(\sup_{t \in I_\rho^\lambda}\dashint_{B_\rho}\frac{\bigl|u-(u)_{Q_\rho^\lambda}\bigr|^2}{\rho^2}\d{x}\right)^{\frac{2}{d+2}}\\
&\quad \quad \quad \times \left\{\dashint_{Q_\rho^\lambda}|\nabla u|^{\frac{2d}{d+2}}\d{x}\d{t}+\left[\frac{\exc_{\frac{2d}{d+2}}(u,Q_\rho^\lambda)}{\rho}\right]^{\frac{2d}{d+2}}\right\}.
\end{align*}
For the excess term on the right side, appealing to Lemma~\ref{SP1} with $\kappa=\frac{2d}{(d+2)p}$, we have
\begin{align}\label{SP3 eq.1}
\left[\frac{\exc_{2}\left(u, Q_\rho^\lambda\right)}{\rho}\right]^{2} &\leq C_1 \left(\sup_{t \in I_\rho^\lambda}\dashint_{B_\rho}\frac{\bigl|u-(u)_{Q_\rho^\lambda}\bigr|^2}{\lambda^{2-p}\rho^2}\d{x}\right)^{\frac{2}{d+2}} \notag\\
&\quad \times\left[\dashint_{Q_\rho}|\nabla u|^{\frac{2d}{d+2}}\d{x}\d{t} +\lambda^{(2-p)\frac{2d}{d+2}}\left(\dashint_{Q_\rho^\lambda}\big[|\nabla u|^{\frac{2d}{d+2}}+|F|^{\frac{2d}{d+2}}\big]\d{x}\d{t}\right)^{p-1} \right.\\
&\quad \quad \quad \quad \left.+\lambda^{(2-p)\frac{2d}{d+2}}\left(\rho^\theta \dashint_{Q_\rho^\lambda}\cE[u]\d{x}\d{t}\right)^{\frac{2d}{(d+2)p} (p-1)}\right]
\end{align}
for a constant $C_1(\data,K)<\infty$.
As used in~\eqref{SP2 eq.3}, we temporarily use the shorthand notation
\begin{equation}\label{SP2 eq.3}
\left\{
\begin{aligned}
& \mathbf{S}:=\sup_{t \in I_\rho^\lambda}\dashint_{B_\rho}\frac{\bigl|u(t)-(u)_{Q_\rho^\lambda}\bigr|^2}{\lambda^{2-p}\rho^2}\d{x}\,, \\
& \mathbf{I}:=\dashint_{Q_\rho^\lambda}\bigl(|\nabla u|^{\frac{2d}{d+2}}+|F|^{\frac{2d}{d+2}} \bigr)\d{x}\d{t}\,, \quad \mbox{and} 
\\ 
& \mathbf{T}:=\rho^\theta \dashint_{Q_\rho^\lambda}\cE[u]\d{x}\d{t}\,.
\end{aligned}
\right.
\end{equation}
and so, the previous display~\eqref{SP3 eq.1} recasts
\begin{align}\label{SP3 eq.2}
\left[\frac{\exc_{2}\left(u, Q_\rho^\lambda\right)}{\rho}\right]^{2} \leq C_1\mathbf{S}^{\frac{d}{d+2}}\bigl(\mathbf{I}+\lambda^{(2-p)\frac{2d}{d+2}}\mathbf{I}^{p-1}+\lambda^{(2-p)\frac{2d}{d+2}}\mathbf{T}^{\frac{2d}{(d+2)p} (p-1)}\bigr)\,.
\end{align}
The proof is now split in two cases. Namely, $\frac{2d}{d+2}<p <2$ and $p \geq 2$. When we are considering the case $\frac{2d}{d+2}<p<2$, H\"{o}lder's inequality and $\lambda$-intrinsic geometry condition~\eqref{intrinsic 1} yield that
\[
\mathbf{I} \leq \left(\dashint_{Q_\rho^\lambda}\bigl[|\nabla u|^p+|F|^p\bigr]\d{x}\d{t}\right)^{\frac{2d}{(d+2)p}} \stackrel{\eqref{intrinsic 1}}{\leq} K^{\frac{2d}{(d+2)p}}\lambda^{\frac{2d}{d+2}}\,,
\]
that is,
\[
\mathbf{I} \leq C_2\lambda^{(2-p)\frac{2d}{d+2}}\mathbf{I}^{p-1}
\]
with a constant $C_2(d,p,K)<\infty$; therefore, it follows from~\eqref{SP3 eq.2} that
\begin{equation*}
\left[\frac{\exc_{2}\left(u, Q_\rho^\lambda\right)}{\rho}\right]^{2} \leq C_3\lambda^{(2-p)\frac{2d}{p(d+2)}}\mathbf{S}^{\frac{d}{d+2}}\bigl(\mathbf{I}^{p-1}+\mathbf{T}^{\frac{2d}{(d+2)p} (p-1)}\bigr)\,.
\end{equation*}
Applying Young's inequality with a pair of exponents
\[
\left(\tfrac{d+2}{2},\tfrac{p(d+2)}{2d(p-1)},\tfrac{p(d+2)}{d(2-p)}\right)
\]
yields 
\[
\left[\frac{\exc_{2}\left(u, Q_\rho^\lambda\right)}{\rho}\right]^{2} \leq \eps \bigl(\lambda^{2-p}\mathbf{S}+\lambda^2\bigr)+C_3\eps^{-\frac{2p+d(2-p)}{2d(p-1)}}\lambda^{2-p} \left(\mathbf{I}^{\frac{p(d+2)}{2d}} + \mathbf{T}\right)
\]
whenever $\eps \in (0,1]$, where $C_3(d,p,L,K)<\infty$. Finally, appealing to H\"{o}lder's inequality and~\eqref{intrinsic 1}, we deduce that 
\begin{align}\label{SP3 eq.3}
\lambda^{p-2}&\left[\frac{\exc_{2}\left(u, Q_\rho^\lambda\right)}{\rho}\right]^{2}
\notag\\[2mm]
&\leq \eps \lambda^p +\eps \sup_{t \in I_\rho^\lambda}\dashint_{B_\rho}\frac{\bigl|u(t)-(u)_{Q_\rho^\lambda}\bigr|^2}{\lambda^{2-p}\rho^2}\d{x} \notag\\[2mm]
& \quad \quad \quad+C_3\eps^{-\frac{1}{p-1}}\left[\left(\dashint_{Q_\rho^\lambda}|\nabla u|^{\frac{2d}{d+2}}\d{x}\d{t}\right)^{\frac{p(d+2)}{2d}}+\dashint_{Q_\rho^\lambda}|F|^p\d{x}\d{t}+ \rho^\theta \dashint_{Q_\rho^\lambda}\cE[u]\d{x}\d{t}\right]\notag\\[2mm]
& \leq \eps \left(\sup_{t \in I_\rho^\lambda}\dashint_{B_\rho}\frac{\bigl|u(t)-(u)_{Q_\rho^\lambda}\bigr|^2}{\lambda^{2-p}\rho^2}\d{x} +\dashint_{Q_\rho^\lambda}|\nabla u|^{p}\d{x}\d{t}\right)\notag\\[2mm]
& \quad \quad \quad+C_3 \eps^{-\frac{1}{p-1}}\left[\left(\dashint_{Q_\rho^\lambda}|\nabla u|^{\kappa p}\d{x}\d{t}\right)^{\nicefrac{1}{\kappa}}+ \rho^\theta \dashint_{Q_\rho^\lambda}\cE[u]\d{x}\d{t}+\dashint_{Q_\rho^\lambda}|F|^p\d{x}\d{t} \right]\,,
\end{align}
where, in the penultimate line, we used that $\eps^{-\frac{2p+d(2-p)}{2d(p-1)}}\leq \eps^{-\frac{1}{p-1}}$.

On the other hand, in the latter case $p\geq 2$, Young's and H\"{o}lder's inequalities yield that
\begin{align*}
\left[\frac{\exc_{2}\left(u, Q_\rho^\lambda \right)}{\rho}\right]^{2} &=\bigl(\delta^{\frac{p-2}{p}}\lambda^{\frac{2(p-2)}{p}}\bigr)\left(\delta^{-\frac{p-2}{p}}\lambda^{-\frac{2(p-2)}{p}}\left[\frac{\exc_{2}\left(u, Q_\rho^\lambda \right)}{\rho}\right]^{2} \right) \\
&\leq \delta\lambda^2+\delta^{-\frac{p-2}{2}}\lambda^{-(p-2)}\left[\frac{\exc_{p}\left(u, Q_\rho^\lambda\right)}{\rho}\right]^{p}\,,
\end{align*}
which together with Lemma~\ref{SP2} yields
\begin{align*}
\left[\frac{\exc_{2}\left(u, Q_\rho^\lambda\right)}{\rho}\right]^{2}  \leq \delta \lambda^2&+\delta^{-\frac{p-2}{2}}\lambda^{-p+2}\left\{\eta \left(\sup_{t \in I_\rho^\lambda} \dashint_{B_\rho}\frac{\bigl|u(t)-(u)_{Q_\rho^\lambda}\bigr|^2}{\lambda^{2-p}\rho^2}\d{x}+\dashint_{Q_\rho^\lambda}|\nabla u|^p\d{x}\d{t}\right)\right.\\
&\left. +C_4\eta^{-\frac{(1-\kappa)p}{2-(1-\kappa)p}}\left[\left(\dashint_{Q_\rho^\lambda}|\nabla u|^{\kappa p}\d{x}\d{t}\right)^{\nicefrac{1}{\kappa}}+\bm{\mathfrak{T}}_{p,\lambda}(\rho)\left(\rho^\theta \dashint_{Q^\lambda_\rho}\cE[u]\d{x}\d{t}\right) \right.\right.\\
&\left. \left. \quad \quad \quad \quad \quad \quad \quad \quad+\dashint_{Q_\rho^\lambda}|F|^p\d{x}\d{t}\right]\right\}
\end{align*}
for all $\delta,\eta \in (0,1]$, where $\kappa:=\max\{\frac{d}{d+2},\frac{q_\ast}{p}\}<1$ and $C_4(\data, K)<\infty$. We select $\delta$ and $\eta$ such that, for any $\eps \in (0,1]$, 
\[
\delta=\eps/2 \quad \mbox{and}\quad \delta^{-\frac{p-2}{2}}\eta K=\eps/2\,,
\]
and hence we have by~\eqref{intrinsic 1} that
\begin{align}\label{SP3 eq.4}
\lambda^{p-2}\left[\frac{\exc_{2}\left(u, Q_\rho^\lambda \right)}{\rho}\right]^{2}  &\leq \eps \lambda^p+\eps\sup_{t \in I_\rho^\lambda} \dashint_{B_\rho}\frac{\bigl|u(t)-(u)_{Q_\rho^\lambda}\bigr|^2}{\lambda^{2-p}\rho^2}\d{x} \notag\\
&\quad +C_5\eps^{-\frac{p}{2}\cdot \frac{(1-\kappa)p}{2-(1-\kappa)p}}\left[\left(\dashint_{Q_\rho^\lambda}|\nabla u|^{\kappa p}\d{x}\d{t}\right)^{\nicefrac{1}{\kappa}} \right. \notag\\
&\left. \quad \quad \quad \quad +\bm{\mathfrak{T}}_{p,\lambda}(\rho)\left(\rho^\theta \dashint_{Q^\lambda_\rho}\cE[u]\d{x}\d{t}\right)+\dashint_{Q_\rho^\lambda}|F|^p\d{x}\d{t}\right]\,.
\end{align}
Using $\eps^{-\frac{p}{2}\cdot \frac{(1-\kappa)p}{2-(1-\kappa)p}}<\eps^{-\nicefrac{p}{2}}$, the desired result follows from~\eqref{SP3 eq.3} and~\eqref{SP3 eq.4}.
\end{proof}


\section{Reverse H\"{o}lder inequality for the spatial gradient}\label{Sect.4}
This section is devoted to the reverse H\"{o}lder inequality for the spatial gradient. 
\begin{lemma}[Reverse H\"{o}lder inequality]\label{t.Reverse Holder}
Let $p>\frac{2d}{d+2}$, $ s \in (0,1)$ and fix $\theta:=\frac{(1-s)p}{p-1}$. Let $u$ be a weak solution to~\eqref{maineq} in the sense of Definition~\ref{dfn u}, under the structural assumptions~(A1)--(A4) and $F \in L^p_{\mathrm{loc}}\left(\Omega_T\,;\R^d\right)$. Let $\kappa \in (0,1)$ be such that $\kappa =\max\left\{\frac{2d}{p(d+2)}, \frac{d}{d+2}, \frac{q_\ast}{p} \right\}$. Suppose that $Q_\rho^\lambda(z_0) \Subset \Omega_T$ with $0<\rho \leq 1$ satisfies the conditions~\eqref{intrinsic 1}--\eqref{intrinsic 2}. There exists a constant $C_{\mathsf{RH}}(\data,K)<\infty$ such that
\begin{align*}
\dashint_{Q_{\rho/2}^\lambda(z_0)}|\nabla u|^p\d{x}\d{t} &\leq C_{\RH}\left[\left(\dashint_{Q_{\rho}^\lambda(z_0)}|\nabla u|^{\kappa p}\d{x}\d{t}\right)^{\nicefrac{1}{\kappa}}+\dashint_{Q_{\rho}^\lambda(z_0)}\bigl(|F|^p+1\bigr)\d{x}\d{t}\right] \\[2mm]
&\quad \quad +C_{\RH}\left(\rho^{\frac{(1-s)p}{p-1}}\dashint_{I_{\rho}^\lambda(t_0)}\int_{\R^d \setminus B_{\rho/2}(x_0)}\frac{\big|u(x,t)-(u)_{Q_{\rho}^\lambda}\big|^{p}}{|x-x_0|^{d+sp}}\d{x}\d{t}\right) \\[2mm]
&\quad \quad +C_{\RH}\bm{\mathfrak{U}}_{p, \lambda}(\rho) \left(\rho^\theta \dashint_{Q_\rho^\lambda(z_0)} \cE[u]\d{x}\d{t}\right)\,,
\end{align*}
where
\[
\bm{\mathfrak{U}}_{p, \lambda}(\rho):=
\left\{
\begin{aligned}
&\frac{1}{\lambda^{p \sigma_0}}\left(\rho^\theta \dashint_{Q_{\rho}^\lambda(z_0)}\cE[u]\d{x}\d{t}\right)^{\sigma_0}  + \frac{1}{\lambda^{p(p-2)}}\left(\rho^\theta \dashint_{Q_\rho^\lambda(z_0)}\cE[u]\d{x}\d{t}\right)^{p-2}&\mbox{if} \quad & p \geq 2\,, \\
&1  &\mbox{if} \quad & \tfrac{2d}{d+2}<p<2\,,
\end{aligned}
\right.
\]
with $\sigma_0:=\frac{(p-2)(1+\kappa)}{2-(1-\kappa)p}$.
\end{lemma}
\begin{proof}
As usual, we may let $z_0=0$. Fix radii $\rho/2 \leq r_1<r_2\leq \rho \leq 1$. Appealing to the Caccioppoli inequality (Lemma~\ref{t.caccioppoli}) with level $k=(u)_{Q_{r_1}^\lambda}$ yields that
\begin{equation}\label{Reverse eq.1}
\sup_{t \in I^\lambda_{r_1}}\dashint_{B_{r_1}}\frac{\big|u(t)-(u)_{Q_{r_1}^\lambda}\big|^2}{\lambda^{2-p}r_1^2}\d{x}+\dashint_{Q_{r_1}^\lambda}|\nabla u|^p\d{x}\d{t} \leq   \mathbf{I}+\mathbf{II}+\mathbf{III}+\mathbf{IV}+\mathbf{V}\,,
\end{equation}
where we discarded the fractional integral on the left side and denoted, for short, 
\begin{equation*}
\left\{
\begin{aligned}
& \mathbf{I}:=C_{\mathsf{Cac}}\dashint_{Q_{r_2}^\lambda}\frac{\bigl|u-(u)_{Q_{r_1}^\lambda}\bigr|^p}{(r_2-r_1)^p}\d{x}\d{t}\,, \\
& \mathbf{II}:=C_{\mathsf{Cac}}\dashint_{Q_{r_2}^\lambda}\frac{\bigl|u-(u)_{Q_{r_1}^\lambda}\bigr|^2}{\lambda^{2-p}\left(r_2^2-r_1^2\right)}\d{x}\d{t}\,,\\
& \mathbf{III}:=C_{\mathsf{Cac}}\dashint_{Q_{r_2}^\lambda}|F|^p\d{x}\d{t}\,, \\
& \mathbf{IV}:=C_{\mathsf{Cac}}\frac{1}{(r_2-r_1)^p}\dashint_{I_{r_2}^\lambda}\int_{B_{r_2}}\dashint_{B_{r_2}}\frac{\bigl|u(x,t)-(u)_{Q_{r_1}^\lambda}\bigr|^p}{|x-y|^{d-(1-s)p}}\d{x}\d{y}\d{t} \,, \quad \mbox{and} 
\\ 
& \mathbf{V}:=C_{\mathsf{Cac}}\dashint_{Q_{r_2}^\lambda}\left(\sup_{x\in B_{(r_1+r_2)/2}}\int_{\R^d \setminus B_{r_2}}\frac{\bigl|u(y,t)-(u)_{Q_{r_1}^\lambda}\bigr|^{p-1}}{|x-y|^{d+sp}}\d{y}\right)\bigl|u-(u)_{Q_{r_1}^\lambda}\bigr|\d{x}\d{t} \,.
\end{aligned}
\right.
\end{equation*}
The quantity $C_{\mathsf{Cac}}$ denotes the constant, depending only on $\data$, as in the Caccioppoli inequality~\eqref{e.caccioppoli}. It can be henceforth absorbed into various constants in the argument below, since it depends only on $\data$.
\smallskip

\emph{Step 1: Estimates of $\mathbf{I}$--$\mathbf{V}$.} To lighten the notation, we denote $
\mathcal{R}_{r_1,r_2}:=\frac{r_2}{r_2-r_1}$. Since by $\rho/2 \leq r_1<r_2\leq \rho$
\begin{equation}\label{e.R}
\frac{r_2-r_1}{r_1}=\frac{r_2}{r_1}-1 \leq 1 \iff 1\leq \frac{r_2}{r_1} \leq \mathcal{R}_{r_1,r_2}
\end{equation}
Lemma~\ref{t.useful lemma} yields that
\begin{equation}\label{Reverse eq.2}
\mathbf{I} \leq C\mathcal{R}_{r_1,r_2}^2 \dashint_{Q_{r_2}^\lambda}\frac{\bigl|u-(u)_{Q_{r_2}^\lambda}\bigr|^p}{r_2^p}\d{x}\d{t},
\end{equation}
where we have used by~\eqref{e.R} that $
|Q_{r_2}^\lambda| / |Q_{r_1}^\lambda| \leq 2^d\mathcal{R}_{r_1,r_2}^2$. Next, turning to the estimate of $\mathbf{II}$, Lemma~\ref{t.useful lemma} and a simple algebra $r_2^2-r_1^2 \geq r_2^2\mathcal{R}_{r_1,r_2}^{-2}$ yield that
\begin{equation}\label{Reverse eq.3}
\mathbf{II} \leq C\lambda^{p-2}\mathcal{R}_{r_1,r_2}^2\dashint_{Q_{r_2}^\lambda}\frac{\bigl(u-(u)_{Q_{r_2}^\lambda}\bigr)^2}{r_2^2}\d{x}\d{t}.
\end{equation}
On the other hand, since for any $x, y \in B_{r_2}$ 
\[
\int_{B_{r_2}}\frac{\d{y}}{|x-y|^{d-(1-s)p}} \leq \int_{B_{2r_2}(x)}\frac{\d{y}}{|x-y|^{d-(1-s)p}}=\frac{c(d)}{(1-s)p}(2r_2)^{(1-s)p},
\]
we apply Lemma~\ref{t.useful lemma} again to get
\begin{equation}\label{Reverse eq.4}
\mathbf{IV} \leq Cr_2^{(1-s)p}\dashint_{Q_{r_2}^\lambda}\frac{\bigl|u-(u)_{Q_{r_1}^\lambda}\bigr|^p}{(r_2-r_1)^p}\d{x}\d{t} \leq C\mathcal{R}_{r_1,r_2}^2\dashint_{Q_{r_2}^\lambda}\frac{\bigl|u-(u)_{Q_{r_2}^\lambda}\bigr|^p}{(r_2-r_1)^p}\d{x}\d{t}
\end{equation}
with a constant $C(\data)<\infty$, where we also used
$r_2^{(1-s)p}\leq 1$ and $|Q_{r_2}^\lambda|/|Q_{r_1}^\lambda| \leq 2^d\mathcal{R}_{r_1,r_2}^2$. 

Finally, we will estimate $\mathbf{V}$. Since
\[
\frac{|y|}{|y-x|}\leq 1+\frac{|x|}{|y-x|} \leq 1+\frac{r_1+r_2}{2|y-x|} \qquad \mbox{and} \qquad |y-x| \geq |y|-|x| \geq \frac{r_2-r_1}{2}
\]
whenever $x \in B_{(r_1+r_2)/2}$ and $y \in \R^d \setminus B_{r_2}$, we obtain that
\[
\frac{|y|}{|y-x|}\leq 1+\frac{r_1+r_2}{r_2-r_1}=2\mathcal{R}_{r_1,r_2}\,.
\]
Young's inequality with the exponents $\left(\frac{p}{p-1},p\right)$ yields that, for every $\tilde{\delta}>0$, 
\begin{align}\label{Reverse eq.5}
\mathbf{V}
& \leq   C_1\mathcal{R}_{r_1,r_2}^{d+sp} \dashint_{I^\lambda_{r_2}}\left(r_2\int_{\R^d \setminus B_{r_2}}\frac{\bigl|u(y,t)-(u)_{Q_{r_1}^\lambda}\bigr|^{p-1}}{|y|^{d+sp}}\d{y}\right)\left(\dashint_{B_{r_2}}\frac{\bigl|u-(u)_{Q_{r_1}^\lambda}\bigr|}{r_2}\d{x}\right)\d{t} \notag\\[2mm]
&\leq \tilde{\delta} \dashint_{I^\lambda_{r_2}}\left(r_2 \int_{\R^d \setminus B_{r_2}}\frac{\bigl|u(y,t)-(u)_{Q_{r_1}^\lambda}\bigr|^{p-1}}{|y|^{d+sp}}\d{y}\right)^{\nicefrac{p}{(p-1)}}\d{t}  \notag\\[2mm]
&\quad \quad \quad \quad +C_1\tilde{\delta}^{1-p}\mathcal{R}_{r_1,r_2}^{p(d+sp)}\dashint_{I_{r_2}^\lambda}\left(\dashint_{B_{r_2}}\left|\frac{u-(u)_{Q_{r_2}^\lambda}}{r_2}\right|\d{x}\right)^p\d{t} \notag\\
&=:\mathbf{V}_1+\mathbf{V}_2\,,
\end{align}
for a constant $C_1(\data)<\infty$, where
\begin{equation*}
\left\{
\begin{aligned}
& \mathbf{V}_1:=\tilde{\delta} \dashint_{I^\lambda_{r_2}}\left(r_2 \int_{\R^d \setminus B_{r_2}}\frac{\bigl|u(y,t)-(u)_{Q_{r_1}^\lambda}\bigr|^{p-1}}{|y|^{d+sp}}\d{y}\right)^{\nicefrac{p}{(p-1)}}\d{t}\,,\\[2mm] 
& \mathbf{V}_2:=C_1\tilde{\delta}^{1-p}\mathcal{R}_{r_1,r_2}^{p(d+sp)}\dashint_{I_{r_2}^\lambda}\left(\dashint_{B_{r_2}}\left|\frac{u-(u)_{Q_{r_2}^\lambda}}{r_2}\right|\d{x}\right)^p\d{t}\,.
\end{aligned}
\right.
\end{equation*}
 We continue estimating, by H\"{o}lder's inequality and~\eqref{formula}, that
\begin{align*}
\dashint_{I_{r_2}^\lambda}&\left(r_2\int_{\R^d \setminus B_{r_2}}\frac{\bigl|u(y,t)-(u)_{Q_{r_1}^\lambda}\bigr|^{p-1}}{|y|^{d+sp}}\d{y}\right)^{\nicefrac{p}{(p-1)}}\d{t}\\
&\leq \dashint_{I_{r_2}^\lambda}r_2^{\nicefrac{p}{(p-1)}}\left(\int_{\R^d \setminus B_{r_2}}\frac{\bigl|u(y,t)-(u)_{Q_{r_1}^\lambda}\bigr|^{p}}{|y|^{d+sp}}\d{y}\right)\left(\int_{\R^d \setminus B_{r_2}}\frac{\d{y}}{|y|^{d+sp}}\right)^{\nicefrac{1}{(p-1)}}\d{t}\\
&\!\!\!\stackrel{\eqref{formula}}{=}c(d,s,p)r_2^{-\frac{sp}{p-1}}\dashint_{I_{r_2}^\lambda}\int_{\R^d \setminus B_{r_2}}\frac{\bigl|u(y,t)-(u)_{Q_{r_1}^\lambda}\bigr|^{p}}{|y|^{d+sp}}\d{y}\d{t}\\
&\!\lesssim \,\,r_2^{\frac{(1-s)p}{p-1}}\dashint_{I_{r_2}^\lambda}\int_{\R^d \setminus B_{r_2}}\frac{\bigl|u(y,t)-(u)_{Q_{\rho}^\lambda}\bigr|^{p}}{|y|^{d+sp}}\d{y}\d{t} \\
&\quad \quad \quad \quad \quad \quad \quad+r_2^{\frac{(1-s)p}{p-1}}\bigl|(u)_{Q_{\rho}^\lambda}-(u)_{Q_{r_1}^\lambda}\bigr|^{p}\int_{\R^d \setminus B_{\rho/2}}\frac{\d{y}}{|y|^{d+sp}}\\
&\!\!\!\stackrel{\eqref{formula}}{\lesssim} \,\,r_2^{\frac{(1-s)p}{p-1}}\dashint_{I_{r_2}^\lambda}\int_{\R^d \setminus B_{r_2}}\frac{\bigl|u(y,t)-(u)_{Q_{\rho}^\lambda}\bigr|^{p}}{|y|^{d+sp}}\d{y}\d{t} \\
&\quad\quad \quad \quad \quad \quad \quad+r_2^{\frac{(1-s)p}{p-1}}\rho^{-sp}\left|(u)_{Q_{\rho}^\lambda}-(u)_{Q_{r_1}^\lambda}\right|^{p}
\end{align*}
with a constant implicit in ``$\lesssim$'' depending on $d,s$ and $p$.  This, together with the estimate
\begin{align*}
\left|(u)_{Q_\rho^\lambda}-(u)_{Q_{r_1}^\lambda}\right| &\leq \frac{|Q_\rho^\lambda|}{|Q_{r_1}^\lambda|}\dashint_{Q_{\rho}^\lambda}\bigl|u-(u)_{Q_\rho^\lambda}\bigr|\d{x}\d{t} \leq 2^{d+2}\rho\left(\dashint_{Q_{\rho}^\lambda}\left|\frac{u-(u)_{Q_\rho^\lambda}}{\rho}\right|^{\kappa p}\d{x}\d{t}\right)^{\nicefrac{1}{\kappa p}}\,,
\end{align*}
yields that
\begin{align}\label{Reverse eq.6}
\mathbf{V}_1\leq c\tilde{\delta}\rho^{\frac{(1-s)p}{p-1}}\dashint_{I_{\rho}^\lambda}\int_{\R^d \setminus B_{\rho/2}}\frac{\bigl|u(y,t)-(u)_{Q_{\rho}^\lambda}\bigr|^{p}}{|y|^{d+sp}}\d{y}\d{t}+c\tilde{\delta}\rho^{\frac{(1-s)p^2}{p-1}}\left(\dashint_{Q_{\rho}^\lambda}\left|\frac{u-(u)_{Q_\rho^\lambda}}{\rho}\right|^{\kappa p}\d{x}\d{t}\right)^{\frac{1}{\kappa}}
\end{align}
with $c(d,s,p)<\infty$, where $\kappa$ is so that $\kappa p \geq q_\ast$. By Lemma~\ref{SP1}, we find that 
\begin{align}\label{e.reverse starting point}
&\left(\dashint_{Q_{\rho}^\lambda}\left|\frac{u-(u)_{Q_\rho^\lambda}}{\rho}\right|^{\kappa p}\d{x}\d{t}\right)^{\frac{1}{\kappa p}} \notag\\
&\quad \quad \leq C_1 \lambda^{2-p}\left(\dashint_{Q_\rho^\lambda}\bigl( |\nabla u|^{\kappa p}+|F|^{\kappa p}\bigr)\d{x}\d{t}\right)^{\frac{p-1}{\kappa p}} \notag \\
&\quad \quad \quad+C_1 \lambda^{2-p} \left(\rho^\theta \dashint_{Q_\rho^\lambda}\cE[u]\d{x}\d{t}\right)^{\frac{p-1}{p}}+C_1\left(\dashint_{Q_\rho^\lambda}|\nabla u|^{\kappa p}\d{x}\d{t}\right)^{\frac{1}{\kappa p}}
\end{align}
with a constant $C_1(\data)<\infty$. We now consider two cases $p\geq 2$ and $\frac{2d}{d+2} < p < 2$. When handling $p\geq 2$, the $\lambda$-intrinsic geometry condition~\eqref{intrinsic 1}$_1$ and H\"{o}lder's inequality assure that
\[
\left[\frac{1}{K}\left(\dashint_{Q_\rho^\lambda}\bigl( |\nabla u|^{\kappa p}+|F|^{\kappa p}\bigr)\d{x}\d{t}\right)^{\nicefrac{1}{\kappa}}\right]^{\nicefrac{1}{p}} \leq \lambda\,,
\]
which, together with the H\"{o}lder inequality, gives
\begin{align}\label{Reverse eq.7}
\left(\dashint_{Q_{\rho}^\lambda}\left|\frac{u-(u)_{Q_\rho^\lambda}}{\rho}\right|^{\kappa p}\d{x}\d{t}\right)^{\frac{1}{\kappa p}} &\leq C_2 \left(\dashint_{Q_\rho^\lambda}|\nabla u|^{\kappa p}\d{x}\d{t}\right)^{\frac{1}{\kappa p}}+C_2 \left(\dashint_{Q_\rho^\lambda}\bigl(|F|^p+1\bigr)\d{x}\d{t}\right)^{\nicefrac{1}{p}} \notag\\
& \quad \quad \quad+C_2 \lambda^{2-p} \left(\rho^\theta \dashint_{Q_\rho^\lambda}\cE[u]\d{x}\d{t}\right)^{\frac{p-1}{p}}\,.
\end{align}
On the one hand, in the latter case $\frac{2d}{d+2}<p < 2$, by~\eqref{intrinsic 1}$_2$ and Young's inequality with the exponent $\left(\frac{1}{2-p}, \frac{1}{p-1}\right)$, we obtain from~\eqref{e.reverse starting point} that
\begin{align}\label{Reverse eq.8}
&\left(\dashint_{Q_{\rho}^\lambda}\left|\frac{u-(u)_{Q_\rho^\lambda}}{\rho}\right|^{\kappa p}\d{x}\d{t}\right)^{\frac{1}{\kappa p}} \notag\\
&\quad \quad \leq \frac{1}{2}\left(\dashint_{Q_{\rho/2}^\lambda }\big(|\nabla u|^p+|F|^p\big)\d{x}\d{t}\right)^{\nicefrac{1}{p}}+C_3\left(\dashint_{Q_\rho^\lambda }\bigl(|\nabla u|^{\kappa p}+|F|^{\kappa p} \bigr)\d{x}\d{t}\right)^{\nicefrac{1}{\kappa p}}\notag\\
&\quad \quad \quad \quad \quad \quad +C_3\lambda^{2-p} \left(\rho^\theta \dashint_{Q_\rho^\lambda}\cE[u]\d{x}\d{t}\right)^{\frac{p-1}{p}}\notag\\
&\quad \quad \leq \frac{1}{2}\left(\dashint_{Q_{\rho/2}^\lambda }|\nabla u|^p\d{x}\d{t}\right)^{\nicefrac{1}{p}}+C_3\left(\dashint_{Q_\rho^\lambda }|\nabla u|^{\kappa p}\d{x}\d{t}\right)^{\nicefrac{1}{\kappa p}}\notag\\
& \quad \quad \quad \quad \quad \quad +C_3\left(\dashint_{Q_\rho^\lambda }\bigl(|F|^{p} +1\bigr)\d{x}\d{t}\right)^{\nicefrac{1}{p}}+C_3\lambda^{2-p} \left(\rho^\theta \dashint_{Q_\rho^\lambda}\cE[u]\d{x}\d{t}\right)^{\frac{p-1}{p}}\,,
\end{align}
where to obtain the last line we used H\"{o}lder's inequality. As a result, the combination of the previous displays~\eqref{Reverse eq.7} and~\eqref{Reverse eq.8} yields that, for all $p>\frac{2d}{d+2}$, 
\begin{align}\label{Reverse key}
&\left(\dashint_{Q_{\rho}^\lambda}\left|\frac{u-(u)_{Q_\rho^\lambda}}{\rho}\right|^{\kappa p}\d{x}\d{t}\right)^{\frac{1}{\kappa p}} \notag\\
&\quad \quad \leq \frac{1}{2}\left(\dashint_{Q_{\rho/2}^\lambda }|\nabla u|^p\d{x}\d{t}\right)^{\nicefrac{1}{p}}+C_3\left(\dashint_{Q_\rho^\lambda }|\nabla u|^{\kappa p}\d{x}\d{t}\right)^{\nicefrac{1}{\kappa p}} \notag\\
&\quad \quad+C_3\left(\dashint_{Q_\rho^\lambda }\bigl(|F|^{p} +1\bigr)\d{x}\d{t}\right)^{\nicefrac{1}{p}}+C_3\lambda^{2-p} \left(\rho^\theta \dashint_{Q_\rho^\lambda}\cE[u]\d{x}\d{t}\right)^{\frac{p-1}{p}}\,,
\end{align}
where $C_3(\data, K)<\infty$. Thus, combining~\eqref{Reverse eq.6} and~\eqref{Reverse key} with~\eqref{Reverse eq.5} and applying the H\"{o}lder inequality to $\mathbf{V}_2$, we have that
\begin{align*}
\mathbf{V} &\leq C_4\tilde{\delta}\left(\rho^{\frac{(1-s)p}{p-1}}\dashint_{I_{\rho}^\lambda}\int_{\R^d \setminus B_{\rho/2}}\frac{\bigl|u(y,t)-(u)_{Q_{\rho}^\lambda}\bigr|^{p}}{|y|^{d+sp}}\d{y}\d{t}\right) \notag\\[2mm]
&\quad+C_4\tilde{\delta}\underbrace{\rho^{\frac{(1-s)p^2}{p-1}}}_{\leq 1}\left[\frac{1}{2^p}\,\dashint_{Q_{\rho/2}^\lambda} |\nabla u|^p\d{x}\d{t}+ C_3^p\left(\dashint_{Q_\rho^\lambda}|\nabla u|^{\kappa p}\d{x}\d{t}\right)^{\nicefrac{1}{\kappa}}\right. \\
&\quad \quad \quad \quad \quad \quad +\left. C_3^p\dashint_{Q_\rho^\lambda}\bigl(|F|^{p}+1 \bigr)\d{x}\d{t}+C_3^p\lambda^{p(2-p)} \left(\rho^\theta \dashint_{Q_\rho^\lambda}\cE[u]\d{x}\d{t}\right)^{p-1}\right]\\
&\quad  +C_4\tilde{\delta}^{1-p}\mathcal{R}_{r_1,r_2}^{p(d+sp)}\dashint_{Q_{r_2}^\lambda}\left|\frac{u-(u)_{Q_{r_2}^\lambda}}{r_2}\right|^p\d{x}\d{t}\,.
\end{align*}
Select $\tilde{\delta}>0$ so that $C_4\tilde{\delta} = \delta/2 $ for given $\delta_0 \in (0,1]$ and then rearrange to obtain that
\begin{align}\label{Reverse eq.9}
\mathbf{V} &\leq \frac{\delta_0}{2} \left(\rho^{\frac{(1-s)p}{p-1}}\dashint_{I_{\rho}^\lambda}\int_{\R^d \setminus B_{\rho/2}}\frac{\bigl|u(y,t)-(u)_{Q_{\rho}^\lambda}\bigr|^{p}}{|y|^{d+sp}}\d{y}\d{t}\right)+\frac{\delta_0}{4}\dashint_{Q_{\rho/2}^\lambda } |\nabla u|^p\d{x}\d{t}\notag\\
&\quad \quad \quad \quad +C_5\left(\dashint_{Q_\rho^\lambda }|\nabla u|^{\kappa p}\d{x}\d{t}\right)^{1/\kappa}+C_5\dashint_{Q_\rho^\lambda}\bigl(|F|^p+1\bigr)\d{x}\d{t} \notag\\
&\quad \quad \quad \quad +C_5\lambda^{p(2-p)} \left(\rho^\theta \dashint_{Q_\rho^\lambda}\cE[u]\d{x}\d{t}\right)^{p-1} \notag\\
&\quad \quad \quad \quad+C_5\delta_0^{1-p}\mathcal{R}_{r_1,r_2}^{p(d+sp)}\dashint_{Q_{r_2}^\lambda}\left|\frac{u-(u)_{Q_{r_2}^\lambda}}{r_2}\right|^p\d{x}\d{t}
\end{align}
for a constant $C_5(\data, K)<\infty$.
Merging the previous four estimates~\eqref{Reverse eq.2}--\eqref{Reverse eq.4} and~\eqref{Reverse eq.9} with~\eqref{Reverse eq.1}, we obtain that
\begin{align}\label{Reverse eq.10}
&\sup_{t \in I^\lambda_{r_1}}\dashint_{B_{r_1}}\frac{\bigl|u(t)-(u)_{Q_{r_1}^\lambda}\bigr|^2}{\lambda^{2-p}r_1^2}\d{x}+\dashint_{Q_{r_1}^\lambda}|\nabla u|^p\d{x}\d{t} \leq \bm{\cA}_1+\bm{\cA}_2+\bm{\cA}_3+\bm{\cB}\,,
\end{align}
where we denoted, for short, for $C_5, C_6(\data, K) <\infty$, 
\begin{equation*}
\left\{
\begin{aligned}
& \bm{\cA}_1:=C_6\mathcal{R}_{r_1,r_2}^2 \dashint_{Q_{r_2}^\lambda}\frac{\bigl|u-(u)_{Q_{r_2}^\lambda}\bigr|^p}{r_2^p}\d{x}\d{t}\,, \\[2mm]
& \bm{\cA}_2:=C_6\delta_0^{1-p}\mathcal{R}_{r_1,r_2}^{p(d+sp)}\dashint_{Q_{r_2}^\lambda}\frac{\bigl|u-(u)_{Q_{r_2}^\lambda}\bigr|^p}{r_2^p}\d{x}\d{t}\,,\\[2mm]
& \bm{\cA}_3:=C_6\lambda^{p-2}\mathcal{R}_{r_1,r_2}^2\dashint_{Q_{r_2}^\lambda}\frac{\bigl|u-(u)_{Q_{r_2}^\lambda}\bigr|^2}{r_2^2}\d{x}\d{t}\,, \quad \mbox{and} 
\\[2mm]
& \bm{\cB}:=\frac{\delta_0}{2} \left(\rho^{\frac{(1-s)p}{p-1}}\dashint_{I_{\rho}^\lambda}\int_{\R^d \setminus B_{\rho/2}}\frac{\bigl|u(y,t)-(u)_{Q_{\rho}^\lambda}\bigr|^{p}}{|y|^{d+sp}}\d{y}\d{t}\right)+\frac{\delta_0}{4}\dashint_{Q_{\rho/2}^\lambda } |\nabla u|^p\d{x}\d{t}\\
& \quad \quad  +C_5\left(\dashint_{Q_\rho^\lambda }|\nabla u|^{\kappa p}\d{x}\d{t}\right)^{\nicefrac{1}{\kappa}}+C_5\dashint_{Q_\rho^\lambda}\bigl(|F|^p+1\bigr)\d{x}\d{t} \\
&\quad \quad +C_5\lambda^{p(2-p)} \left(\rho^\theta \dashint_{Q_\rho^\lambda}\cE[u]\d{x}\d{t}\right)^{p-1}\,.
\end{aligned}
\right.
\end{equation*}
\medskip

\emph{Step 2: Estimate of $\bm{\cA}_1$--$\bm{\cA}_3$ in~\eqref{Reverse eq.10}.} Appealing to Lemma~\ref{SP2} and using $\cR_{r_1,r_2} \geq 1$, we deduce that, for any $\delta_0 \in (0,1]$ and $\eps \in (0,1]$,
\begin{align}\label{Reverse eq.11}
\bm{\cA}_1+\bm{\cA}_2  &\leq C_6\delta_0^{1-p}\left(\mathcal{R}_{r_1,r_2}^{2+p}+\mathcal{R}_{r_1,r_2}^{p(d+sp)}\right)\dashint_{Q_{r_2}^\lambda}\left|\frac{u-(u)_{Q_{r_2}^\lambda}}{r_2}\right|^p\d{x}\d{t} \notag\\
\quad &\leq  C_6\delta_0^{1-p}\mathcal{R}_{r_1,r_2}^{p(d+p)}\eps \left(\sup_{t \in I_{r_2}^\lambda}\dashint_{B_{r_2}}\frac{\bigl|u(t)-(u)_{Q_{r_2}^\lambda}\bigr|^2}{\lambda^{2-p} r_2^2}\d{x} +\dashint_{Q_{r_2}^\lambda}|\nabla u|^p\d{x}\d{t}\right) \notag\\
\quad &\quad +C_7\delta_0^{1-p}\mathcal{R}_{r_1,r_2}^{p(d+p)}\eps^{-\bm{\gamma}_p}\left[\left(\dashint_{Q_{r_2}^\lambda}|\nabla u|^{\kappa p}\d{x}\d{t}\right)^{\nicefrac{1}{\kappa}}+\bm{\mathfrak{T}}_{p,\lambda}(r_2)\left(r_2^\theta \dashint_{Q_{r_2}^\lambda}\cE[u]\d{x}\d{t}\right) \right.\notag\\
&\quad \quad \quad \quad \quad \quad \quad \quad \quad \quad \quad \quad \quad \quad \quad \quad \quad \quad \quad \quad  \,\,\,\left.+\dashint_{Q_{r_2}^\lambda}|F|^p\d{x}\d{t}\right]\,,
\end{align}
where $C_7(\data, K)<\infty$, and $\bm{\mathfrak{T}}_{p, \lambda}(r)$ is given by
\[
\bm{\mathfrak{T}}_{p, \lambda}(r):=
\left\{
\begin{aligned}
&\frac{1}{\lambda^{p \sigma_0}}\left(r^\theta \dashint_{Q_{r}^\lambda}\cE[u]\d{x}\d{t}\right)^{\sigma_0}  &\mbox{if} \quad & p \geq 2\,, \\
&1  &\mbox{if} \quad & 1<p<2\,,
\end{aligned}
\right.
\]
with $\sigma_0:=\frac{(p-2)(1+\kappa)}{2-(1-\kappa)p}$ for short. By definition, it is straightforward to check that
\begin{equation}\label{e.reverse T}
\bm{\mathfrak{T}}_{p, \lambda}(r_2) \leq 2^{(d+2)\sigma_0}\bm{\mathfrak{T}}_{p, \lambda}(\rho)\,,
\end{equation}
and thus
\begin{equation}\label{e.reverse E}
\bm{\mathfrak{T}}_{p, \lambda}(r_2) \left(r_2^\theta \dashint_{Q_{r_2}^\lambda}\cE[u]\d{x}\d{t}\right) \leq 2^{(d+2)(\sigma_0+1)}\bm{\mathfrak{T}}_{p, \lambda}(\rho)\left(\rho^\theta \dashint_{Q_{\rho}^\lambda}\cE[u]\d{x}\d{t}\right) \,.
\end{equation}
On the other hand, Lemma~\ref{SP3} yields that 
\begin{align}\label{Reverse eq.12}
\bm{\cA}_3  & \leq C_8\mathcal{R}_{r_1,r_2}^2\eps \left(\sup_{t \in I_{r_2}^\lambda}\dashint_{B_{r_2}}\frac{\bigl|u(t)-(u)_{Q_{r_2}^\lambda}\bigr|^2}{\lambda^{2-p} r_2^2}\d{x} +\dashint_{Q_{r_2}^\lambda}|\nabla u|^p\d{x}\d{t}\right)\notag\\
&\quad \quad +C_8\mathcal{R}_{r_1,r_2}^2\eps^{-\left(\frac{p}{2} \vee \frac{1}{p-1}\right)}\left[\left(\dashint_{Q_{r_2}^\lambda}|\nabla u|^{\kappa p}\d{x}\d{t}\right)^{\nicefrac{1}{\kappa}}+\bm{\mathfrak{T}}_{p,\lambda}(r_2)\left(r_2^\theta \dashint_{Q_{r_2}^\lambda}\cE[u]\d{x}\d{t}\right) \right. \notag\\
&\quad \quad \quad \quad \quad \quad \quad \quad \quad \quad \quad \quad \quad \quad \quad \quad \quad \quad \quad \quad  \,\,\,\left.+\dashint_{Q_{r_2}^\lambda}|F|^p\d{x}\d{t}\right]\,.
\end{align}
for a constant $C_8(\data,K)<\infty$. Combining~\eqref{Reverse eq.11},~\eqref{e.reverse E}~and~\eqref{Reverse eq.12} with~\eqref{Reverse eq.10} and rearranging, we obtain, for a constant $C_{9}(\data,K)<\infty$, 
\begin{align*}
\sup_{t \in I^\lambda_{r_1}}&\dashint_{B_{r_1}}\frac{\bigl|u(t)-(u)_{Q_{r_1}^\lambda}\bigr|^2}{\lambda^{2-p}r_1^2}\d{x}+\dashint_{Q_{r_1}^\lambda}|\nabla u|^p\d{x}\d{t} \\[2mm]
&\leq C_{9}\delta_0^{1-p}\mathcal{R}_{r_1,r_2}^{p(d+p)}\eps\left(\sup_{t \in I_{r_2}^\lambda}\dashint_{B_{r_2}}\frac{\bigl|u(t)-(u)_{Q_{r_2}^\lambda}\bigr|^2}{\lambda^{2-p} r_2^2}\d{x}+\dashint_{Q_{r_2}^\lambda}|\nabla u|^p\d{x}\d{t}\right) \notag \\[2mm]
&\quad \quad +C_{9}\delta_0^{1-p}\mathcal{R}_{r_1,r_2}^{p(d+p)}\eps^{-\bm{\beta}_p}\bm{\cE}+\bm{\cB}\,,
\end{align*}
where $\bm{\beta}_p:=\max\{\bm{\gamma}_p, \frac{p}{2}, \frac{1}{p-1}\}>0$ and, to shorten the notation, we wrote
\[
\bm{\cE}:=\left(\dashint_{Q_{\rho}^\lambda}|\nabla u|^{\kappa p}\d{x}\d{t}\right)^{\nicefrac{1}{\kappa}}+\bm{\mathfrak{T}}_{p,\lambda}(\rho)\left(\rho^\theta \dashint_{Q_{\rho}^\lambda}\cE[u]\d{x}\d{t}\right)+\dashint_{Q_{\rho}^\lambda}\bigl(|F|^p+1\bigr)\d{x}\d{t}\,.
\]
We choose $\eps \in (0,1]$ so small that
\[
C_{9}\delta_0^{1-p}\mathcal{R}_{r_1,r_2}^{p(d+p)}\eps = \frac{1}{2} \quad \iff \quad \eps =\frac{\delta_0^{p-1}}{2C_{9}\mathcal{R}_{r_1,r_2}^{p(d+p)}}\,.
\]
Thus, the previous display implies that, for a constant $C_{10}(\data, K)<\infty$,
%
\begin{align*}
\sup_{t \in I^\lambda_{r_1}}&\dashint_{B_{r_1}}\frac{\bigl|u(t)-(u)_{Q_{r_1}^\lambda}\bigr|^2}{\lambda^{2-p}r_1^2}\d{x}+\dashint_{Q_{r_1}^\lambda}|\nabla u|^p\d{x}\d{t} \\[2mm]
&\leq \frac{1}{2}\left(\sup_{t \in I_{r_2}^\lambda}\dashint_{B_{r_2}}\frac{\bigl|u(t)-(u)_{Q_{r_2}^\lambda}\bigr|^2}{\lambda^{2-p} r_2^2}\d{x}+\dashint_{Q_{r_2}^\lambda}|\nabla u|^p\d{x}\d{t}\right) \notag \\[2mm]
&\quad \quad +\frac{C_{10}\delta_0^{(1-p)(1+\bm{\beta}_p)}\rho^{p(d+p)(1+\bm{\beta}_p)}}{(r_2-r_1)^{p(d+p)(1+\bm{\beta}_p)}}\bm{\cE}+\bm{\cB}\,,
\end{align*}
therefore appealing to Lemma~\ref{t.iteration}, we deduce that,
\begin{align*}
\sup_{t \in I^\lambda_{\rho/2}}&\dashint_{B_{\rho/2}}\frac{\bigl|u(t)-(u)_{Q_{\rho/2}^\lambda}\bigr|^2}{\lambda^{2-p}(\rho/2)^2}\d{x}+\dashint_{Q_{\rho/2}^\lambda}|\nabla u|^p\d{x}\d{t} \\
&\leq C_{11}\delta^{-\bm{\alpha}_p}\bm{\cE}+C_{12}\bm{\cB}\\[2mm]
&=C_{11}\delta^{-\bm{\alpha}_p}\bm{\cE}+\frac{C_{12}\delta_0}{4}\dashint_{Q_{\rho/2}^\lambda } |\nabla u|^p\d{x}\d{t} \\
&\quad \quad +\frac{C_{11}\delta_0}{2} \left(\rho^{\frac{(1-s)p}{p-1}}\dashint_{I_{\rho}^\lambda}\int_{\R^d \setminus B_{\rho/2}}\frac{\bigl|u(y,t)-(u)_{Q_{\rho}^\lambda}\bigr|^{p}}{|y|^{d+sp}}\d{y}\d{t}\right)\\
& \quad \quad  +C_{12}\underbrace{\left[\left(\dashint_{Q_\rho^\lambda }|\nabla u|^{\kappa p}\d{x}\d{t}\right)^{\nicefrac{1}{\kappa}} +\dashint_{Q_\rho^\lambda}\bigl(|F|^p+1\bigr)\d{x}\d{t}\right]}_{\leq \bm{\cE}} \\
&\quad \quad +C_{12}\lambda^{p(2-p)} \left(\rho^\theta \dashint_{Q_\rho^\lambda}\cE[u]\d{x}\d{t}\right)^{p-1}
\end{align*}
for constants $C_{11}, C_{12}(\data, K)<\infty$ because $\bm{\beta}_p$ depends only on $p$, where $\bm{\alpha}_p:=(p-1)(1+\bm{\beta}_p)>0$. 
Selecting $\delta_0$ so small that $C_{11}\delta_0/4 =1/2$ and reabsorbing then leads to
\begin{align}\label{Reverse eq.13}
\dashint_{Q_{\rho/2}^\lambda}|\nabla u|^p\d{x}\d{t} &\leq C_{13}\bm{\cE}+C_{13}\left(\rho^{\frac{(1-s)p}{p-1}}\dashint_{I_{\rho}^\lambda}\int_{\R^d \setminus B_{\rho/2}}\frac{\bigl|u(x,t)-(u)_{Q_{\rho}^\lambda}\bigr|^{p}}{|x|^{d+sp}}\d{x}\d{t}\right) \notag\\
&\quad \quad \quad  +C_{13}\lambda^{p(2-p)} \left(\rho^\theta \dashint_{Q_\rho^\lambda}\cE[u]\d{x}\d{t}\right)^{p-1}\,.
\end{align}
\smallskip

\emph{Step 3: The conclusion.} In the case $\frac{2d}{d+2}<p<2$, we estimate the last integral on the right side of~\eqref{Reverse eq.13}. By H\"{o}lder's inequality with $\left(\frac{1}{2-p}, \frac{1}{p-1}\right)$ and~\eqref{intrinsic 1}$_2$, we find that
\[
C_{13}\lambda^{p(2-p)} \left(\rho^\theta \dashint_{Q_\rho^\lambda}\cE[u]\d{x}\d{t}\right)^{p-1} \leq \frac{1}{2}\dashint_{Q_{\rho/2}^\lambda}\bigl(|\nabla u|^p+|F|^p\bigr)\d{x}\d{t}+C\rho^\theta \dashint_{Q_\rho^\lambda}\cE[u]\d{x}\d{t}\,.
\]
After plugging this into the previous display~\eqref{Reverse eq.13}, reabsorbing and reorganizing then yields that, for a constant $C(\data, K)<\infty$,
\begin{align}\label{Reverse eq.14}
\dashint_{Q_{\rho/2}^\lambda}|\nabla u|^p\d{x}\d{t} &\leq C\bm{\cE}+C\left(\rho^{\frac{(1-s)p}{p-1}}\dashint_{I_{\rho}^\lambda}\int_{\R^d \setminus B_{\rho/2}}\frac{\bigl|u(x,t)-(u)_{Q_{\rho}^\lambda}\bigr|^{p}}{|x|^{d+sp}}\d{x}\d{t}\right) \notag\\
&\quad \quad \quad  +C\rho^\theta \dashint_{Q_\rho^\lambda}\cE[u]\d{x}\d{t}\,.
\end{align}
This is the desired result in the case $\frac{2d}{d+2}<p<2$.

Finally, define
\[
\bm{\mathfrak{U}}_{p, \lambda}(\rho):=
\left\{
\begin{aligned}
&\frac{1}{\lambda^{p \sigma_0}}\left(\rho^\theta \dashint_{Q_{\rho}^\lambda}\cE[u]\d{x}\d{t}\right)^{\sigma_0}  + \frac{1}{\lambda^{p(p-2)}}\left(\rho^\theta \dashint_{Q_\rho^\lambda}\cE[u]\d{x}\d{t}\right)^{p-2}&\mbox{if} \quad & p \geq 2\,, \\
&1  &\mbox{if} \quad & \tfrac{2d}{d+2}<p<2\,,
\end{aligned}
\right.
\]
with $\sigma_0:=\frac{(p-2)(1+\kappa)}{2-(1-\kappa)p}$. Consequently, unifying~\eqref{Reverse eq.13} and~\eqref{Reverse eq.14} and reorganizing complete the proof of Lemma~\ref{t.Reverse Holder}.
\end{proof}

Finally, we state a ``$p$-tail''-controlled inequality, which plays a crucial role in Section~\ref{Sect.5}.
\begin{lemma}[$p$-tail-controlled inequality]\label{t.p-tail-controlled inequality}
Fix $p>\frac{2d}{d+2}$ and $ s \in (0,1)$ and set $\theta:=\frac{(1-s)p}{p-1}$. Let $u$ be a weak solution to~\eqref{maineq} in the sense of Definition~\ref{dfn u}, under the assumptions~(A1)--(A4) and $F \in L^p_{\mathrm{loc}}\left(\Omega_T\,;\R^d\right)$. Then for every space-time cylinder $Q_{\rho}^\lambda(z_0) \Subset \Omega_T$ with $0<\rho \leq 1$ satisfying $\lambda$-intrinsic geometry conditions~\eqref{intrinsic 1}--\eqref{intrinsic 2}, there exists a constant $C_{\TAIL}(\data,K)<\infty$ such that
\begin{align}\label{e.p-tail-controlled inequality}
\rho^{\theta}\dashint_{I_{\rho}^\lambda(t_0)}&\int_{\R^d \setminus B_{\rho/2}(x_0)}\frac{\big|u(x,t)-(u)_{Q_{\rho}^\lambda(z_0)}\big|^{p}}{|x-x_0|^{d+sp}}\d{x}\d{t} \\\notag
&\leq C_{\TAIL}\rho^{\theta p}\lambda^p+C_{\TAIL}(1+\bm{\mathfrak{S}}_{p, \lambda}(\rho))\left(\rho^\theta \dashint_{Q_\rho^\lambda(z_0)}\cE[u](x,t)\d{x}\d{t}\right)\,,
\end{align}
where we denoted
\[
\cE[u](x,t):=\int_{\R^d} \frac{\big|u(y,t)-u(x,t)\big|^p}{|y-x|^{d+sp}}\d{y}
\]
and
\[
\bm{\mathfrak{S}}_{p, \lambda}(\rho):=
\left\{
\begin{aligned}
&\frac{1}{\lambda^{p(p-2)}}\left(\rho^\theta \dashint_{Q_\rho^\lambda}\cE[u]\d{x}\d{t}\right)^{p-2}&\mbox{if} \quad & p \geq 2\,, \\
&1  &\mbox{if} \quad & \tfrac{2d}{d+2}<p<2\,.
\end{aligned}
\right.
\]
\end{lemma}

\begin{proof}
By translation, we may let $z_0=0$ without loss of generality. For short, we set
\[
\mathbf{T}^p_\rho:=\rho^{\theta}\dashint_{I_{\rho}^\lambda}\int_{\R^d \setminus B_{\rho/2}}\frac{\big|u(y,t)-(u)_{Q_{\rho}^\lambda}\big|^{p}}{|y|^{d+sp}}\d{y}\d{t}\,.
\]
A straightforward computation yields that
\[
\mathbf{T}_\rho^p \leq C\left(\mathbf{S}_1+\mathbf{S}_2+\mathbf{S}_3\right)\,,
\]
where
\[
\left\{
\begin{aligned}
\mathbf{S}_1&:=\rho^{\theta}\dashint_{I_{\rho}^\lambda}\int_{B_{\rho} \setminus B_{\rho/2}}\frac{\big|u(y,t)-(u(t))_{B_{\bar{\rho}}}\big|^{p}}{|y|^{d+sp}}\d{y}\d{t}\,, \\
\mathbf{S}_2&:=\rho^{\theta}\dashint_{I_{\rho}^\lambda}\int_{\R^d \setminus B_{\rho}}\frac{\big|u(y,t)-(u(t))_{B_{\bar{\rho}}}\big|^{p}}{|y|^{d+sp}}\d{y}\d{t}\,, \quad \mbox{and} \quad\\
\mathbf{S}_3&:=\rho^{\theta}\dashint_{I_{\rho}^\lambda}\big|(u(t))_{B_{\bar{\rho}}}-(u)_{Q^\lambda_\rho}\big|^p\d{t} \int_{\R^d \setminus B_{\rho/2}} \frac{\d{y}}{|y|^{d+sp}}\,.
\end{aligned}
\right.
\]
Here, the intermediate radius $\bar{\rho} \in [\nicefrac{\rho}{2},\nicefrac{3\rho}{4}]$ is as in Lemma~\ref{t.gluing}. For $\mathbf{S}_1$, we have, using Lemma~\ref{t.useful lemma}, the Poincar\'{e} inequality and~\eqref{intrinsic 1}, 
\begin{align*}
\mathbf{S}_1 &\leq C\rho^{\theta-sp}\dashint_{I_\rho^\lambda}\dashint_{B_{\rho}}\big|u(y,t)-(u(t))_{B_{\bar{\rho}}}\big|^{p}\d{y}\d{t}\\
&\leq C\rho^{\theta-sp}\dashint_{I_\rho^\lambda}\dashint_{B_{\rho}}\big|u(y,t)-(u(t))_{B_{\rho}}\big|^{p}\d{y}\d{t}\\
&\leq C\rho^{\theta+(1-s)p}\dashint_{Q_{\rho}^\lambda}|\nabla u|^p\d{y}\d{t}\stackrel{\eqref{intrinsic 1}}{\leq}C\rho^{\theta p}\lambda^p
\end{align*}
for a constant $C(\data, K)<\infty$, where to obtain the last line we used $\theta+(1-s)p=\theta p$.

We turn to estimate $\mathbf{S}_2$. Appealing to the H\"{o}lder inequality yields that
\[
\big|u(y,t)-(u(t))_{B_{\bar{\rho}}}\big|^{p} \leq \dashint_{B_{\bar{\rho}}}\big|u(y,t)-u(x,t)\big|^p\d{x}
\]
and for every $y \in \R^d \setminus B_{\rho}$ and $x \in B_{\bar{\rho}}$, we have
$|y-x| <2|y|$, and combining these estimations with Fubini's theorem gives that
\begin{align*}
\mathbf{S}_2 &\leq \rho^\theta\dashint_{I_{\rho}^\lambda}\int_{\R^d \setminus B_{\rho}} \dashint_{B_{\bar{\rho}}}\frac{\big|u(y,t)-u(x,t)\big|^{p}}{|y|^{d+sp}} \d{x}\d{y}\d{t} \\
&\leq 2^{d+sp}\rho^\theta\dashint_{I_{\rho}^\lambda} \dashint_{B_{\bar{\rho}}}\left(\int_{\R^d \setminus B_{\rho}} \frac{\big|u(y,t)-u(x,t)\big|^{p}}{|y-x|^{d+sp}}\d{y}\right) \d{x}\d{t}\\
&\leq C\rho^\theta \dashint_{Q_\rho^\lambda}\cE[u](x,t)\d{x}\d{t},
\end{align*}
where $\displaystyle \cE[u](x,t):=\int_{\R^d}\frac{\big|u(y,t)-u(x,t)\big|^{p}}{|y-x|^{d+sp}}\d{y}$. Using~\eqref{formula}, we have
\begin{align*}
\mathbf{S}_3 &\leq C\rho^{\theta-sp}\dashint_{I_\rho^\lambda}\big|(u(t))_{B_{\bar{\rho}}}-(u)_{Q_\rho^\lambda}\big|^p\d{t}\\
&\leq C\rho^{\theta-sp} \left(\dashint_{I_\rho^\lambda} \left|(u(t))_{B_{\bar{\rho}}}-\dashint_{I_\rho^\lambda}(u(\tau))_{B_{\bar{\rho}}}\d{\tau}\right|^p\d{t}+ \left|\dashint_{I_\rho^\lambda}(u(\tau))_{B_{\bar{\rho}}}\d{\tau}-(u)_{Q_\rho^\lambda}\right|^p\right).
\end{align*}
In view of Lemma~\ref{t.gluing}, H\"{o}lder's inequality and~\eqref{intrinsic 2}, we have, for a $C(\data,K)<\infty$, 
\begin{align*}
\dashint_{I_\rho^\lambda} &\left|(u(t))_{B_{\bar{\rho}}}-\dashint_{I_\rho^\lambda}(u(\tau))_{B_{\bar{\rho}}}\d{\tau}\right|^p\d{t} \\
&\leq \dashint_{I_\rho^\lambda}  \dashint_{I_\rho^\lambda} \left|(u(t))_{B_{\bar{\rho}}}-(u(\tau))_{B_{\bar{\rho}}}\right|^p \d{\tau}\d{t}\\
&\leq \dashint_{I_\rho^\lambda}  \dashint_{I_\rho^\lambda} \left[C\lambda^{2-p} \rho \dashint_{Q_\rho^\lambda} \big(|\nabla u|^{p-1}+|F|^{p-1}\big)\d{x}\d{\tilde{t}}\right]^p\d{\tau}\d{t} \\
&\quad \quad +\dashint_{I_\rho^\lambda}  \dashint_{I_\rho^\lambda} \left[C\lambda^{2-p} \rho \left(\rho^\theta \dashint_{Q_\rho^\lambda}\cE[u]\d{x}\d{\tilde{t}}\right)^{\nicefrac{(p-1)}{p}}\right]^{p}\d{\tau}\d{t} \\
&\leq C\lambda^{(2-p)p}\rho^p \left(\dashint_{Q_\rho^\lambda} \big(|\nabla u|^{p}+|F|^{p}\big)\d{x}\d{\tilde{t}}\right)^{p-1}\\
&\quad \quad +C\lambda^{(2-p)p}\rho^p \left(\rho^\theta \dashint_{Q_\rho^\lambda}\cE[u]\d{x}\d{\tilde{t}}\right)^{p-1}\\
&\!\!\!\stackrel{\eqref{intrinsic 2}}{\leq} C\lambda^p\rho^p+C\lambda^{(2-p)p}\rho^p \left(\rho^\theta \dashint_{Q_\rho^\lambda}\cE[u]\d{x}\d{\tilde{t}}\right)^{p-1}\,.
\end{align*}
In the case $1<p<2$, the H\"{o}lder inequality with $\left(\frac{1}{2-p},\,\frac{1}{p-1}\right)$ implies that
\[
C\lambda^{(2-p)p}\rho^p \left(\rho^\theta \dashint_{Q_\rho^\lambda}\cE[u]\d{x}\d{\tilde{t}}\right)^{p-1} \leq C \lambda^p \rho^{p}+C\rho^{p} \left(\rho^\theta \dashint_{Q_\rho^\lambda}\cE[u]\d{x}\d{\tilde{t}}\right)\,.
\]
Thus, defining
\[
\bm{\mathfrak{S}}_{p, \lambda}(\rho):=
\left\{
\begin{aligned}
&\frac{1}{\lambda^{p(p-2)}}\left(\rho^\theta \dashint_{Q_\rho^\lambda}\cE[u]\d{x}\d{t}\right)^{p-2}&\mbox{if} \quad & p \geq 2\,, \\
&1  &\mbox{if} \quad & \tfrac{2d}{d+2}<p<2\,,
\end{aligned}
\right.
\]
we can obtain unifiedly that
\[
\dashint_{I_\rho^\lambda} \left|(u(t))_{B_{\bar{\rho}}}-\dashint_{I_\rho^\lambda}(u(\tau))_{B_{\bar{\rho}}}\d{\tau}\right|^p\d{t} \leq C\lambda^p \rho^p+\bm{\mathfrak{S}}_{p, \lambda}(\rho) \rho^p\left(\rho^\theta \dashint_{Q_\rho^\lambda}\cE[u]\d{x}\d{t}\right)\,.
\]
Next, observe that
\[
\dashint_{I_\rho^\lambda}(u(\tau))_{B_{\bar{\rho}}}\d{\tau}-(u)_{Q_\rho^\lambda} = \dashint_{I_\rho^\lambda} \dashint_{B_{\bar{\rho}}} \big(u(y,\tau)-(u(\tau))_{B_\rho}\big)\d{y}\d{\tau},
\]
and thus Lemma~\ref{t.useful lemma},~\eqref{intrinsic 2} and the H\"{o}lder and Poincar\'{e} inequalities imply that
\begin{align*}
\left|\dashint_{I_\rho^\lambda}(u(\tau))_{B_{\bar{\rho}}}\d{\tau}-(u)_{Q_\rho^\lambda}\right|^p &\leq \dashint_{I_\rho^\lambda} \dashint_{B_{\bar{\rho}}}\big|u(y,\tau)-(u(\tau))_{B_\rho}\big|^p\d{y}\d{\tau} \\
&\leq C(d) \dashint_{I_\rho^\lambda}\dashint_{B_{\rho}}\rho^p|\nabla u|^p\d{y}\d{\tau}\\
&\!\!\!\stackrel{\eqref{intrinsic 2}}{\leq} C\rho^p\lambda^p
\end{align*}
for a constant $C(d,K)<\infty$. Hence, combining these yields that
\[
\mathbf{S}_3 \leq C\rho^{\theta-sp} \left[\rho^p\lambda^p+\bm{\mathfrak{S}}_{p, \lambda}(\rho) \rho^p\left(\rho^\theta \dashint_{Q_\rho^\lambda}\cE[u]\d{x}\d{t}\right)\right]\,. 
\]
Combining the previous estimates of $\mathbf{S}_1$-$\mathbf{S}_3$ yields~\eqref{e.p-tail-controlled inequality}.
\end{proof}

\section{Proof of Theorem~\ref{t.gradhigher}}\label{Sect.5}
In this final section, we provide the proof of Theorem~\ref{t.gradhigher}. Since the proof is technically involved, we split it into four subsections below, with the main arguments are in~Sections~\ref{Sect.5.2} and~\ref{Sect.5.3}. The proof begins with a ``stopping time argument'', as in, for instance~\cite{KiLe00, AM07}. Unlike in the purely local setting, however, the nonlocal quantity that naturally arises in the previous estimates cannot be controlled by the stopping time argument alone. To overcome this difficulty, we employ a modification of the intrinsic Calder\'{o}n-Zygmund-type decomposition in~\cite[Section 9]{Min07} and also~\cite{CP98, AD16}.

\subsection{Setup and preliminary observations}\label{Sect.5.1}
Before going on, let us make a few preliminary observations.

Denote $\theta:=\frac{(1-s)p}{p-1}$ for short. Fix $\lambda_0 \geq 1$ to be specified later. Let $R \in (0,1]$ satisfy $R \leq R_0$ where $R_0$ is to be small and specified later again. For $\lambda \geq \lambda_0$ and $R/2<r_1\leq r_2 \leq R$, we define
\[
R_\lambda:=\frac{r_2-r_1}{10^{10d}}(1 \wedge \lambda^{\frac{p-2}{2}})\,.
\]
We then have the following simple geometry:

\begin{lemma}\label{t.inclusion cylinder}
For every $\widetilde{z}\in Q_{r_1} \equiv B_{r_1}(x_0) \times (t_0-r_1^2, t_0+r_1^2)$ and $ \rho \in (0,50R_\lambda]$ with $\lambda \geq \lambda_0$, we have
\begin{equation}\label{e.inclusion cylinder}
Q_\rho^\lambda(\widetilde{z}) \subset Q_{r_2}\,.
\end{equation}
\end{lemma}

\begin{proof}[Proof of Lemma~\ref{t.inclusion cylinder}]
Take $z=(x,t) \in Q_\rho^\lambda(\widetilde{z}) $ arbitrarily. For the spatial direction, it is straightforward to check that
\[
|x-x_0| \leq |x-\tilde{x}|+|\tilde{x}-x_0| \leq r_1+\rho \leq r_1+50R_\lambda <r_2\,.
\]
On the other hand,
\[
|t-t_0| \leq |t-\tilde{t}|+|\tilde{t}-t_0| \leq 2\lambda^{2-p}\rho^2+r_1^2<5000\lambda^{2-p}R_\lambda^2+r_1^2\,.
\]
It suffices to check this in the case that
\[
5000\lambda^{2-p}R_\lambda^2+r_1^2<r_2^2\,.
\]
Indeed, since $\lambda^{2-p}(1\wedge \lambda^{p-2}) \leq 1$ for every $p >\frac{2d}{d+2}$ and $\lambda \geq \lambda_0$, we have
\begin{align*}
r_2^2-\left(5000\lambda^{2-p}R_\lambda^2+r_1^2\right)&=r_2^2-r_1^2-\frac{5000(r_2-r_1)^2}{10^{20d}}\lambda^{2-p}(1\wedge \lambda^{p-2})\\
&\geq (r_2-r_1)\left(r_2+r_1-\frac{5000}{10^{20d}}(r_2-r_1)\right)>0\,.
\end{align*}
Combining these observations,~\eqref{e.inclusion cylinder} immediately follows.
\end{proof}

Next, define
\[
\displaystyle \mathbf{E}_r:=\int_{Q_r}\big(|\nabla u|^p+|F|^p\big)\d{x}\d{t}\,. 
\]
In view of Lemma~\ref{t.inclusion cylinder}, it is straightforward to check, for every $\rho \in [R_\lambda/50,\,R_\lambda]$, that
\begin{equation*}
\dashint_{Q_\rho^\lambda(\tilde{z})} \bigl(|\nabla u|^p+|F|^p\bigr)\d{x}\d{t} \leq \frac{\lambda^p}{\omega_d\left(\frac{r_2-r_1}{50\cdot 10^{10d}}\right)^{d+2}\lambda^{\gamma_p}} \mathbf{E}_{r_2}\,.
\end{equation*}
Denote $H(d):=2(50\cdot 10^{10d})^{d+2} / \omega_d$ for short. Take $\eta \in (0,1)$ to satisfy $\eta \leq \eta_0$, where $\eta_0 \in (0,1)$ to be specified later in~\eqref{e.pfgradhigher13}. We now select $\lambda_0$ to satisfy
\begin{align}\label{e.pfgradhigher1}
\lambda_0 =\eta^{-\nicefrac{1}{p}}&+\left[\frac{H(d)}{(r_2-r_1)^{d+2}}\mathbf{E}_{r_2}\right]^{\nicefrac{1}{\gamma_p}} \notag\\
&+\indc_{\{\theta \geq d+2\}}\left[\frac{(r_2-r_1)^{\theta-d-2}}{\eta}\int_{Q_{r_2}}\cE[u]\d{x}\d{t}\right]^{{\frac{2 \indc_{\{p<2\}}}{4+(2-p)_+(\theta-d-2)}}}\,,
\end{align}
obtaining
\[
\omega_d\left(\frac{r_2-r_1}{50\cdot 10^{10d}}\right)^{d+2}\lambda_0^{\gamma_p} \geq 2\mathbf{E}_{r_2}\,,
\]
where $(2-p)_+:=(2-p) \vee 0$. All in all, for every $\lambda \geq \lambda_0$, $\rho \in [R_\lambda/50,\,R_\lambda]$ and $\tilde{z} \in Q_{r_1}$, we deduce that
\begin{equation}\label{e.pfgradhigher2}
\dashint_{Q^\lambda_\rho(\tilde{z})}\bigl(|\nabla u|^p+|F|^p\bigr)\d{x}\d{t}  \leq \frac{1}{2}\lambda_0^p \leq \frac{1}{2}\lambda^p\,.
\end{equation}
\subsection{Stopping time argument}\label{Sect.5.2}
For $\lambda \geq \lambda_0$ and $r \in (0,R)$, let $\mathbf{\Xi}(r, \lambda)$ be the set of Lebesgue points and super level set of $|\nabla u|$, that is,
%
\[
\mathbf{\Xi}(r, \lambda):=\Biggl\{ \tilde{z} \in Q_r :  \lim_{\rho \searrow  0}\left\|\nabla u- \nabla u (\widetilde{z})\right\|_{\underline{L}^1(Q_\rho^\lambda(\widetilde{z}))}=0 \quad \mbox{and} \quad |\nabla u|(\tilde{z}) >\lambda \Biggr\}\,.
\]
We then find that $\bigl|Q_r \cap \{|\nabla u|>\lambda \} \setminus \mathbf{\Xi}(r, \lambda)\bigr|=0$. For each $\lambda>\lambda_0$ and $\tilde{z} \in \mathbf{\Xi}(r_1, \lambda)$, there exists $\rho_{\tilde{z}}^\ast  \in (0,\,R_\lambda/50]$ such that
\begin{equation}\label{e.pfgradhigher2'}
\left\{
\begin{aligned}
& 
\dashint_{Q^\lambda_{\rho_{\tilde{z}}^\ast}(\tilde{z})}\bigl(|\nabla u|^p+|F|^p\bigr)\d{x}\d{t}=\lambda^p\,,
\quad \mbox{while} 
\\ & 
\dashint_{Q^\lambda_{\rho}(\tilde{z})}\bigl(|\nabla u|^p+|F|^p\bigr)\d{x}\d{t} < \lambda^p\,,\quad \forall \rho \in (\rho_{\tilde{z}}^\ast,\,R_\lambda]\,.
\end{aligned}
\right.
\end{equation}
Indeed, the existence of such $\rho_{\tilde{z}}^\ast$ is clear since $\tilde{z} \in \mathbf{\Xi}(r_1, \lambda)$, which gives, by H\"{o}lder's inequality,
%
%
\begin{align*}
\liminf_{\rho \searrow 0}\dashint_{Q^\lambda_\rho(\tilde{z})}\bigl(|\nabla u|^p+|F|^p\bigr)\d{x}\d{t} &\geq  \liminf_{\rho \searrow 0}\left(\dashint_{Q^\lambda_\rho(\tilde{z})}|\nabla u|\d{x}\d{t} \right)^p\\
&=|\nabla u (\tilde{z})|^p >\lambda^p\,.
\end{align*}
The mapping $\rho \mapsto  \dashint_{Q^\lambda_\rho(\tilde{z})}\bigl(|\nabla u|^p+|F|^p\bigr)\d{x}\d{t}$ is continuous, and by~\eqref{e.pfgradhigher2}
%
\[
\dashint_{Q^\lambda_{R_\lambda/50}(\tilde{z})}\bigl(|\nabla u|^p+|F|^p\bigr)\d{x}\d{t} < \lambda^p\,,
\]
giving~\eqref{e.pfgradhigher2'}. Hence, by~\eqref{e.pfgradhigher2'}, we obtain in particular the ``stopping property''
\begin{equation}\label{e.pfgradhigher3}
\left\{
\begin{aligned}
& 
\dashint_{Q^\lambda_{\rho_{\tilde{z}}^\ast}(\tilde{z})}\bigl(|\nabla u|^p+|F|^p\bigr)\d{x}\d{t}=\lambda^p
\,, \quad \mbox{and} 
\\ & 
\dashint_{2aQ^\lambda_{\rho_{\tilde{z}}^\ast}(\tilde{z})}\bigl(|\nabla u|^p+|F|^p\bigr)\d{x}\d{t} < \lambda^p\,, \quad a \in \{1,\ldots, 5\}\,.
\end{aligned}
\right.
\end{equation}
For the radius $2a\rho_{\tilde{z}}^\ast \leq 10\rho_{\tilde{z}}^\ast <R_\lambda/2$ and, by Lemma~\ref{t.inclusion cylinder}, $10Q^\lambda_{\rho_{\tilde{z}}^\ast}(\tilde{z}) \subset Q_{r_2}$ follows. In particular, we have that $2aQ^\lambda_{\rho_{\tilde{z}}^\ast}(\tilde{z}) \subset Q_{r_2}$ for $a \in \{1,\ldots, 5\}$. The Vitali covering theorem~\cite[Theorem 1.5.1]{EG92} yields a countable collection of pairwise disjoint cylinders $\cF:=\left\{Q_i^\lambda:=Q^\lambda_{2\rho_{z_i}^\ast}(z_i)\right\}_{i \in \N}$  in $\left\{Q^\lambda_{2\rho_{z}^\ast}(\widetilde{z})\right\}_{\widetilde{z} \in \mathbf{\Xi}(r_1, \lambda)}$ such that
\[
\mathbf{\Xi}(r_1, \lambda) \subset \bigcup_{i \in \N} 5Q_i^\lambda\,.
\]
Hence, using the stopping property~\eqref{e.pfgradhigher3}, 
\[
\int_{5Q_i^\lambda}|\nabla u|^p\d{x}\d{t} \leq \lambda^p|5Q_i^\lambda|=5^{d+2}\lambda^p|Q_i^\lambda|\,,
\]
and so we obtain
\begin{equation}\label{e.pfgradhigher4}
\int_{\mathbf{\Xi}(r_1, \lambda)}|\nabla u|^p\d{x}\d{t} \leq 5^{d+2}\lambda^p\sum_{i=1}^\infty|Q_i^\lambda|\,.
\end{equation}
Moreover, by this stopping property, every selected cylinder $Q_i^\lambda=Q_{2\rho_{z_i}^\ast}^\lambda(z_i)$ satisfies the intrinsic geometry~\eqref{intrinsic 1} with $K=1$. Indeed, by~\eqref{e.pfgradhigher3} we find that
\[
\dashint_{Q_i^\lambda}\bigl(|\nabla u|^p+|F|^p \bigr)\d{x}\d{t} <\lambda^p=\dashint_{\frac{1}{2}Q_i^\lambda}\bigl(|\nabla u|^p+|F|^p \bigr)\d{x}\d{t} \,,
\]
which is exactly~\eqref{intrinsic 1} with $K=1$.

\subsection{Estimate on super-level sets}\label{Sect.5.3}
Denote, for short, 
\[
\mathbf{T}^p_\rho:=\rho^{\theta}\dashint_{I_{\rho}^\lambda}\int_{\R^d \setminus B_{\rho/2}}\frac{\big|u(x,t)-(u)_{Q_{\rho}^\lambda}\big|^{p}}{|y|^{d+sp}}\d{y}\d{t}\,.
\]
Define a ``bad Tail family''
\[
\cB_{\tail}:=\left\{Q_i^\lambda=Q^\lambda_{2\rho_{z_i}^\ast}(z_i) : (2\rho_{z_i}^\ast)^\theta \dashint_{Q_i} \cE[u](x,t)\d{x}\d{t} >\eta \lambda^p \right\}\,,
\]
and a ``bad inhomogeneity family''
\[
\cB_{\mathrm{inhom}}:=\left\{Q_i^\lambda=Q^\lambda_{2\rho_{z_i}^\ast}(z_i) :  \dashint_{Q_i} \bigl(|F|^p+1\bigr)\d{x}\d{t} >\eta \lambda^p \right\}\,.
\]
In the remaining set, we define a ``good family'' as
\[
\cG:=\cF \setminus \left( \cB_{\tail} \cup \cB_{\mathrm{inhom}} \right)\,.
\]
We also define indices in the following: for $j \in \N$, 
\[
\left\{
\begin{aligned}
& j \in \mathrm{Ind} (\cB_{\tail}) &\iff & \quad Q_j^\lambda  \in \cB_{\tail}\,, \\
& j \in \mathrm{Ind} (\cB_{\mathrm{inhom}}) &\iff & \quad Q_j^\lambda \in \cB_{\mathrm{inhom}}\,, \quad \mbox{and} 
\\ 
&  j \in \mathrm{Ind} (\cG) &\iff &\quad Q_j^\lambda \in \cG\,.
\end{aligned}
\right.
\]
\emph{Step 1: Super-level set estimate for $Q_i^\lambda \in \cG$}. By construction, for $Q_i^\lambda \in \cG$ we have
\[
\dashint_{\frac{1}{2}Q_i^\lambda}\bigl(|\nabla u|^p+|F|^p\bigr)\d{x}\d{t}=\lambda^p\,.
\]
On the other hand, 
\[
\dashint_{\frac{1}{2}Q_i^\lambda}|F|^p\d{x}\d{t} \leq 2^{d+2} \dashint_{Q_i^\lambda}|F|^p\d{x}\d{t} \leq 2^{d+2}\eta \lambda^p\,.
\]
We temporarily assume that $\eta_0 <\nicefrac{1}{2^{d+4}}$ to obtain that
\begin{equation}\label{e.pfgradhigher5}
\dashint_{\frac{1}{2}Q_i^\lambda}|\nabla u|^p\d{x}\d{t} = \lambda^p-\dashint_{\frac{1}{2}Q_i^\lambda}|F|^p\d{x}\d{t} >\lambda^p-\frac{1}{4}\lambda^p=\frac{3}{4}\lambda^p\,.
\end{equation}
Now, select $R_0$ such that $C_{\mathsf{Tail}} R_0^{\theta p} \leq \eta$. Appealing to Lemma~\ref{t.p-tail-controlled inequality} with $Q_i^\lambda \in \cG$, we get, for a constant $C_{\mathsf{Tail}}(\data)<\infty$,
\begin{equation*}
\mathbf{T}_{2\rho_{z_i}^\ast}^p \leq C_{\TAIL} (2\rho_{z_i}^\ast)^{\theta p}\lambda^p+C_{\TAIL}\left(1+\bm{\mathfrak{S}}_{p, \lambda}(2\rho_{z_i}^\ast)\right)\underbrace{\left((2\rho_{z_i}^\ast)^{\theta}\dashint_{Q_i^\lambda}\cE[u]\d{x}\d{t}\right)}_{\leq \eta \lambda^p}\,,
\end{equation*}
We then break into two cases $\frac{2d}{d+2}<p<2$ and $p \geq 2$. In the former case $\frac{2d}{d+2}<p<2$, by definition, $\bm{\mathfrak{S}}_{p, \lambda}(2\rho_{z_i}^\ast)=1$. In the latter case, since $Q_i^\lambda \in \cG$, we crudely estimate that
\begin{equation*}
\bm{\mathfrak{S}}_{p, \lambda}(2\rho_{z_i}^\ast) \leq \eta^{p-2} \leq 1\,.
\end{equation*}
Thus, using this and
\[
1 \leq \eta \lambda_0^p \leq  \eta \lambda^p \quad \mbox{and} \quad 
\rho_{z_i}^\ast \leq \frac{R_\lambda}{50} \leq \frac{R}{10^{10d}} <R_0\,,
\]
we get
\begin{equation}\label{e.pfgradhigher6}
\mathbf{T}_{2\rho_{z_i}^\ast}^p \leq \bigl(2^{\frac{p^2}{p-1}}+1\bigr)C_{\TAIL}\eta \lambda^p\,.
\end{equation}
Using the previous two displays~\eqref{e.pfgradhigher5} and~\eqref{e.pfgradhigher6}, the reverse H\"{o}lder inequality (Lemma~\ref{t.Reverse Holder}) implies that
\begin{align*}
\frac{3}{4}\lambda^p &\leq \dashint_{\frac{1}{2}Q_i^\lambda}|\nabla u|^p\d{x}\d{t} \\
&\leq C_{\RH}\left[\left(\dashint_{Q_i^\lambda}|\nabla u|^{\kappa p}\d{x}\d{t}\right)^{\nicefrac{1}{\kappa}}+\dashint_{Q_i^\lambda}\bigl(|F|^p+1\bigr)\d{x}\d{t} \right]+C_{\RH}\mathbf{T}_{2\rho_{z_i}^\ast}^p \\
&\quad \quad \quad +C_{\RH}\bm{\mathfrak{U}}_{p, \lambda}(2\rho_{z_i}^\ast) \underbrace{\left((2\rho_{z_i}^\ast)^\theta \dashint_{Q_i^\lambda}\cE[u]\d{x}\d{t}\right)}_{\leq \eta \lambda^p}\\
&\!\!\!\!\!\!\!\!\stackrel{\eqref{e.pfgradhigher5}, \eqref{e.pfgradhigher6}}{\leq} C_{\RH}\left(\dashint_{Q_i^\lambda}|\nabla u|^{\kappa p}\d{x}\d{t}\right)^{\nicefrac{1}{\kappa}} + C_{\RH}\bigl(1+4^{\frac{p^2}{p-1}}C_{\TAIL} \bigr)\eta \lambda^p + C_{\RH}\bm{\mathfrak{U}}_{p, \lambda}(2\rho_{z_i}^\ast) \eta \lambda^p\,.
\end{align*}
Again, we break into two cases $\frac{2d}{d+2}<p<2$ and $p \geq 2$. In the former case $\frac{2d}{d+2}<p<2$, by definition, $\bm{\mathfrak{U}}_{p, \lambda}(2\rho_{z_i}^\ast)=1$. In the latter case, since $Q_i^\lambda \in \cG$, we crudely estimate that
\begin{align*}
\bm{\mathfrak{U}}_{p, \lambda}(2\rho_{z_i}^\ast)&=\frac{1}{\lambda^{p \sigma_0}}\left((2\rho_{z_i}^\ast)^\theta \dashint_{Q_{2\rho_{z_i}^\ast}^\lambda}\cE[u]\d{x}\d{t}\right)^{\sigma_0}  + \frac{1}{\lambda^{p(p-2)}}\left((2\rho_{z_i}^\ast)^\theta \dashint_{Q_{2\rho_{z_i}^\ast}^\lambda}\cE[u]\d{x}\d{t}\right)^{p-2}\\
&\leq \eta^{\sigma_0}+\eta^{p-2} \leq 2\,.
\end{align*}
Thus, we have that
\[
\frac{3}{4}\lambda^p \leq C_{\RH}\left(\dashint_{Q_i^\lambda}|\nabla u|^{\kappa p}\d{x}\d{t}\right)^{\nicefrac{1}{\kappa}} + C_{\RH}\bigl(3+4^{\frac{p^2}{p-1}}C_{\TAIL} \bigr)\eta \lambda^p
\]
Assume that $\eta_0 \in (0,1)$ satisfies
\[
C_{\RH}\bigl(3+4^{\frac{p^2}{p-1}}C_{\TAIL} \bigr)\eta_0 \leq \frac{1}{4}\,,
\]
Therefore, rearranging gives that
\begin{equation}\label{e.pfgradhigher7}
\frac{1}{2}\lambda^p \leq C_{\mathsf{RH}}\left(\dashint_{Q_i^\lambda}|\nabla u|^{\kappa p}\d{x}\d{t}\right)^{\nicefrac{1}{\kappa}}\,.
\end{equation}
Thus, for sufficiently small $C_\star>0$ to be determined later we bound
\begin{align*}
\dashint_{Q_i^\lambda}|\nabla u|^{\kappa p}\d{x}\d{t}&=\frac{1}{|Q_i^\lambda|}\left(\int_{Q_i^\lambda \cap \{|\nabla u|<C_\star \lambda\}} |\nabla u|^{\kappa p}\d{x}\d{t}+\int_{Q_i^\lambda \cap \{|\nabla u| \geq C_\star \lambda\}} |\nabla u|^{\kappa p}\d{x}\d{t}\right)\\[2mm]
&<(C_\star \lambda)^{\kappa p}+\frac{1}{|Q_i^\lambda|}\int_{Q_i^\lambda \cap \{|\nabla u| \geq C_\star \lambda\}} |\nabla u|^{\kappa p}\d{x}\d{t} \\[2mm]
&\!\!\stackrel{\eqref{e.pfgradhigher7}}{\leq} (2C_{\mathsf{RH}})^\kappa \left(\dashint_{Q_i^\lambda}|\nabla u|^{\kappa p}\d{x}\d{t}\right)C_\star^{\kappa p}+\frac{1}{|Q_i^\lambda|}\int_{Q_i^\lambda \cap \{|\nabla u| \geq C_\star \lambda\}} |\nabla u|^{\kappa p}\d{x}\d{t}\,.
\end{align*}
Select $C_\star>0$ so that
\[
(2C_{\mathsf{RH}})^\kappa C_\star^{\kappa p}=\frac{1}{2} \quad \iff \quad C_\star=\left(\frac{1}{2^{\kappa+1}C_{\mathsf{RH}}^\kappa} \right)^{\nicefrac{1}{(\kappa p)}}\,.
\]
This leads to
\[
\dashint_{Q_i^\lambda}|\nabla u|^{\kappa p}\d{x}\d{t} \leq \frac{2}{|Q_i^\lambda|}\int_{Q_i^\lambda \cap \{|\nabla u| \geq C_\star \lambda\}} |\nabla u|^{\kappa p}\d{x}\d{t}\,,
\]
and thus combining this with~\eqref{e.pfgradhigher7} yields that, for $Q_i^\lambda \in \cG$, 
\begin{equation}\label{e.pfgradhigher8}
|Q_i^\lambda| \leq \frac{2^{\kappa +1}C_{\mathsf{RH}}^\kappa}{\lambda^{\kappa p} }\int_{Q_i^\lambda \cap \{|\nabla u| \geq C_\star \lambda\}} |\nabla u|^{\kappa p}\d{x}\d{t} \quad \mbox{with} \quad C_\star=\left(\frac{1}{2^{\kappa+1}C_{\mathsf{RH}}^\kappa} \right)^{\nicefrac{1}{(\kappa p)}}\,.
\end{equation}
Using this and the disjointness, we deduce that
\begin{align}\label{e.pfgradhigherA}
\lambda^p \sum_{i \in \mathrm{Ind}(\cG)} |Q_i^\lambda| &\leq 2^{\kappa +1}C_{\mathsf{RH}}^\kappa\lambda^{p-\kappa p} \sum_{i \in \mathrm{Ind}(\cG)} \int_{Q_i^\lambda \cap \{|\nabla u| \geq C_\star \lambda\}} |\nabla u|^{\kappa p}\d{x}\d{t} \notag\\
&\leq 2^{\kappa +1}C_{\mathsf{RH}}^\kappa\lambda^{p-\kappa p}\int_{Q_{r_2} \cap \{|\nabla u| \geq C_\star \lambda\}} |\nabla u|^{\kappa p}\d{x}\d{t}\,.
\end{align}
\smallskip

\emph{Step 2: Super-level set estimate for $Q_i^\lambda \in \cB_{\mathrm{inhom}}$}.  For this, we deploy another tool. For a fixed $\lambda$ above, let us introduce a \emph{maximal operator} $\cM_{R_\lambda}[\varphi](z)$ defined for each $\varphi \in L^1_{\mathrm{loc}}(\R^{d+1})$ by
\[
\cM_{R_\lambda}[\varphi](z):=\sup_{0<r\leq R_\lambda}\dashint_{Q_r^\lambda(z)}|\varphi(y,\tau)|\d{y}\d{\tau}\,.
\]
Furthermore, we define the fractional maximal operator $\cM_{\theta, R_\lambda}$ by
\[
\cM_{\theta, R_\lambda}[\varphi](z):=\sup_{0<r\leq R_\lambda}r^\theta \dashint_{Q_r^\lambda(z)} |\varphi(y,\tau)|\d{y}\d{\tau}\,.
\]
%
As explained in~\cite[Theorem 1.1]{Ste70} or~\cite[Theorem 1.15]{KLV}, for every $\varphi \in L^1(\R^{d+1})$, we have the \emph{weak-type $L^1$ estimate}
\begin{equation}\label{e.pfgradhigher9}
\left| \bigl\{\cM_{R_\lambda}[\varphi]>a \bigr\}\right| \leq \frac{C(d)}{a}\|\varphi\|_{L^1(\R^{d+1})}\quad \mbox{for}\,\,a>0\,.
\end{equation}
For $Q_i^\lambda \in \cB_{\mathrm{inhom}}$, appealing to H\"{o}lder's inequality yields
\begin{equation}\label{e.pfgradhigher10}
(\eta \lambda^p)^{1+\sigma}<\left(\dashint_{Q_i^\lambda} \bigl(|F|^p+1 \bigr)\d{x}\d{t} \right)^{1+\sigma}\leq \dashint_{Q_i^\lambda} \bigl(|F|^p+1 \bigr)^{1+\sigma}\d{x}\d{t}\,.
\end{equation}
Since $Q_i^\lambda \subset Q^\lambda_{4\rho_{z_i}^\ast}(z)$ for any $z \in Q_i^\lambda$, $4\rho_{z_i}^\ast <R_\lambda$ and also $Q_i^\lambda \subset Q_{r_2}$, in view of Lemma~\ref{t.inclusion cylinder}.  Using this and~\eqref{e.pfgradhigher10}, we bound, for every $i \in\mathrm{Ind}(\cB_{\mathrm{inhom}})$,
\begin{align*}
\cM_{R_\lambda}\left[\bigl(|F|^p+1\bigr)^{1+\sigma }\indc_{Q_{r_2}}\right](z) &\geq \dashint_{Q_{4\rho_{z_i}^\ast}^\lambda(z)} \bigl(|F|^p+1\bigr)^{1+\sigma}\indc_{Q_{r_2}}\d{y}\d{\tau} \\
&\geq 2^{-d-2}\dashint_{Q_i^\lambda}\bigl(|F|^p+1\bigr)^{1+\sigma}\d{y}\d{\tau} \\
&\!\!\!\stackrel{\eqref{e.pfgradhigher10}}{\geq} 2^{-d-2}(\eta \lambda^p)^{1+\sigma}\,,
\end{align*}
and hence
\[
\bigcup_{i \in\mathrm{Ind}(\cB_{\mathrm{inhom}})} Q_i^\lambda \subset \left\{ \cM_{R_\lambda}\left[\bigl(|F|^p+1\bigr)^{1+\sigma}\indc_{Q_{r_2}}\right] \geq 2^{-d-2}(\eta \lambda^p)^{1+\sigma}\right\}\,.
\]
Using the disjointness, we have
\begin{align*}
\sum_{i \in \mathrm{Ind}(\cB_{\mathrm{inhom}})} |Q_i^\lambda| &=\left|\bigcup_{i \in\mathrm{Ind}(\cB_{\mathrm{inhom}})} Q_i^\lambda \right| \\[2mm]
&\leq \Bigl|\left\{ \cM_{R_\lambda}\left[\bigl(|F|^p+1\bigr)^{1+\sigma}\indc_{Q_{r_2}}\right] \geq 2^{-d-2}(\eta \lambda^p)^{1+\sigma}\right\}\Bigr|\,.
\end{align*}
Thus, appealing to~\eqref{e.pfgradhigher9} with $\varphi=\bigl(|F|^p+1\bigr)^{1+\sigma} \indc_{Q_{r_2}}$ ensures that
\begin{equation}\label{e.pfgradhigherB}
\lambda^p \sum_{i \in \mathrm{Ind}(\cB_{\mathrm{inhom}})} |Q_i^\lambda| \leq C\eta^{-1-\sigma}\lambda^{-p\sigma}\int_{Q_{r_2}}\bigl(|F|^p+1\bigr)^{1+\sigma}\d{x}\d{t}\,.
\end{equation}
\smallskip

\emph{Step 3: Super-level set estimate for $Q_i^\lambda \in \cB_{\tail}$}. Let $Q_i^\lambda \in \cB_{\tail}$ and denote $\varphi:=\cE[u]\indc_{Q_{r_2}}$. Notice that $Q_i^\lambda  \subset Q^\lambda_{4\rho_{z_i}^\ast}(z)$ for any $z \in Q_i^\lambda$ and $4\rho_{z_i}^\ast <R_\lambda$. Also, since $Q_i^\lambda \subset Q_{r_2}$, $\varphi=\cE[u]$ on $Q_i^\lambda$. We then split the argument into two cases $\theta<d+2$ or $ \theta \geq d+2$. In the former case, using these observations, we have that, for every $i \in \mathrm{Ind}(\cB_{\mathrm{Tail}})$, 
\begin{align*}
\cM_{\theta, R_\lambda}\bigl[\varphi \bigr](z) &\geq (4\rho_{z_i}^\ast)^{\theta}\dashint_{Q^\lambda_{4\rho_{z_i}^\ast}(z)}\varphi \d{x}\d{t}\\
&\geq 2^{\theta-d-2}(2\rho_{z_i}^\ast)^{\theta}\dashint_{Q_i^\lambda}\cE[u]\d{x}\d{t} >2^{\theta-d-2} \eta \lambda^p\,,
\end{align*}
and therefore 
\[
\bigcup_{i \in \mathrm{Ind}(\cB_{\mathrm{Tail}})} Q_i^\lambda \subset \Bigl\{ \cM_{\theta, R_\lambda}\bigl[\cE[u]\indc_{Q_{r_2}}\bigr] > 2^{\theta-d-2} \eta \lambda^p\Bigr\}\,.
\]
Using the disjointness, we deduce that
\begin{equation}\label{e.pfgradhigherCpre}
\sum_{i \in \mathrm{Ind}(\cB_{\mathrm{Tail}})} |Q_i^\lambda|= \left|\bigcup_{i \in \mathrm{Ind}(\cB_{\mathrm{Tail}})} Q_i^\lambda\right| \leq \Bigl|\Bigl\{ \cM_{\theta, R_\lambda}\bigl[\cE[u]\indc_{Q_{r_2}}\bigr] > 2^{\theta-d-2} \eta \lambda^p\Bigr\}\Bigr|\,.
\end{equation}
We will deduce the following fractional weak type estimate.
\begin{lemma}\label{t.weak L1-estimate}
Let $\theta=\frac{(1-s)p}{p-1}$ satisfy $\theta <d+2$ and set $q:=\frac{d+2}{d+2-\theta}>1$. There exists a constant $C(d,\theta)<\infty$ such that for every $a>0$ and $\varphi \in L^1(\R^{d+1})$
\begin{equation*}
\Bigl| \bigl\{\cM_{\theta, R_\lambda}[\varphi]>a \bigr\}\Bigr| \leq C\frac{\lambda^{(p-2)(q-1)}}{a^q}\|\varphi\|_{L^1(\R^{d+1})}^q\,.
\end{equation*}
\end{lemma}
\begin{proof}[Proof of Lemma~\ref{t.weak L1-estimate}]
Denote $E_a:=\bigl\{\cM_{\theta, R_\lambda}[\varphi]>a \bigr\}$ for short. For $z \in E_a$, we select $r_z \leq R_\lambda$ so that
\[
a<r_z^\theta \dashint_{Q_{r_z}^\lambda(z)}|\varphi|\d{x}\d{t}\,.
\]
Applying Vitali's covering theorem~\cite[Theorem 1.5.1]{EG92}, there is a countable, pairwise disjoint subfamily $\left\{Q_j^{\lambda, a}:=Q^\lambda_{r_{z_j}}(z_j)\right\}_{j \in \N} \subset \left\{Q^\lambda_{r_{z}}(z)\right\}_{z \in E_a}$ so that
\begin{equation}\label{e.pfgradhigher11}
E_a\subset \bigcup_{j \in \N} 5Q_j^{\lambda, a}\,.
\end{equation}
Since
\[
a<r_{z_j}^\theta \dashint_{Q_j^{\lambda, a}}|\varphi|\d{x}\d{t} \quad \iff \quad r_{z_j}^{d+2-\theta} \leq C(d)\frac{\lambda^{p-2}}{a}\int_{Q_j^{\lambda, a}}|\varphi|\d{x}\d{t}\,,
\]
we have
\begin{align*}
\bigl|Q_j^{\lambda, a}\bigr|=2\omega_d\lambda^{2-p}(r_{z_j}^{d+2-\theta})^q \leq C(d)\frac{\lambda^{(p-2)(q-1)}}{a^q}\left(\int_{Q_j^{\lambda, a}}|\varphi|\d{x}\d{t}\right)^q\,.
\end{align*}
Using~\eqref{e.pfgradhigher11} and the disjointness of $\{Q_j^{\lambda, a}\}_{j \in \N}$, we find that
\begin{align*}
|E_a| &\leq 5^{d+2}\sum_{j=1}^\infty \bigl|Q_j^{\lambda, a}\bigr| \\
&\!\!\!\!\stackrel{\eqref{e.pfgradhigher11}}{\leq}\frac{C}{a^q}\lambda^{(p-2)(q-1)}\sum_{j=1}^\infty \left(\int_{Q_j^{\lambda, a}}|\varphi|\d{x}\d{t}\right)^q\\
&\leq \frac{C}{a^q}\lambda^{(p-2)(q-1)}\left(\sum_{j=1}^\infty \int_{Q_j^{\lambda, a}}|\varphi|\d{x}\d{t}\right)^q \\
&\leq C\frac{\lambda^{(p-2)(q-1)}}{a^q}\|\varphi\|_{L^1(\R^{d+1})}^q\,,
\end{align*}
finishing the proof of Lemma~\ref{t.weak L1-estimate}.
\end{proof}

Thus, appealing to Lemma~\ref{t.weak L1-estimate} with $\varphi=\cE[u]\indc_{Q_{r_2}}$ and using~\eqref{e.pfgradhigherCpre}, we deduce that %
\begin{equation}\label{e.pfgradhigherC}
\lambda^p \sum_{i \in \mathrm{Ind} (\cB_{\tail})}|Q_i^\lambda| 
\leq C\eta^{-q} \lambda^{-\frac{2\theta}{d+2-\theta}} \left(\int_{Q_{r_2}} \cE[u]\d{x}\d{t}\right)^q
\end{equation}
for a constant $C(d,\theta)<\infty$, where we computed that $(p-2)(q-1)-pq=-p-\frac{2\theta}{d+2-\theta}$.

In the latter case $\theta \geq d+2$, we find the following lemma.
\begin{lemma}\label{t.empty}
Let $\theta \geq d+2$. For $\lambda \geq \lambda_0$ with $\lambda_0$ being as in~\eqref{e.pfgradhigher1}, we have $\cB_{\tail}=\varnothing$.
\end{lemma}

\begin{proof}[Proof of Lemma~\ref{t.empty}]
We start with the observation that
\[
\theta=\frac{(1-s)p}{p-1} \geq d+2 \quad \mbox{and} \quad p >\frac{2d}{d+2} \quad \Longrightarrow \quad  d=2,3\,,
\]
and therefore, we find that $\omega_2=\pi>1$ and $\omega_3=\nicefrac{4\pi}{3}>1$, obtaining
\[
\left(\omega_d (25\cdot 10^{10d})^{\theta-d-2}\right)^{-1} \leq 1\,.
\]
Observe further that, by $d \geq 2$, $\theta \geq d+2 \geq 4$ follows, whereas assuming $p \geq 2$, we find that
\[
\theta=\frac{(1-s)p}{p-1} \leq \frac{p}{p-1} \leq 2\,;
\]
a contradiction. Thus, it suffices to check the assertion for the case $\frac{2d}{d+2}<p<2$. A straightforward calculation implies that, in view of the definition of $R_\lambda$, 
\begin{align}\label{e.pfgradhigher12}
(2\rho_{z_i}^\ast)^{\theta} \dashint_{Q_i^\lambda}\cE[u]\d{x}\d{t}&=\frac{\lambda^{p-2}}{\omega_d}(2\rho_{z_i}^\ast)^{\theta-d-2}\int_{Q_i^\lambda}\cE[u]\d{x}\d{t} \notag\\
&\leq \frac{\lambda^{p-2}\bigl[(r_2-r_1)(1 \wedge \lambda^{\frac{p-2}{2}})\bigr]^{\theta-d-2}}{\omega_d (25\cdot 10^{10d})^{\theta-d-2}}\int_{Q_{r_2}}\cE[u]\d{x}\d{t}\notag\\
&\leq \lambda^p (r_2-r_1)^{\theta-d-2}\lambda^{-2+\frac{(p-2)(\theta-d-2)}{2}}\int_{Q_{r_2}}\cE[u]\d{x}\d{t} \notag\\
&\leq \lambda^p (r_2-r_1)^{\theta-d-2}\lambda_0^{-2+\frac{(p-2)(\theta-d-2)}{2}}\int_{Q_{r_2}}\cE[u]\d{x}\d{t}\,.
\end{align}
%
%
%
%
%
By~\eqref{e.pfgradhigher1} we find that
\begin{equation*}
\lambda_0^{-2+\frac{(p-2)(\theta-d-2)}{2}}(r_2-r_1)^{\theta-d-2}\int_{Q_{r_2}}\cE[u]\d{x}\d{t} \leq \eta\,,
\end{equation*}
and thus
\[
(2\rho_{z_i}^\ast)^{\theta} \dashint_{Q_i^\lambda}\cE[u]\d{x}\d{t} \leq \eta\lambda^p\,.
\]
Combining the above arguments implies the desired conclusion.
\end{proof}

\smallskip
\emph{Step 4: The conclusion.}  We present the final piece of the argument of the main result. When considering $\theta \geq d+2$, Lemma~\ref{t.empty} yields $\cB_{\tail}=\varnothing$. The argument with the $\cB_{\tail}$-term omitted and the contribution of $\lambda_0$ treated by Young's inequality with conjugate exponents $\left(1+\nicefrac{\eps p}{\gamma_p},\,\nicefrac{(\eps p+\gamma_p)}{\eps p}\right)$ immediately yields~\eqref{e.gradhigher2}. Thus, we present the details only for the case $\theta <d+2$. 

We henceforth take a parameter $\eta_0 \in (0,1)$, which must satisfy
\begin{equation}\label{e.pfgradhigher13}
\eta_0 =\min \left\{\frac{1}{2^{d+5}},\,\frac{1}{50C_{\RH}(3+4^{\frac{p^2}{p-1}}C_{\TAIL})} \right\}
\end{equation}
for our arguments to go through. Combining~\eqref{e.pfgradhigherA},~\eqref{e.pfgradhigherB} and~\eqref{e.pfgradhigherC} with~\eqref{e.pfgradhigher4}, we obtain that
\begin{align*}
\int_{\mathbf{\Xi}(r_1, \lambda)}&|\nabla u|^p\d{x}\d{t} \\
&\leq 5^{d+2} \left(\lambda^p \sum_{i \in \mathrm{Ind}(\cG)} |Q_i^\lambda| +\lambda^p \sum_{i \in \mathrm{Ind}(\cB_{\tail})} |Q_i^\lambda| +\lambda^p \sum_{i \in \mathrm{Ind}(\cB_{\mathrm{inhom}})} |Q_i^\lambda| \right) \\
&\leq C\lambda^{p-\kappa p}\int_{Q_{r_2} \cap \{|\nabla u| \geq C_\star \lambda\}} |\nabla u|^{\kappa p}\d{x}\d{t}+C\eta^{-1-\sigma}\lambda^{-p\sigma}\int_{Q_{r_2}}\bigl(|F|^p+1\bigr)^{1+\sigma}\d{x}\d{t}\\
&\quad \quad \quad+C\eta^{-q} \lambda^{-\frac{2\theta}{d+2-\theta}} \left(\int_{Q_{r_2}} \cE[u]\d{x}\d{t}\right)^q\,,
\end{align*}
where $C(\data)<\infty$ and $C_\star=\left(2^{\kappa+1}C_{\mathsf{RH}}^\kappa \right)^{\nicefrac{-1}{(\kappa p)}}$, since both $C_{\RH}$ and $\kappa$ depend on $\data$ as well. Thus, after taking $\eta=\eta_0$, rearranging this yields
\begin{align}\label{e.pfgradhigher14}
\int_{\mathbf{\Xi}(r_1, \lambda)}&|\nabla u|^p\d{x}\d{t} \notag \\
&\leq C\lambda^{p-\kappa p}\int_{Q_{r_2} \cap \{|\nabla u| \geq C_\star \lambda\}} |\nabla u|^{\kappa p}\d{x}\d{t}+C\lambda^{-p\sigma}\int_{Q_{r_2}}\bigl(|F|^p+1\bigr)^{1+\sigma}\d{x}\d{t} \notag\\
&\quad \quad \quad+C\lambda^{-\frac{2\theta}{d+2-\theta}} \left(\int_{Q_{r_2}} \cE[u]\d{x}\d{t}\right)^q\,.
\end{align}

\subsection{Proof of the gradient estimate}\label{Sect.5.4}

For $k>\lambda_0$ and $\eps>0$, we set
\[
|\nabla u|_k:=|\nabla u| \wedge k\,,
\]
and
\[
\bm{\Phi}_k(r):=\int_{Q_r} |\nabla u|^p |\nabla u|_k^{\eps p}\d{x}\d{t}\,.
\]
By $|\nabla u| \in L^p_{\mathrm{loc}}(\Omega_T)$, we have that $\bm{\Phi}_k(r)<\infty$. The Cavalieri principle says that
\begin{align*}
|\nabla u|_k^{\eps p}&=\eps p\int_0^{\lambda_0}\lambda^{\eps p-1} \indc_{\{|\nabla u| >\lambda\}}\d{\lambda}+\eps p\int_{\lambda_0}^k \lambda^{\eps p-1} \indc_{\{|\nabla u| >\lambda\}}\d{\lambda}\\
&\leq \lambda_0^{\eps p}+\eps p\int_{\lambda_0}^k \lambda^{\eps p-1} \indc_{\{|\nabla u| >\lambda\}}\d{\lambda}\,.
\end{align*}
Multiplying this last inequality by $|\nabla u|^p$ and integrating over $Q_{r_1}$, in view of Fubini's theorem,
\begin{align*}
\int_{Q_{r_1}}&|\nabla u|^p|\nabla u|_k^{\eps p}\d{x}\d{t} \\
&\leq \lambda_0^{\eps p}\int_{Q_{r_1}}|\nabla u|^p\d{x}\d{t}+\eps p\int_{Q_{r_1}} |\nabla u|^p\left(\int_{\lambda_0}^k \lambda^{\eps p-1} \indc_{\{|\nabla u| >\lambda\}}\d{\lambda}\right)\d{x}\d{t} \\
&=\lambda_0^{\eps p}\int_{Q_{r_1}}|\nabla u|^p\d{x}\d{t}+\eps p\int_{\lambda_0}^k \lambda^{\eps p-1}\left(\int_{\mathbf{\Xi}(r_1, \lambda)} |\nabla u|^p \d{x}\d{t}\right)\d{\lambda}\,.
\end{align*}
The combination of this and~\eqref{e.pfgradhigher14} yields
\begin{align}\label{e.pfgradhigher15}
\bm{\Phi}_k(r_1) &\leq \lambda_0^{\eps p}\int_{Q_{r_1}}|\nabla u|^p\d{x}\d{t}+\eps p\int_{\lambda_0}^k \lambda^{\eps p-1}\left(\int_{\mathbf{\Xi}(r_1, \lambda)} |\nabla u|^p \d{x}\d{t}\right)\d{\lambda} \notag\\
&\!\!\!\!\stackrel{\eqref{e.pfgradhigher14}}{\leq} \lambda_0^{\eps p}\int_{Q_{r_1}}|\nabla u|^p\d{x}\d{t}+\mathbf{I}+\mathbf{II}+\mathbf{III}\,,
\end{align}
where
\[
\left\{
\begin{aligned}
& \mathbf{I}:=C\eps p\int_{\lambda_0}^k \lambda^{\eps p-1+p-\kappa p} \left(\int_{\mathbf{\Xi}(r_2, C_\star \lambda)} |\nabla u|^{\kappa p} \d{x}\d{t}\right)\d{\lambda}\,, \\[2mm]
&\mathbf{II}:=C\eps p\int_{\lambda_0}^k \lambda^{\eps p-1-p\sigma} \left(\int_{Q_{r_2}}\bigl(|F|^p+1\bigr)^{1+\sigma}\d{x}\d{t}\right)\d{\lambda}\,,  \quad \mbox{and} 
\\[2mm]
&\mathbf{III}:=C\eps p\int_{\lambda_0}^k \lambda^{\eps p-1-\frac{2\theta}{d+2-\theta}} \left(\int_{Q_{r_2}}\cE[u]\d{x}\d{t}\right)^q\d{\lambda}\,.
\end{aligned}
\right.
\]
It is straightforward to check that, using Fubini's theorem,
\begin{align*}
\mathbf{I}&=C\eps \int_{\lambda_0}^k \int_{Q_{r_2}}\lambda^{\eps p-1+p-\kappa p} |\nabla u|^{\kappa p} \indc_{\{|\nabla u| >C_\star \lambda\}} \d{x}\d{t}\d{\lambda}\\
&=C\eps p\int_{Q_{r_2}} |\nabla u|^{\kappa p} \int_{\lambda_0}^k \lambda^{\eps p-1+p-\kappa p}  \indc_{\{|\nabla u| >C_\star \lambda\}} \d{\lambda}\d{x}\d{t} \\
&=C\eps p\int_{Q_{r_2}} |\nabla u|^{\kappa p} \left(\int_{\lambda_0}^{k \wedge |\nabla u|/C_\star} \lambda^{\eps p-1+p-\kappa p} \d{\lambda}\right)\d{x}\d{t} \\
&\leq \frac{C\eps p}{\eps p+p-\kappa p}\int_{Q_{r_2}}|\nabla u|^{\kappa p} \left(k\wedge \frac{|\nabla u|}{C_\star}\right)^{\eps p+p-\kappa p}\d{x}\d{t}\,.
\end{align*}
Since $C_\star \in (0,1)$, a simple analysis gives that
\[
|\nabla u|^{\kappa p} \left(k\wedge \frac{|\nabla u|}{C_\star}\right)^{\eps p+p-\kappa p} \leq \frac{1}{C_\star^{\eps p+p-\kappa p}}|\nabla u|^p|\nabla u|_k^{\eps p}\,,
\]
and thus, 
\[
\mathbf{I} \leq \underbrace{\frac{C\eps p}{(\eps p+p-\kappa p)C_\star^{\eps p+p-\kappa p}}}_{=:\Theta_{C_\star}}\underbrace{\int_{Q_{r_2}}|\nabla u|^p |\nabla u_k^{\eps p}\d{x}\d{t}}_{=\bm{\Phi}_k(r_2)}\,.
\]
$\Theta_{C_\star} \to 0$ as $\eps \searrow  0$; therefore, there exists $\eps_0(\data, C_\star)>0$ such that $\Theta_{C_\star} \leq 1/2$ for every $\eps \in (0,\eps_0)$. Hence
\[
\mathbf{I} \leq \frac{1}{2}\bm{\Phi}_k(r_2)\,.
\]
We now turn our attention to $\mathbf{II}$. We select $\eps>0$ to satisfy $\eps<\sigma$, thereby getting
\[
\mathbf{II}\leq \frac{C\eps p \lambda_0^{p(\eps -\sigma)}}{p(\sigma-\eps)}\int_{Q_{r_2}}\bigl(|F|^p+1\bigr)^{1+\sigma}\d{x}\d{t}\,.
\]
Similarly, if we enforce $\eps< \frac{2\theta}{p(d+2-\theta)}$ then
\[
\mathbf{III}\leq \frac{C\eps p \lambda_0^{\eps p-\frac{2\theta}{d+2-\theta}}}{\frac{2\theta}{d+2-\theta}-\eps p}\left(\int_{Q_{r_2}}\cE[u]\d{x}\d{t}\right)^q\,.
\]
Furthermore, we select $\eps <\min\left\{\sigma,\,\frac{2\theta}{p(d+2-\theta)}\right\}$ to satisfy
\[
\frac{\eps}{\sigma-\eps} \leq 1 \quad \mbox{and} \quad \frac{\eps p}{\frac{2\theta}{d+2-\theta}-\eps p} \leq 1\,,
\]
and hence 
\[
\mathbf{II}+\mathbf{III} \leq C\left(\int_{Q_{r_2}}\cE[u]\d{x}\d{t}\right)^q+C\int_{Q_{r_2}}\bigl(|F|^p+1\bigr)^{1+\sigma}\d{x}\d{t}
\]
because $\lambda_0>1$. Combining the previous estimates with~\eqref{e.pfgradhigher15} yields, upon choosing
$\eps \leq \min \bigl\{\eps_0, \nicefrac{\sigma}{2},\,\nicefrac{\theta}{p(d+2-\theta)}\bigr\}$, that
\begin{align}\label{e.pfgradhigher16}
\bm{\Phi}_k(r_1) \leq \frac{1}{2}\bm{\Phi}_k(r_2) &+ \lambda_0^{\eps p} \int_{Q_{r_1}}|\nabla u|^p\d{x}\d{t} \notag\\
&+C\left(\int_{Q_{r_2}}\cE[u]\d{x}\d{t}\right)^q+C\int_{Q_{r_2}}\bigl(|F|^p+1\bigr)^{1+\sigma}\d{x}\d{t}\,.
\end{align}
Recalling the choice of $\lambda_0$ in~\eqref{e.pfgradhigher1} with $\eta=\eta_0$ being as in~\eqref{e.pfgradhigher13} and the monotonicity of $r \mapsto \mathbf{E}_r$, we find that
\begin{equation*}
\lambda_0^{\eps p} \leq c\left[1+\frac{H(d)^{\nicefrac{\eps p}{\gamma_p}}}{(r_2-r_1)^{\nicefrac{(d+2)\eps p}{\gamma_p}}}\mathbf{E}_{R}^{\nicefrac{\eps p}{\gamma_p}}\right] \\
\end{equation*}
 where $c(\data)<\infty$, since three quantities $C_{\RH}$, $C_{\mathsf{Tail}}$ and $\kappa$ depend on $\data$ as well. Substituting this into~\eqref{e.pfgradhigher16} with rearrangements yields, 
 \begin{align*}
\bm{\Phi}_k(r_1) \leq \frac{1}{2}\bm{\Phi}_k(r_2) &+  \frac{C}{(r_2-r_1)^{\nicefrac{(d+2)\eps p}{\gamma_p}}}\mathbf{E}_{R}^{1+\nicefrac{\eps p}{\gamma_p}} \notag\\[2mm]
&+C\mathbf{E}_R+C\left(\int_{Q_{R}}\cE[u]\d{x}\d{t}\right)^q+C\int_{Q_{R}}\bigl(|F|^p+1\bigr)^{1+\sigma}\d{x}\d{t}\,.
\end{align*}
Appealing to Lemma~\ref{t.iteration}, we find that there exists $C_\ast(\data, \eps)<\infty$ such that
\begin{align*}
\bm{\Phi}_k(R/2) &\leq \frac{C_\ast}{R^{\nicefrac{(d+2)\eps p}{\gamma_p}}} \mathbf{E}_{R}^{1+\nicefrac{\eps p}{\gamma_p}} +C_\ast\mathbf{E}_R\\
& \quad \quad \quad +C_\ast\left(\int_{Q_{R}}\cE[u]\d{x}\d{t}\right)^q+C_\ast\int_{Q_{R}}\bigl(|F|^p+1\bigr)^{1+\sigma}\d{x}\d{t}\,.
\end{align*}
Observe that, by Fatou's lemma, 
\begin{align*}
\liminf_{k \to \infty} \bm{\Phi}_k(R/2) \geq \int_{Q_{R/2}}|\nabla u|^{p(1+\eps)}\d{x}\d{t}\,.
\end{align*}
Finally, this and rearranging imply that
\begin{align*}
\dashint_{Q_{R/2}}&|\nabla u|^{p(1+\eps)}\d{x}\d{t} \\
&\leq C_\ast\left(\dashint_{Q_R}\bigl(|\nabla u|^p+|F|^p\bigr)\d{x}\d{t} \right)^{1+\nicefrac{\eps p}{\gamma_p}} +C_\ast\dashint_{Q_R}\bigl(|\nabla u|^p+|F|^p\bigr)\d{x}\d{t}\\[2mm]
& \quad \quad +C_\ast\left(\dashint_{Q_{R}}\cE[u]\d{x}\d{t}\right)^q \underbrace{R^{(d+2)(q-1)}}_{\leq 1}\,+\,C_\ast\dashint_{Q_{R}}\bigl(|F|^p+1\bigr)^{1+\sigma}\d{x}\d{t}\\
&\leq C_\ast \left(\dashint_{Q_R}\bigl[|\nabla u|^p+(|F|^p+1)^{1+\sigma}\bigr]\d{x}\d{t}\right)^{1+\nicefrac{\eps p}{\gamma_p}}+C_\ast\left(\dashint_{Q_{R}}\cE[u]\d{x}\d{t}\right)^q\,.
\end{align*}
This completes the proof of Theorem~\ref{t.gradhigher}.

\makeatletter
\renewcommand{\thesection}{\Alph{section}.\arabic{subsection}}
\makeatother

\appendix

\makeatletter
\renewcommand{\thefigure}{\Alph{section}.\arabic{figure}}
\@addtoreset{figure}{section}
\makeatother

\subsubsection*{\bf Acknowledgments}
The author wishes to thank Wontae Kim for a useful discussion.


\end{document}